\documentclass[12pt,a4paper]{amsart}
\usepackage{amsfonts}
\usepackage{amsthm}
\usepackage{amsmath}
\usepackage{amssymb}
\usepackage{mathabx}
\usepackage{bm}
\usepackage{amscd}
\usepackage[latin2]{inputenc}
\usepackage{t1enc}
\usepackage[mathscr]{eucal}
\usepackage{indentfirst}
\usepackage{graphicx}
\usepackage{graphics}
\usepackage{pict2e}
\usepackage{epic}
\numberwithin{equation}{section}
\usepackage[margin=2.9cm]{geometry}
\usepackage{epstopdf} 
\usepackage[colorlinks,linkcolor=blue]{hyperref}
\usepackage[capitalise,noabbrev]{cleveref}
\usepackage{todonotes}

\crefformat{equation}{(#2#1#3)}
\crefrangeformat{equation}{(#3#1#4) to~(#5#2#6)}

\usepackage[normalem]{ulem}

\allowdisplaybreaks

\theoremstyle{plain}
\newtheorem{Thm}{Theorem}[section]
\newtheorem*{Thm*}{Theorem}
\newtheorem{Lem}[Thm]{Lemma}

\newtheorem{Prop}[Thm]{Proposition}

\theoremstyle{definition}

\newtheorem{Rem}[Thm]{Remark}
\newtheorem{?}[Thm]{Problem}

\newcommand{\ovl}{\overline}
\newcommand{\p}{\partial}
\newcommand{\R}{\mathbb{R}}
\newcommand{\e}{\varepsilon}

\newcommand{\ovlul}{\ovl{u}_l}
\newcommand{\ovlur}{\ovl{u}_r}

\let\oldtocsection=\tocsection

\let\oldtocsubsection=\tocsubsection

\let\oldtocsubsubsection=\tocsubsubsection

\renewcommand{\tocsection}[2]{\hspace{0em}\oldtocsection{#1}{#2}}
\renewcommand{\tocsubsection}[2]{\hspace{1em}\oldtocsubsection{#1}{#2}}
\renewcommand{\tocsubsubsection}[2]{\hspace{2em}\oldtocsubsubsection{#1}{#2}}

\begin{document}


\title[Shock profiles and rarefaction waves under periodic perturbations]{Asymptotic stability of shock profiles and rarefaction waves under periodic perturbations for 1-d convex scalar viscous conservation laws}

\author[Z. Xin]{Zhouping XIN}
\thanks{This research is partially supported by Zheng Ge Ru Foundation, Hong Kong RGC Earmarked Research Grants, CUHK-14300917, CUHK-14305315, and CUHK-14302917. }
\address[Z. Xin]{The Institute of Mathematical Sciences \&  Department of Mathematics, The Chinese University of Hong Kong, Shatin, N.T., Hong Kong}
\email{zpxin@ims.cuhk.edu.hk}

\author[Q. Yuan]{Qian YUAN}
\address[Q. Yuan]{The Institute of Mathematical Sciences \&  Department of Mathematics, The Chinese University of Hong Kong, Shatin, N.T., Hong Kong}
\email{qyuan103@link.cuhk.edu.hk}

\author[Y. Yuan]{Yuan YUAN}
\thanks{The research of Yuan Yuan is also supported by the Start-up Research Grant of South China Normal University (8S0328). }
\address[Y. Yuan]{South China Research Center for Applied Mathematics and Interdisciplinary Studies, South China Normal University, Guangzhou, Guangdong, China}
\email{yyuan2102@m.scnu.edu.cn}


\maketitle

\begin{abstract} 
	This paper studies the asymptotic stability of shock profiles and rarefaction waves under space-periodic perturbations for one-dimensional convex scalar viscous conservation laws. For the shock profile, we show that the solution approaches the background shock profile with a constant shift in the $ L^\infty(\R) $ norm at exponential rates. 
	The new phenomena contrasting to the case of localized perturbations is that the constant shift cannot be determined by the initial excessive mass in general, which indicates that the periodic oscillations at infinities make contributions to this shift. And the vanishing viscosity limit for the shift is also shown. 
	The key elements of the poof consist of the construction of an ansatz which tends to two periodic solutions as $ x \rightarrow \pm\infty, $ respectively, and the anti-derivative variable argument, and an elaborate use of the maximum principle.
	For the rarefaction wave, we also show the stability in the $ L^\infty(\R) $ norm.
\end{abstract}

\tableofcontents

\section{Introduction and main results}
We consider the Cauchy problem for one-dimensional convex scalar viscous conservation laws
\begin{align}
\p_t u^\nu + \p_x f(u^\nu) & = \nu \p_x^2 u^\nu, \qquad x \in \R, ~t>0, \label{CL} \\
u^\nu(x,0) & = u_0(x), \qquad x\in\R, \label{ic}
\end{align}
where $ u^\nu(x,t) \in \R $ is the unknown, the flux $f$ is smooth and strictly convex, and $ \nu\in (0,1] $ denotes the viscosity. 
This paper is concerned with the asymptotic behavior of the solution $ u^\nu(x,t) $ to \eqref{CL}, \eqref{ic} with $ u_0(x) \in L^\infty(\R) $ satisfying
\begin{equation}\label{ic-2-per}
\begin{cases}
|u_0(x) - \ovlul - w_{0l}(x)| \leq C_0 e^{\beta_0 x}, &  \\
|u_0(x) - \ovlur - w_{0r}(x)| \leq C_0 e^{-\beta_0 x}, & 
\end{cases} \quad x\in\R,
\end{equation}
where $ \ovlul, \ovlur, C_0>0 $ and $ \beta_0>0 $ are constants, and $w_{0l}, w_{0r} \in L^\infty(\R)$ are two arbitrary periodic functions.

It is well known from \cite{Oleinik1957,Kruzkov1970} that for $ \nu>0, $ the equation \eqref{CL} generates a semi-group $ \{S_t^\nu: L^\infty(\R) \rightarrow L^\infty(\R); t\geq 0 \} $ to ensure that, for any initial data $u_0\in L^{\infty}(\R)$, the function $ u^\nu(x,t) := S_t^\nu u_0 $ is the unique bounded solution to \eqref{CL}, \cref{ic}, which is smooth for $ t>0, $ and satisfies the initial condition in the weak sense: for any continuous functions $\varphi(x,t)$ compactly supported in $\R  \times [0, +\infty), $ there holds
\begin{equation*}
\int_{\R}[ \varphi(x,t)u^\nu(x,t)-\varphi(x,0)u_0(x)]\ dx \rightarrow 0 \quad\text{as}~t\rightarrow0.
\end{equation*}
The semi-group $ S_t^\nu $ satisfies the following classical Co-properties:
\begin{itemize}
	\item (Comparison) If $u_0, v_0 \in L^{\infty}(\R)$ and $u_0\leq v_0$ almost everywhere, then $S_t^\nu u_0\leq S_t^\nu v_0$ for any $ x \in \R, ~ t>0. $
	
	\item (Contraction) If $u_0, v_0 \in L^{\infty}(\R)$ and $u_0-v_0\in L^{1}(\R)$, then $S_t^\nu u_0- S_t^\nu v_0\in L^{1}(\R)$ and 
	$\|S_t^\nu u_0 - S_t^\nu v_0\|_{L^1(\R)} $ is non-increasing with respect to $ t. $
	
	\item (Conservation) If $u_0, v_0 \in L^{\infty}(\R)$ and $u_0-v_0\in L^{1}(\R)$, then 
	$$\int_{\R} (S_t^\nu u_0-S_t^\nu v_0) dx=\int_{\R}(u_0-v_0) dx  \qquad \forall ~t \geq 0. $$
\end{itemize}
Moreover, there exists a constant $ E>0, $ depending only on $ f$ and $ \|u_0\|_{L^\infty}, $ such that
\begin{equation}\label{entropy}
\partial_x u^\nu(x,t) \leq \frac{E}{t}  \qquad \forall x\in \R,~ t>0.
\end{equation}

Shocks and rarefaction waves are most important nonlinear solutions to conservation laws. A viscous shock profile $ \phi^\nu(x-st) $ is a classical traveling wave solution to the viscous conservation law \eqref{CL}, solving the problem:
\begin{equation}\label{ode1}
\begin{cases}
& \nu (\phi^\nu)'' = f'(\phi^\nu)(\phi^\nu)'-s(\phi^\nu)',\\
& \lim\limits_{x\rightarrow -\infty}\phi^\nu(x)=\ovlul,~ \lim\limits_{x\rightarrow +\infty}\phi^\nu(x)=\ovlur,
\end{cases}
\end{equation}
where $ \ovlul>\ovlur $ and $ s $ is the shock speed defined by the Rankine-Hugoniot condition: 
\begin{equation*}
s =\dfrac{f(\ovlul)-f(\ovlur)}{\ovlul-\ovlur}.
\end{equation*}
The existence of the shock profile follows from a simple phase plane analysis, and can also follow from the center-manifold theorem in Kopell-Howard \cite{Kopell1975}. 
For $ \ovlul<\ovlur, $ a centered rarefaction wave
\begin{equation*}
u^R(x,t)= 
\begin{cases}
\ovl{u}_l, &\quad \frac{x}{t} < f'(\ovl{u}_l), \\
(f')^{-1}(\frac{x}{t}), &\quad f'(\ovl{u}_l)\leq \frac{x}{t}\leq f'(\ovl{u}_r), \\
\ovl{u}_r, &\quad \frac{x}{t}> f'(\ovl{u}_r),
\end{cases}
\end{equation*}	
is an entropy weak solution to the Riemann problem for the inviscid conservation law \cref{CL} with $ \nu=0. $

\vspace{0.3cm}

When the initial data $ u_0 $ approaches constant states as $ |x|\rightarrow+\infty, $ i.e. $ w_{0l}=w_{0r} = 0 $ in \cref{ic-2-per}, the asymptotic behaviors of solutions to \cref{CL}, \cref{ic} have been studied widely so far. The pioneering work of Hopf \cite{Hopf1950} showed the $ L^\infty $ stability of constants with a decay rate $ 1/\sqrt{t}, $ by using an explicit representation of the solution to the Burgers' equation.
Later, Il'in-Ole\v{\i}nik \cite{Oleinik1960} applied the maximum principle on the anti-derivative variables to prove the $ L^\infty $ stability of constants, shocks and rarefaction waves for general convex conservation laws. Moreover, Freist{\"{u}}hler-Serre \cite{Freistuhler1998} combined a lap number argument and maximum principle to prove the $ L^1 $ stability of viscous shock profiles. 
For more results, we refer to \cite{Nishihara1985,Harabetian1988,Howard1999} for the one-dimensional scalar case, \cite{Goodman1989,Xin1990,Hoff2000, Hoff2002,Kenig2006} for the multi-dimensional scalar case, and \cite{Matsumura1985,Goodman1986,Liu1985, Liu1997,Xin1988,Liu1988,Szepessy1993} for the important one-dimensional system case.

For the initial data which keeps oscillating at infinities, Lax \cite{Lax1957} was the first one to study the periodic data. He showed that the entropy periodic solutions to inviscid scalar conservation laws approach their constant averages in the $ L^\infty $ norm at algebraic rates. And then Glimm-Lax \cite{Glimm1970} and Dafermos \cite{Dafermos1995} used the generalized characteristics to extend the results to some $ 2\times 2 $ systems.
Besides the constants, 
Xin-Yuan-Yuan \cite{Xin2019} and Yuan-Yuan \cite{Yuan2019} proved the $ L^\infty $ stability of shocks and rarefaction waves under periodic perturbations for the scalar inviscid conservation laws, by using the generalized characteristics.  
One can also see  \cite{Dalibard2010,Dalibard2017} for the $ L^1 $ stability of the stationary viscous `` shock profile '' which connects two periodic functions as end states, where the flux in the equation \cref{CL} is $ f(x,u) $ that is periodic with respect to $ x. $
 
In this paper, we deal with the initial data \eqref{ic-2-per}, where $ w_{0l}, w_{0r} $ are arbitrary periodic functions, in order to see how the initial oscillations at infinities influence the stability of these two nonlinear waves in the viscous case.
Throughout this paper, we let $ \phi(x) $ denote any fixed shock profile solving \cref{ode1} with $ \nu =1, $ and then for any $ \nu>0, $ we define
\begin{equation}\label{def-phi-nu}
\phi^\nu(x) := \phi\left( \frac{x}{\nu}\right).
\end{equation}
Then $\phi^\nu$ solves \cref{ode1} for any $ \nu>0 $ and tends to the inviscid Lax-shock
\begin{equation*}
u^S(x,t) = \begin{cases}
\ovlul,  &\quad  x<st, \\
\ovlur,  &\quad  x>st,
\end{cases}
\end{equation*}
almost everywhere as the viscosity $ \nu\rightarrow 0+. $
With the constants $ \ovlul, \ovlur $ and periodic perturbations $ w_{0l}, w_{0r} \in L^\infty(\R) $ in \cref{ic-2-per}, we let $ u_l^\nu,$ $ u_r^\nu $ denote the periodic solutions to \cref{CL} with the respective initial data
\begin{equation}\label{def-ulur-2}
u_l^\nu(x,0) =\ovlul+w_{0l}(x), \quad u_r^\nu(x,0) =\ovlur+w_{0r}(x).
\end{equation}

\vspace{0.3cm}

Now the main results of this paper are stated as follows:

\begin{Thm}\label{Thm-2-periodic}
	Assume that $ \nu>0 $ and the initial data $ u_0\in L^\infty(\R) $ satisfies \cref{ic-2-per} with $ \ovlul>\ovlur $ and the periodic functions $ w_{0l}, w_{0r} \in L^\infty $ with the respective periods $ p_l, p_r>0, $ satisfying
	\begin{equation}\label{zero-ave-2-period}
	\frac{1}{p_l} \int_0^{p_l} w_{0l}(x) dx = \frac{1}{p_r} \int_0^{p_r} w_{0r}(x) dx = 0.
	\end{equation}
	Then the unique bounded solution $ u^\nu $ to \cref{CL}, \cref{ic} satisfies 
	\begin{equation}\label{ineq-Thm-shock}
	\sup_{x\in\R} \left| u^\nu(x,t) - \phi^\nu(x-st-X_\infty^\nu) \right| \leq Ce^{-\mu t}, \quad t>0,
	\end{equation}
	with the constant shift $ X_\infty^\nu = \frac{1}{\ovlul-\ovlur} \left( X_{\infty,1}^\nu + X_{\infty,2}^\nu\right) $ given by
	\begin{align}
	X_{\infty,1}^\nu = & \int_{-\infty}^{0} (u_0-\phi^\nu-w_{0l})(x)dx + \int_{0}^{+\infty} (u_0-\phi^\nu-w_{0r})(x) dx, \label{shift-local-2-periodic}\\
	X_{\infty,2}^\nu = & \int_{0}^{+\infty}\frac{1}{p_l}\int_{0}^{p_l} [f(u_l^\nu(x,t))-f(\ovlul)] dxdt -\frac{1}{p_l}\int_{0}^{p_l}\int_{0}^{x} w_{0l}(y) dydx \label{def-X-2} \\
	&  - \int_{0}^{+\infty}\frac{1}{p_r} \int_{0}^{p_r} [f(u_r^\nu(x,t)) -f(\ovlur)] dxdt+\frac{1}{p_r}\int_{0}^{p_r} \int_{0}^{x} w_{0r}(y) dydx. \notag
	\end{align}
\end{Thm}

\vspace{0.3cm}

\cref{Thm-2-periodic} shows that in contrast to the case of localized perturbations, besides the localized part of the initial perturbation, the periodic oscillations at infinities generate another shift $ X_{\infty,2}^\nu $ to the background viscous shock profile. The next theorem shows that this shift is non-zero in general even in the case the periods of the perturbations at $ x=\pm \infty $ are the same. 

\begin{Thm}\label{Thm-1-periodic}
	Under the assumptions of \cref{Thm-2-periodic}, if $ w_{0l}=w_{0r} = w_0, $
	the constant $ X_{\infty,2}^\nu $ defined in \cref{def-X-2} may be non-zero in general.
	More precisely, 
	\begin{enumerate}
		\item[(1) ] for the Burgers' equation, i.e. $ f(u) = u^2/2, $ it holds that $X_{\infty,2}^\nu=0;$ 
	
		\item[(2) ] for any periodic perturbation $w_0$ with zero average, if $$ 0<\|w_0\|_{L^{\infty}(\R)}< (\ovlul-\ovlur)/2, $$ 
		there exists a smooth and strictly convex flux $f$ such that $ X_{\infty,2}^\nu \neq 0. $ 
	\end{enumerate}
\end{Thm}

However, for the inviscid conservation law \cref{CL} with $ \nu=0, $ it is shown in \cite{Xin2019} that if $ w_{0l} = w_{0r}, $ the entropy solution $ u^0 $ tends to the background shock $ u^S $ with no shift, i.e. $ X_{\infty,2}^0 = 0 $ in this case.  The vanishing viscosity limit for the shift $ X_{\infty,2}^\nu $ is presented in the next theorem,  which agrees with the results in \cite{Xin2019,Yuan2019}.

\begin{Thm}\label{Thm-vanishing}
	Under the assumptions of \cref{Thm-2-periodic}, as the viscosity $ \nu\rightarrow 0+, $
	\begin{equation}\label{vanishing limit}
	X_{\infty,2}^\nu \rightarrow X_{\infty,2}^0 := - \min_{x\in\R} \int_{0}^{x} w_{0l}(y)dy + \min_{x\in\R} \int_{0}^{x} w_{0r}(y)dy.
	\end{equation}
	Furthermore, if both $ w_{0l} $ and $ w_{0r} $ have bounded total variations on their respective periodic domains, there exists a constant $ C>0, $ independent of the viscosity $ \nu, $ such that
	\begin{equation}\label{rate-2}
	\left| X_{\infty,2}^\nu - X_{\infty,2}^0 \right| \leq C \nu^{1/5}.
	\end{equation}
\end{Thm}

\vspace{0.2cm}

At last, we state the result for rarefaction waves.

\begin{Thm}\label{Thm-rare}
	Assume that $ \nu>0 $ and the initial data $ u_0 $ satisfies \cref{ic-2-per} with $ \ovlul<\ovlur $ and the periodic functions $ w_{0l}, w_{0r} \in L^\infty $ with the respective periods $ p_l, p_r>0, $ satisfying \cref{zero-ave-2-period}. Then the unique bounded solution $ u^\nu $ to \eqref{CL}, \eqref{ic} satisfies
	$$
	\sup_{x\in\R} | u^\nu(x,t)-u^R(x,t) | \rightarrow 0 ~ \quad \text{as}~t\rightarrow \infty.
	$$
\end{Thm}

\vspace{0.3cm}

Consequently, different from the localized perturbations, i.e.,  the background shock can absorb all localized perturbations on the two sides and finally tends to itself with no shift (if the perturbation has zero mass), the periodic perturbations at infinities produce infinite perturbations onto the background shock, ended up with a constant shift, which cannot be determined explicitly by the initial perturbations in the viscous case.

\vspace{0.3cm}

The \textbf{main difficulty} to prove \cref{Thm-2-periodic} is that the perturbation $ u^\nu-\phi^\nu $ is not integrable anymore, which makes it difficult to use an anti-derivative argument as before that plays an important role in the previous study of the stability under localized perturbations. In fact, if one considers the equation of the perturbation, the coefficient of the zero-order term is $ f''(\phi^\nu)(\phi^\nu)'<0, $ which makes it harder to use either the maximum principle or the energy method.
\textbf{One of the key elements} of our proof is that we find an ansatz $ \psi_{X^\nu(t)}^\nu(x,t) $, where $ X^\nu(t) $ is a shift function (see \eqref{ode-shift}) such that the difference $ u^\nu - \psi_{X^\nu(t)}^\nu $ is integrable and has zero mass for large time.
Therefore, it is plausible to study the equation of the anti-derivative variable of the difference  $ u^\nu - \psi_{X^\nu(t)}^\nu $
(for the Burgers' equation, the ansatz actually coincides with the solution $ u^\nu $ at an arithmetic sequence of time $ \{t_k\}; $ see \cref{Prop-Burgers}).
We show that the limit of $ X^\nu(t) -st $  as $ t\rightarrow +\infty $ is actually the constant shift $ X_\infty^\nu $ in \cref{Thm-2-periodic}. And although the ansatz $ \psi_{X^\nu(t)}^\nu $ is not a solution to \cref{CL}, the error (see \cref{source}) decays exponentially both in space and in time. Then following the idea of Il'in-Ole\v{\i}nik \cite{Oleinik1960}, one can construct auxiliary functions and use the maximum principle to obtain our main results. 

\vspace{0.5cm}


\section{Preliminaries and ansatz}

We first present some useful lemmas and introduce some notations, and then construct the ansatz. 
Important properties of the ansatz will be stated as propositions, which will be proved in the rest of the paper.
In the end of this section, we outline the organizations of the proof.

\begin{Lem}\label{Lem-periodic}
	Assume that $ u_0(x) \in L^{\infty}(\R) $ is periodic with period $ p>0 $ and average $ \ovl{u} = \frac{1}{p} \int_{0}^{p} u_0(x) dx. $ Then there exists a constant $ \alpha>0, $ depending only on $ p $ and $ \nu, $ such that for any integers $ k, l\geq 0, $ the periodic solution $ u^\nu(x,t) $ to \eqref{CL},\cref{ic} satisfies that
	\begin{equation}\label{expdecay0}
		\| \p_t^k \p_x^l (u^\nu-\ovl{u}) \|_{L^{\infty}(\R)} \leq C e^{-\alpha t},  \quad t\geq 1,
	\end{equation}
	where $ C>0 $ is independent of time $ t. $
\end{Lem}
The proof of \cref{Lem-periodic} can be obtained by standard energy estimates and the Poincar\'{e} inequality, which is given in the Appendix A. Thus, there exists an $ \alpha>0, $ depending on $ p_l,p_r $ and $ \nu, $ such that $ u_l^\nu $ and $ u_r^\nu $ satisfy \cref{expdecay0} with $ \ovl{u} = \ovlul $ and $ \ovlur, $ respectively.

\vspace{0.3cm}

\begin{Lem}\label{Lem-decay}
	Assume that $u_0, \tilde{u}_0 \in L^{\infty}(\R)$ and there exist constants $ C>0 $ and $ \delta \in \R $ such that
	$$|u_0(x)-\tilde{u}_0(x)| \leq Ce^{\delta x}, \quad  x \in \R. $$ 
	Then it holds that
	\begin{equation*}
	|S_t^\nu u_0-S_t^\nu \tilde{u}_0|\leq C(t) e^{\delta x},  
	\quad x\in\R, ~t> 0,
	\end{equation*}
	where the constant $ C(t)>0 $ is bounded on any compact subset of $ [0,+\infty). $
\end{Lem}
The proof of \cref{Lem-decay}, based on approximate solutions solving linear parabolic equations, is given in the Appendix B. 

\vspace{0.3cm}

For the periodic solutions $ u^\nu_l $ and $ u_r^\nu $ defined in \cref{def-ulur-2}, the following result can follow from \cref{Lem-decay}.

\begin{Lem}\label{Lem-decay-ulr}
	Assume that the initial data $ u_0 \in L^\infty(\R) $ satisfies \cref{ic-2-per}. Then the unique solution $ u^\nu $ to \cref{CL}, \cref{ic} satisfies that
	\begin{equation}\label{end-states}
	\begin{cases}
	|u^\nu(x,t)-u_l^\nu(x,t)| \leq C(t) e^{\beta_0 x}, & \\
	|u^\nu(x,t)-u_r^\nu(x,t)| \leq C(t) e^{-\beta_0 x}, &
	\end{cases} \quad x\in\R, t > 0,
	\end{equation}
	where $ C(t) $ is bounded on any compact subset of $ [0,+\infty). $
\end{Lem}

\vspace{0.2cm}

To prove \cref{Thm-2-periodic}, we now construct the ansatz.
For the shock profile $ \phi^\nu $ defined in \cref{def-phi-nu}, we first define the function:
\begin{equation}\label{def-g}
g^\nu(x):=\frac{\phi^\nu(x)-\ovlur}{\ovlul-\ovlur}, \quad x\in\R.
\end{equation}

\begin{Lem}\label{Lem-g}
	The function $g^\nu\in C^{\infty}(\R)$ satisfies that
	\begin{itemize}
		\item[(i)] there exist positive constants $\beta_1 $ and $ \beta_2,$ depending on $\nu, f, \ovlul $ and  $\ovlur,$ such that
		\begin{equation}\label{g1}
		\beta_1\leq \frac{-(g^\nu)'(x)}{(\ovlul-\ovlur)g^\nu(x)(1-g^\nu(x))}\leq \beta_2, \qquad x\in \R;
		\end{equation}
		\item[(ii)] with the inequality \eqref{g1}, there exists a constant $C>0,$ depending on $\nu, f, \ovlul$ and $ \ovlur,$ such that
		\begin{equation}\label{g2}
		\begin{aligned}
		&\frac{1}{C}e^{-\beta_2 x}\leq  g^\nu(x) \leq Ce^{-\beta_1 x}, &\quad x>0,\\
		&\frac{1}{C}e^{\beta_2 x}\leq  1-g^\nu(x) \leq C e^{\beta_1 x}, &\quad x<0,\\ 
		&\frac{1}{C}e^{-\beta_2|x|}\leq  -(g^\nu)'(x)\leq C e^{-\beta_1|x|}, &\quad x\in \R.
		\end{aligned}
		\end{equation}
	\end{itemize}
\end{Lem}
The proof of \cref{Lem-g} can be found in \cite{Goodman1992}. And we will give a simplified proof for the scalar viscous conservation laws in the Appendix C.
For convenience, in the following part of this paper we define
\begin{equation}\label{def-beta}
\beta := \min\{ \beta_0, \beta_1, \beta_2 \}>0.
\end{equation}

\vspace{0.2cm}

Then for any $ C^1 $ curve $ \xi(t): [0,+\infty)\rightarrow\R, $ we set
\begin{equation}\label{def-shift-g}
g_{\xi}^\nu(x) := g^\nu(x-\xi(t)), \quad x\in\R, t\geq 0,
\end{equation}
with the derivatives:
\begin{equation*}
(g_{\xi}^\nu)^{(k)}(x) := (g^\nu)^{(k)}(x-\xi(t)), \quad k\geq 1.
\end{equation*}
Motivated by \cref{Lem-decay-ulr} and the formula of the viscous shock profile
$$ \phi_\xi^\nu(x) := \phi^\nu(x-\xi(t)) = \ovlul g_\xi^\nu(x) + \ovlur (1-g_\xi^\nu(x)), $$
we construct the ansatz as
\begin{equation}\label{def-psi}
\psi_{\xi}^\nu(x,t) := u_l^\nu(x,t) g_{\xi}^\nu(x) + u_r^\nu(x,t)(1-g_{\xi}^\nu(x)).	
\end{equation}
It is noted that the shift $\xi(t)$ appears only in $g^\nu$.
Thus $\psi_{\xi}^\nu$ satisfies
\begin{equation}\label{psi-equ}
\p_t \psi_{\xi}^\nu + \p_x f(\psi_{\xi}^\nu) - \nu \p_x^2 \psi_{\xi}^\nu=h_{\xi}^\nu,
\end{equation}
where the source term $h_\xi^\nu$ is
\begin{equation}\label{source}
\begin{aligned}
h_{\xi}^\nu =~ & \p_x f(\psi_\xi^\nu) -\p_x f(u_l^\nu)g_\xi^\nu - \p_x f(u_r^\nu)(1-g_\xi^\nu) - 2 \nu \p_x(u_l^\nu-u_r^\nu) (g_\xi^\nu)' \\
& - (u_l^\nu - u_r^\nu) \left( (g_\xi^\nu)' ~\xi' + \nu (g_\xi^\nu)'' \right),
\end{aligned}
\end{equation}
which can be rewritten as
\begin{equation}\label{source-1}
\begin{aligned}
h_\xi^\nu= ~& \p_x \left[ \left(f(\psi_\xi^\nu)-f(u_l^\nu)\right)g_\xi^\nu + \left(f(\psi_\xi^\nu)-f(u_r^\nu)\right)(1-g_\xi^\nu) - 2\nu (u_l^\nu-u_r^\nu)(g_\xi^\nu)' \right] \\ 
& + \left(f(u_l^\nu)-f(u_r^\nu)\right) (g_\xi^\nu)' - (u_l^\nu-u_r^\nu) (g_\xi^\nu)' \xi' + \nu (u_l^\nu-u_r^\nu) (g_\xi^\nu)''.
\end{aligned}	
\end{equation}
This, together with the fact $ \nu (g_\xi^\nu)''= f'(\phi_\xi^\nu)(g_\xi^\nu)' - s (g_\xi^\nu)', $ yields that
\begin{equation}\label{source-2}
\begin{aligned}
h_\xi^\nu = & \left( f'(\psi_\xi^\nu) - f'(u_l^\nu) \right)\p_x u_l^\nu g_\xi^\nu +  \left( f'(\psi_\xi^\nu) - f'(u_r^\nu) \right) \p_x u_r^\nu (1-g_\xi^\nu) \\
& - 2 \nu \p_x(u_l^\nu-u_r^\nu) (g_\xi^\nu)' - (u_l^\nu-u_r^\nu)\left(\xi'(t)- s + f'(\phi_\xi^\nu)-f'(\psi_\xi^\nu) \right) (g_\xi^\nu)'.
\end{aligned}
\end{equation}
The formulas \eqref{source-1} and \eqref{source-2} will be used later. 

The ansatz $ \psi_{\xi}^\nu $ is expected to satisfy $\int_\R \left( u^\nu - \psi_{\xi}^\nu \right)(x,t) dx \equiv 0, $ so that the anti-derivative variable can vanish at both infinities $ |x| \rightarrow\infty. $
Under the assumptions of \cref{Thm-2-periodic}, it follows from \cref{Lem-periodic} that there exists a large $ T_0>0 $ such that
\begin{equation}\label{small}
0<\frac{\ovlul-\ovlur}{2} < \left( u_l^\nu-u_r^\nu \right)(x,t) < 2 (\ovlul-\ovlur), \quad x\in\R, t\geq T_0.
\end{equation}
The time $ T_0 $ is chosen to guarantee $ \int_\R (u_l^\nu-u_r^\nu)(x,t) (g^\nu_\xi)' dx < 0 $ for all $ t\geq T_0. $ 
It is noted that the terms appearing in the square brackets of \eqref{source-1} vanish as $ |x| \rightarrow \infty. $ For the equation of the perturbation,
\begin{equation}\label{equ-u-psi}
\p_t (u^\nu-\psi_\xi^\nu) + \p_x \left( f(u^\nu)-f(\psi_\xi^\nu) \right) = \nu \p_x^2 (u^\nu -\psi_\xi^\nu) - h_\xi^\nu,
\end{equation}
we aim to choose a curve $ \xi = X^\nu(t) $ such that $ \int_{\R} h_{X^\nu}^\nu(x,t) dx =0. $ Integrating \eqref{source-1} with respect to $ x, $ we require that the curve $ X^\nu(t) $ solves the problem:
\begin{equation}\label{ode-shift}
\begin{cases}
(X^\nu)'(t) = F^\nu(X^\nu(t),t), & \quad t>T_0, \\
X^\nu(T_0)=X_0^\nu, & 
\end{cases}
\end{equation}
where 
\begin{equation}\label{def-F-nu}
F^\nu(\xi,t) :=
\dfrac{\int_\R \left[ \nu (u_l^\nu-u_r^\nu) (g_\xi^\nu)''  +\left(f(u_l^\nu)-f(u_r^\nu)\right) (g_\xi^\nu)' \right] dx}{\int_\R (u_l^\nu-u_r^\nu) (g_\xi^\nu)' dx},
\end{equation}
and the initial data $ X^\nu(T_0) = X_0^\nu $ is chosen so that
\begin{equation}\label{def-X-0}
\int_\R (u^\nu-\psi_{X_0^\nu}^\nu)(x,T_0) dx =0.
\end{equation}

\begin{Prop}\label{Prop-shift}
	Under the assumptions of \cref{Thm-2-periodic},
	there exists a unique $ X_0^\nu\in\R $ such that \cref{def-X-0} holds, and	the problem \cref{ode-shift} admits a unique $ C^\infty $ solution $ X^\nu(t) $ with
	\begin{equation}\label{shift-infty} 
	|X^\nu(t)-st-X_\infty^\nu| \leq C e^{-\alpha t}, \quad t\geq T_0,
	\end{equation}	
	where $ X_\infty^\nu $ is the constant defined in \cref{Thm-2-periodic} and the constant $C>0$ is independent of time $ t. $
\end{Prop}

\begin{Rem}
	The choice of $ X^\nu(t) $ can make the source term $ h_{X^\nu}^\nu( x,t)$ and its anti-derivative variable $ \int_{-\infty}^{x} h_{X^\nu}^\nu( y,t) dy $ decay exponentially fast both in space and in time (see Proposition \ref{prop-U} for details), which plays an important role in the proof of \cref{Thm-2-periodic}. 
\end{Rem}

With the shift curve $ X^\nu(t) $ so determined, we have the following results.

\begin{Prop}\label{Prop-ansatz}
	Under the assumptions of \cref{Thm-2-periodic}, there exist constants $C>0$ and $0<\mu\leq \alpha,$ independent of time $ t, $ such that the unique bounded solution $ u^\nu $ to \eqref{CL}, \eqref{ic} satisfies
	\begin{equation}\label{ineq-ansatz}
		\sup_{x \in \R}|u^\nu(x,t)-\psi_{X^\nu(t)}^\nu(x,t)|\leq C e^{-\mu t},  \quad t\geq T_0.
	\end{equation}
\end{Prop}

\begin{Prop}\label{Prop-Burgers}
	Assume that $ f(u) = u^2/2 $ in \cref{CL}, and the initial data satisfies
	\begin{equation}\label{ic-Burgers}
	u_0(x) = \phi^\nu(x) + w_0(x),
	\end{equation}
	where $ w_0(x) $ is periodic with period $ p>0 $ and zero average. Then at each time $ t_k = kp/(\ovlul-\ovlur), ~k= 0, 1, 2, \cdots, $ the solution $ u^\nu $ to \cref{CL}, \cref{ic} satisfies that
	\begin{equation}\label{burgers-sol}
	\begin{aligned}
	& u^\nu(x,t_k) =  \psi_{st_k}^\nu(x,t_k),  \quad x \in \R.
	\end{aligned}
	\end{equation}
	And the shift function $ X^\nu $ in \cref{Prop-shift}
	satisfies that for $ t_k \geq T_0, $
	\begin{equation}\label{X-k}
	X^\nu(t_k) = s t_k.
	\end{equation}
\end{Prop}

\begin{Rem}
	It can follow from \cref{burgers-sol} and \cref{X-k} that for $ t_k \geq T_0, $ $ u^\nu(x,t_k) = \psi^\nu_{X^\nu}(x,t_k), $ which is compatible with \cref{ineq-ansatz}. And it also implies that the ansatz $ \psi_{X^\nu}^\nu $ is a suitable choice to approach the actual solution $ u^\nu. $
\end{Rem}

\vspace{0.3cm}

This paper proceeds as follows: We first prove Theorems \ref{Thm-2-periodic}, \ref{Thm-1-periodic} and \ref{Thm-vanishing} for viscous shock profiles in \cref{Sec-shock}. 
More precisely, it is shown that \cref{Thm-2-periodic} follows from Propositions \ref{Prop-shift} and \ref{Prop-ansatz} easily.  
\cref{Prop-shift} for the shift function $ X^\nu(t)$ is proved in \cref{Sec-shift}.
\cref{Prop-ansatz} is proved in \cref{Sec-ansatz-proof}, which is independent of \cref{Prop-shift}.
In \cref{Sec-Burgers}, we prove the result (1) in \cref{Thm-1-periodic} and \cref{Prop-Burgers} for the Burgers' equation.
The proof of \cref{Thm-1-periodic} is completed in \cref{Sec-counter-ex}, where a strictly convex flux $ f $ is constructed such that $X_{\infty,2}^\nu \neq 0. $
In \cref{Sec-vanishing}, we prove \cref{Thm-vanishing} for the vanishing viscosity limit for $ X_{\infty,2}^\nu. $ At last, Theorem \ref{Thm-rare} for rarefaction waves is proved in \cref{Sec-rare}. 

\vspace{0.5cm}


\section{Stability of shock profiles}\label{Sec-shock}

\cref{Thm-2-periodic} can follow from \cref{Lem-periodic}, Propositions \ref{Prop-shift} and \ref{Prop-ansatz}. In fact, it holds that
\begin{align*}
| u^\nu(x,t) - \phi^\nu(x-st-X_\infty^\nu) | & \leq | u^\nu(x,t) - \psi_{X^\nu}^\nu(x,t)| + \left| \psi_{X^\nu}^\nu(x,t) - \phi^\nu(x-X^\nu(t)) \right| \\
& \quad + | \phi^\nu(x-X^\nu(t)) - \phi^\nu(x-st-X_\infty^\nu) | \\
& \leq C e^{-\mu t} + |u_l^\nu(x,t)-\ovlul| + |u_r^\nu(x,t)-\ovlur| + Ce^{-\alpha t} \\
& \leq C e^{-\mu t},
\end{align*}
which proves \cref{ineq-Thm-shock}. Thus, it remains to prove Propositions \ref{Prop-shift} and \ref{Prop-ansatz} to finish the proof of \cref{Thm-2-periodic}.


\subsection{Shift function} \label{Sec-shift}~\\

\subsubsection{Existence and uniqueness of the shift function}~\\

For any $ \xi\in\R, $ 
$$ \int_\R (u^\nu - \psi^\nu_\xi)(x,T_0) dx = \int_\R \left( u^\nu - u_l^\nu g^\nu_\xi - u_r^\nu (1-g^\nu_\xi) \right)(x,T_0) dx. $$ 
As $ \xi \rightarrow -\infty $ (resp., $ +\infty $), $ g^\nu_\xi(x) = g^\nu(x-\xi) \rightarrow 0 $ (resp., $ 1 $), then due to \cref{end-states} and \cref{small}, one has that $ \int_\R (u^\nu - \psi^\nu_\xi)(x,T_0) dx \rightarrow +\infty $ (resp., $ -\infty $) as $ \xi \rightarrow -\infty $ (resp., $ +\infty $). Thus, there exists an $ X_0^\nu \in \R $ such that \cref{def-X-0} holds. And the uniqueness follows from \cref{small} and the strict monotonicity of $ g^\nu. $

Now we will prove the existence and uniqueness of $ X^\nu(t) $ solving the problem \eqref{ode-shift}.
\begin{Lem}\label{Lem-shift}
	The problem \eqref{ode-shift} has a unique $ C^\infty $ solution $X^\nu(t):[T_0,+\infty)\rightarrow\R,$ satisfying
	\begin{equation}\label{X-prime}
	\lvert (X^\nu)'(t) - s \rvert \leq C e^{-\alpha t},  \quad  t\geq T_0,
	\end{equation}
	where $ C>0 $ is independent of time $ t. $
\end{Lem}

\begin{proof}
	By \cref{Lem-g} and \cref{small}, the denominator of $ F^\nu(\xi,t) $ satisfies that
	\begin{align*}
	\int_\R (u_l^\nu-u_r^\nu)(x,t)(g_\xi^\nu)'(x) dx & \leq
	\frac{\ovlul-\ovlur}{2} \int_\R (g_\xi^\nu)'(x) dx \\
	& = -\frac{\ovlul-\ovlur}{2} <0.	
	\end{align*}
	Since for any $ k\geq 1, $ $ (g^\nu)^{(k)} $ is integrable and $ u_l^\nu, u_r^\nu $ are bounded and smooth, thus $ F^\nu $ is smooth and all the derivatives are bounded. Then the existence and uniqueness of $ X^\nu $ can follow from the Cauchy-Lipschitz theorem.
	
	Now we prove \eqref{X-prime}.
	By Lemmas \ref{Lem-periodic} and \ref{Lem-g}, one can get that
	\begin{align*}
	\int_{-\infty}^{+\infty} (u_l^\nu-u_r^\nu) (g_{X^\nu}^\nu)'' dx &  = \int_{-\infty}^{+\infty} (u_l^\nu- \ovlul-u_r+\ovlur) (g_{X^\nu}^\nu)'' dx = O(e^{-\alpha t}), \\
	\int_{-\infty}^{+\infty}(f(u_l^\nu)-f(u_r^\nu)) (g_{X^\nu}^\nu)' dx & = - f(\ovlul)+f(\ovlur) + O(e^{-\alpha t}), \\
	\int_{-\infty}^{+\infty} (u_l^\nu-u_r^\nu) (g_{X^\nu}^\nu)' dx & = -\ovlul + \ovlur + O(e^{-\alpha t}).
	\end{align*}
	Therefore, \eqref{X-prime} holds true.
\end{proof}

\vspace{0.2cm}

\subsubsection{The limit of the shift function as $ t\rightarrow +\infty $}\label{sec-shift-infinity}~\\

In order to compute $ \lim\limits_{t\rightarrow +\infty} X^\nu(t)-st $ in terms of the initial data $ u_0(x), $ the information of the solution $ u^\nu(x,t) $ for $ t \in [0,T_0] $ should be used. However, 
when $t\in[0,T_0]$, it may fail to find a unique $\xi$ such that $ \int_\R \left( u^\nu - \psi_{\xi}^\nu \right)(x,t) dx = 0 $ or fail to ensure the denominator of $ F^\nu, $ $ \int_\R (u_l^\nu-u_r^\nu)(x,t)(g_\xi^\nu)'(x) dx, $ is non-zero. Therefore, we need the following modifications to extend the definition of $X^\nu(t)$ on $ [0,T_0]. $

For the bounded periodic solutions $ u_l^\nu $ and $ u_r^\nu, $ one can first choose a large number $ M > 0 $ such that for all $ t \geq 0, $
\begin{equation}\label{add-M}
\begin{aligned}
\int_\R (u_l^\nu-u_r^\nu)(x,t)(g_\xi^\nu)'(x) dx - M & \leq C \int_\R \left| (g_\xi^\nu)'(x) \right|dx - M \\
& = C - M < 0.
\end{aligned}
\end{equation}
Thanks to \cref{add-M}, there exists a unique solution $ \hat{X}^\nu(t) \in C^1[0,T_0] $ to the problem
\begin{equation}\label{ode-shift-2}
\begin{cases}
(\hat{X}^\nu)'(t) = \hat{F}^\nu(\hat{X}^\nu(t),t), & \quad t \leq T_0, \\
\hat{X}^\nu(T_0)=X_0^\nu, & 
\end{cases}
\end{equation}
where 
\begin{equation}\label{def-F-nu-2}
\hat{F}^\nu(\xi,t) :=
\dfrac{\int_\R \left[ \nu (u_l^\nu-u_r^\nu) (g_\xi^\nu)''  +(f(u_l^\nu)-f(u_r^\nu)) (g_\xi^\nu)' \right] dx}{\int_\R (u_l^\nu-u_r^\nu) (g_\xi^\nu)' dx - M}.
\end{equation}
The proof of existence and uniqueness of $ \hat{X}^\nu $ is similar to that in \cref{Lem-shift}.
Now we claim that
\begin{equation}\label{def-X-hat-0}
M (X_0^\nu - \hat{X}^\nu_0) + \int_\R \left(u_0(x)-\psi_{\hat{X}^\nu_0}^\nu(x,0) \right) dx = 0,
\end{equation}
where $ \hat{X}^\nu_0 := \hat{X}^\nu(0). $

Indeed, for any $ N>0, $ one can choose a cut-off function $ \varphi_N(x) \in C^\infty_0(\R) $ satisfying $ \varphi_N(x) = 1, $ if $ |x|<N, $ and $ \varphi_N(x)=0, $ if $ |x|>N+1. $ Then by multiplying $ \varphi_N(x) $ on each side of \eqref{equ-u-psi} and integrating by parts, one can get that for any $ t>0, $
\begin{align}
& \int_\R (u^\nu-\psi_{\hat{X}^\nu}^\nu)(x,t) \varphi_N(x) dx \label{equ-11} \\
= & \int_\R (u^\nu-\psi_{\hat{X}^\nu}^\nu)(x,0) \varphi_N(x) dx  + \nu \int_0^{t} \int_\R (u^\nu-\psi_{\hat{X}^\nu}^\nu) \varphi_N'' dx d\tau \notag \\
& + \int_0^{t} \int_\R \left( f(u^\nu)-f(\psi^\nu_{\hat{X}^\nu}) \right) \varphi_N' dxd\tau - \int_0^{t} \int_\R h_{\hat{X}^\nu}^\nu \varphi_N dxd\tau, \notag
\end{align}
Thus, by applying Lemmas \ref{Lem-g} and \ref{Lem-psi}, 
one can take limit $ N \rightarrow +\infty $ and use the dominated convergence theorem in \cref{equ-11} to get that
\begin{equation}\label{equality-1}
\int_\R (u^\nu-\psi_{\hat{X}^\nu}^\nu)(x,t) dx = \int_\R \left(u_0(x)-\psi_{\hat{X}^\nu_0}^\nu(x,0)\right) dx - \int_0^t \int_\R h_{\hat{X}^\nu}^\nu(x,\tau) dx d\tau.
\end{equation}
It follows from \cref{source-1} and \cref{ode-shift-2} that 
\begin{align*}
- \int_\R h_{\hat{X}^\nu}^\nu(x,t) dx = M \frac{d}{dt} \hat{X}^\nu(t).
\end{align*}
This, together with \cref{def-X-0} and \cref{equality-1}, yields \cref{def-X-hat-0}.

\vspace{0.2cm}

Now, we define 
\begin{equation}\label{def-X-tilde}
\tilde{X}^\nu(t) := \begin{cases}
X^\nu(t), \quad t>T_0, & \\
\hat{X}^\nu(t), \quad 0\leq t \leq T_0, &
\end{cases}
\end{equation}
then $ \tilde{X}^\nu $ is a Lipschitz continuous curve on $ [0,+\infty). $
For $ y \in (0,1), N\in \mathbb{N}^* $ and $ t>T_0, $ we define the domain
\begin{align*}
\Omega^N_y & := \{(x,\tau):~ \tilde{X}^\nu(\tau) +(-N+y) p_l \leq x\leq \tilde{X}^\nu(\tau)+ (N+y) p_r,~ 0 \leq \tau\leq t \};
\end{align*}
see Figure \ref{Fig-domain}. 
\begin{figure}[htbp!]
	\scalebox{0.5}{\includegraphics{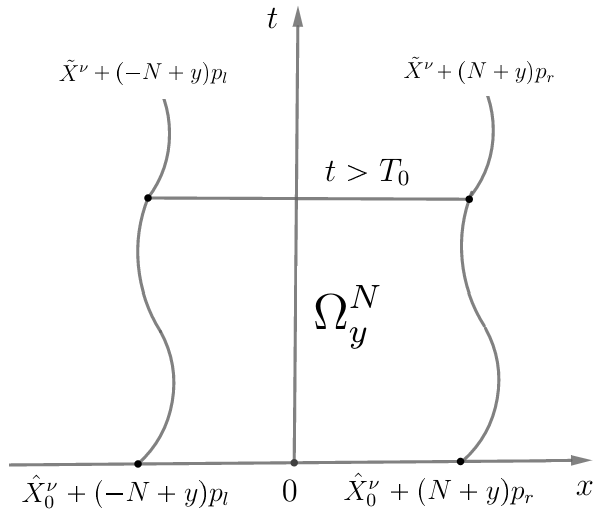}}
	\caption{}\label{Fig-domain}
\end{figure}


It follows from the equations of $ u_l^\nu $ and $ u_r^\nu $ that
$$ \iint_{\Omega^N_y} \left\lbrace \left( \p_t u_l^\nu + \p_x f(u_l^\nu) - \nu \p_x^2 u_l^\nu \right) g_{\tilde{X}^\nu}^\nu+ \left( \p_t u_r^\nu + \p_x f(u_r^\nu) - \nu \p_x^2 u_r^\nu \right)(1-g_{\tilde{X}^\nu}^\nu) \right\rbrace dx d\tau = 0. $$
Then integration by parts yields that
\begin{equation}\label{00}
\begin{aligned}
& \iint_{\Omega^N_y} \left\lbrace -(\tilde{X}^\nu)'(t) (u_l^\nu-u_r^\nu) + \left(f(u_l^\nu)- \nu\partial_x u_l^\nu\right)- (f(u_r^\nu)- \nu\partial_x u_r^\nu) \right\rbrace (g_{\tilde{X}^\nu}^\nu)' dx d\tau \\
= ~& A^N(y,t) - A^N(y,0) - B_l^N(y,t) + B_r^N(y,t),
\end{aligned}
\end{equation}
where 
\begin{equation}\label{terms-div}
\begin{aligned}
A^N(y,t) &:= \int_{\tilde{X}^\nu(t)+(-N+y)p_l}^{\tilde{X}^\nu(t)+(N+y)p_r}\Big[u_l^\nu(x,t)g_{\tilde{X}^\nu}^\nu(x)+u_r^\nu(x,t)(1-g_{\tilde{X}^\nu}^\nu(x))\Big] dx, \\
A^N(y,0) &:= \int_{\hat{X}_0^\nu + (-N+y)p_l}^{\hat{X}_0^\nu + (N+y)p_r} \left[u_l^\nu(x,0)g_{\hat{X}_0^\nu}^\nu(x)+u_r^\nu(x,0) \left(1-g_{\hat{X}_0^\nu}^\nu(x)\right)\right] dx, \\
B_l^N(y,t) &:= \int_{0}^{t} \Big\{ \left(f(u_l^\nu)- \nu \p_x u_l^\nu \right) g_{\tilde{X}^\nu}^\nu + \left(f(u_r^\nu)- \nu\p_x u_r^\nu \right) (1-g_{\tilde{X}^\nu}^\nu) \\
& \qquad \qquad - (\tilde{X}^\nu)'(\tau) \left[ u_l^\nu g_{\tilde{X}^\nu}^\nu + u_r^\nu (1-g_{\tilde{X}^\nu}^\nu) \right] 
\Big\} (\tilde{X}^\nu(\tau)+(-N+y)p_l,\tau) d\tau, \\
B_r^N(y,t) &:= \int_{0}^{t} \Big\{ \left(f(u_l^\nu)- \nu \p_x u_l^\nu \right) g_{\tilde{X}^\nu}^\nu + \left(f(u_r^\nu)- \nu\p_x u_r^\nu \right) (1-g_{\tilde{X}^\nu}^\nu) \\
& \qquad \qquad - (\tilde{X}^\nu)'(\tau) \left[ u_l^\nu g_{\tilde{X}^\nu}^\nu + u_r^\nu (1-g_{\tilde{X}^\nu}^\nu) \right] 
\Big\} (\tilde{X}^\nu(\tau)+(N+y)p_r,\tau) d\tau.
\end{aligned}	
\end{equation}

It follows from \cref{ode-shift} and \cref{ode-shift-2} that the left hand side of \eqref{00}
\begin{align*}
& \iint_{\Omega^N_y} \quad \cdots ~ dxd\tau = \iint_{\Omega^N_y \cap \{ 0<\tau<T_0\}} \cdots ~dxd\tau \quad + \iint_{\Omega^N_y \cap \{ T_0<\tau<t\}} \cdots ~dxd\tau \\
\rightarrow \quad & - M \int_{0}^{T_0} (\hat{X}^\nu)'(t) dt + 0 =-  M(X_0^\nu - \hat{X}_0^\nu) \qquad \text{ as } N\rightarrow +\infty.
\end{align*}
Thus it remains to evaluate the right hand side of \eqref{00} as $N\rightarrow +\infty$.
\vspace{0.6cm}

\textbf{(i) The integrals on $ \{\tau =0\} $ and $ \{\tau = t \} $.}

\vspace{0.2cm}

Set
\begin{equation}\label{def-wlr}
w_l^\nu(x,t):=u_l^\nu(x,t)-\ovlul, \quad w_r^\nu(x,t):=u_r^\nu(x,t)-\ovlur. 
\end{equation}
Then it follows from \cref{Lem-periodic} that $ \| w_l^\nu \|_{L^\infty} + \| w_r^\nu \|_{L^\infty} \leq Ce^{-\alpha t}. $
Since $ w_{0l} $ and $ w_{0r} $ are both of zero average, one has that
\begin{equation}\label{J-N}
\begin{aligned}
& J^N(y,t) =:~  A^N(y,t) - A^N(y,0) \\ 
=~ & \int_{(-N+y)p_l}^{(N+y)p_r} \Big[ w_l^\nu(x+\tilde{X}^\nu(t),t) g^\nu(x) + w_r^\nu(x+\tilde{X}^\nu(t),t) \left(1-g^\nu(x)\right) \Big] dx  \\
& -\int_{\hat{X}^\nu_0+(-N+y)p_l}^{\hat{X}^\nu_0+(N+y)p_r} \Big[ w_{0l}(x) g^\nu (x-\hat{X}^\nu_0 ) + w_{0r}(x)\left(1-g^\nu (x-\hat{X}^\nu_0 )\right)\Big] dx \\
= ~& O(e^{-\alpha t}) +  \int_{\hat{X}^\nu_0+(-N+y)p_l}^{\hat{X}^\nu_0+yp_l} (w_{0l}-w_{0r})(x) \left(1-g^\nu(x-\hat{X}^\nu_0)\right) dx \\
& - \int_{\hat{X}^\nu_0+yp_l}^{\hat{X}^\nu_0+yp_r} \Big[ w_{0l}(x) g^\nu(x-\hat{X}^\nu_0) + w_{0r}(x)\left(1-g^\nu(x-\hat{X}^\nu_0)\right)\Big] dx \\
& - \int_{\hat{X}^\nu_0+yp_r}^{\hat{X}^\nu_0+(N+y)p_r} (w_{0l}-w_{0r})(x) g^\nu(x-\hat{X}^\nu_0) dx.
\end{aligned}
\end{equation}
Then one can get that
\begin{equation}\label{J-def}
\begin{aligned}
J(y,t) := & \lim\limits_{N\rightarrow \infty} J^N(y,t) \\
= ~& O(e^{-\alpha t}) + \int_{-\infty}^{\hat{X}^\nu_0+yp_l} (w_{0l}-w_{0r})(x) \left(1-g^\nu(x-\hat{X}^\nu_0)\right) dx \\
&  - \int_{\hat{X}^\nu_0+yp_l}^{\hat{X}^\nu_0+yp_r} \Big[ w_{0l}(x) g^\nu(x-\hat{X}^\nu_0) + w_{0r}(x)\left(1-g^\nu(x-\hat{X}^\nu_0)\right)\Big] dx \\
& - \int_{\hat{X}^\nu_0+yp_r}^{+\infty} (w_{0l}-w_{0r})(x) g^\nu(x-\hat{X}^\nu_0) dx.
\end{aligned}
\end{equation}
By \cref{def-X-hat-0}, it holds that
\begin{align}
M( X_0^\nu & - \hat{X}^\nu_0) = -\int_\R \Big[ u_0(x) - \phi^\nu(x-\hat{X}^\nu_0) - w_{0l}(x) g^\nu(x-\hat{X}^\nu_0)  \label{M-equality}\\ 
& \qquad \qquad \qquad \quad - w_{0r}(x) \left(1-g^\nu(x-\hat{X}^\nu_0)\right) \Big] dx \notag \\
& = -\int_{-\infty}^{0} \left( u_0-\phi^\nu-w_{0l} \right)(x)dx - \int_0^{+\infty} \left( u_0-\phi^\nu-w_{0r} \right)(x)dx \notag \\
& \quad + (\ovlul-\ovlur) \hat{X}^\nu_0 - \int_{-\infty}^{0} (w_{0l}-w_{0r})(x) \left(1-g^\nu(x-\hat{X}^\nu_0)\right) dx \notag \\
& \quad + \int_{0}^{+\infty} (w_{0l}-w_{0r})(x) g^\nu(x-\hat{X}^\nu_0) dx. \notag
\end{align}
This, together with \cref{J-def}, yields that
\begin{align}
J(y,t) = ~& O(e^{-\alpha t}) - M(X_0^\nu - \hat{X}^\nu_0) + (\ovlul-\ovlur) \hat{X}^\nu_0 \notag \\
& - \int_{-\infty}^{0} \left( u_0-\phi^\nu-w_{0l} \right)(x)dx - \int_0^{+\infty} \left( u_0-\phi^\nu-w_{0r} \right)(x)dx \notag \\
& + \int_0^{\hat{X}^\nu_0+yp_l} w_{0l}(x)dx - \int_0^{\hat{X}^\nu_0+yp_r} w_{0r}(x)dx. \notag
\end{align}
Note that for $ i = l, r, $ since $ w_{0i} $ has zero average, $ \int_0^y w_{0i}(x) dx $ is periodic with respective to $ y $ with period $ p_i$. Therefore,
\begin{align*}
\int_0^1 \int_0^{\hat{X}^\nu_0+yp_i} w_{0i}(x)dx dy & = \frac{1}{p_i} \int_0^{p_i} \int_0^{\hat{X}^\nu_0+y} w_{0i}(x)dx dy  \\
& = \frac{1}{p_i} \int_0^{p_i} \int_0^{y} w_{0i}(x)dx dy.
\end{align*}
Thus one can get that
\begin{align}
\int_{0}^{1} J(y,t) dy = ~& O(e^{-\alpha t}) - M(X_0^\nu - \hat{X}^\nu_0) + (\ovlul-\ovlur) \hat{X}^\nu_0 \label{J}  \\
& - \int_{-\infty}^{0} \left( u_0-\phi^\nu-w_{0l} \right)(x)dx - \int_0^{+\infty} \left( u_0-\phi^\nu-w_{0r} \right)(x)dx \notag \\
& + \frac{1}{p_l} \int_0^{p_l} \int_0^{y} w_{0l}(x)dx dy  - \frac{1}{p_r} \int_0^{p_r} \int_0^{y} w_{0r}(x)dx dy . \notag
\end{align}
\vspace{0.5cm}

\textbf{(ii) The integrals on two sides.} 

\vspace{0.2cm}

Since $ u_l^\nu $ is periodic, it holds that
\begin{equation}\label{B-l-N}
\begin{aligned}
B_l^N(y,t) = & \int_{0}^{t} \Big\{ \left(f(u_l^\nu)- \nu \p_x u_l^\nu \right)(\tilde{X}^\nu(\tau)+yp_l,\tau)~ g^\nu\left((-N+y)p_l\right) \\
& \qquad -  (\tilde{X}^\nu)'(\tau)~ u_l^\nu(\tilde{X}^\nu(\tau)+yp_l,\tau) \\
& \qquad + \left[ \cdots \right] (1-g_{\tilde{X}^\nu}^\nu)(\tilde{X}^\nu(\tau)+(-N+y)p_l,\tau) 
\Big\} d\tau,
\end{aligned}
\end{equation}
where $ [\cdots] $ denotes the remaining terms which are bounded. Then taking the limit $N\rightarrow +\infty$ in \cref{B-l-N} and using \cref{Lem-g}, one can get that
\begin{align*}
\lim\limits_{N\rightarrow +\infty} \int_{0}^{1} B_l^N(y,t) dy & =  \int_{0}^{t} \frac{1}{p_l} \int_{0}^{p_l} f(u_l^\nu)(\tilde{X}^\nu + x,\tau) dxd\tau - \int_{0}^{t} (\tilde{X}^\nu)'(\tau) \ovlul d\tau  \\
& = \int_{0}^{t} \frac{1}{p_l} \int_{0}^{p_l} f(u_l^\nu)(x,\tau) dx d\tau - \ovlul (\tilde{X}^\nu(t) - \hat{X}_0^\nu).
\end{align*}
Similarly, it holds that
\begin{align*}
\lim\limits_{N\rightarrow +\infty} \int_{0}^{1} B_r^N(y,t) dy = \int_{0}^{t} \frac{1}{p_r} \int_{0}^{p_r} f(u_r^\nu)(x,\tau)  dx d\tau - \ovlur (\tilde{X}^\nu(t) - \hat{X}_0^\nu).
\end{align*}

\vspace{0.3cm}

Now, with the calculations in \textbf{(i)} and \textbf{(ii)}, one can integrate the equation \eqref{00} with respect to $y$ over $(0,1),$ and then let $ N \rightarrow + \infty, $ to get that for any $ t>T_0, $
\begin{equation}\label{equality-shift}
\begin{aligned}
& \int_{0}^{1} J(y,t) dy + (\ovlul-\ovlur)(\tilde{X}^\nu(t) - \hat{X}_0^\nu) \\
= & \int_{0}^{t} \Big[ \frac{1}{p_l} \int_{0}^{p_l} f(u_l^\nu)(x,\tau) dx - \frac{1}{p_r} \int_{0}^{p_r} f(u_r^\nu)(x,\tau) dx \Big] d\tau - M(X_0^\nu - \hat{X}^\nu_0).
\end{aligned}
\end{equation}
Note also that for $ i = l,r, $
\begin{align*}
\int_{0}^{t} \frac{1}{p_i} \int_{0}^{p_i} f(u_i^\nu) dy d\tau = & \int_{0}^{t} \frac{1}{p_i} \int_{0}^{p_i} \left( f(u_i^\nu)- f(\ovl{u}_i) \right)  dyd\tau + f(\ovl{u}_i) t \\
& = \int_{0}^{+\infty} \frac{1}{p_i} \int_{0}^{p_i} \left( f(u_i^\nu)- f(\ovl{u}_i) \right) dyd\tau + O(e^{-\alpha t}) + f(\ovl{u}_i) t,
\end{align*}
where \cref{Lem-periodic} is used. Then it can follow from \cref{J} and \cref{equality-shift} that for $ t>T_0, $
\begin{equation}\label{equ}
X^\nu(t)-st = \tilde{X}^\nu(t)-st = O(e^{-\alpha t}) + X_\infty^\nu,
\end{equation}
where $ X_\infty^\nu $ is defined in \cref{Thm-2-periodic}. The proof of \cref{Prop-shift} is finished.

\vspace{0.6cm}


\subsection{Decay to ansatz}\label{Sec-ansatz-proof}~\\

In this section, we will prove \cref{Prop-ansatz}. First, $ u^\nu(x,t) $ will be proved to be close to $ u^\nu_l(x,t) $ (resp., $ u^\nu_r(x,t) $) in the region $ x<st-N_{\e} $ (resp. $ x>st-N_{\e}$) with large enough $t$ and $N_{\e}$ (independent of time); see \cref{prop-unidecay}.
Then motivated by \cite{Oleinik1960}, the equation of the anti-derivative variable of $u^\nu-\psi_{X^\nu}^\nu$ is studied and the comparison principle is applied to prove \cref{Prop-ansatz}.

\vspace{0.3cm}

\subsubsection{Time-independent estimates}~\\

The following result follows directly from \cref{Lem-decay-ulr}.

\begin{Lem}\label{Lem-psi}
	Under the assumptions of \cref{Thm-2-periodic}, there holds that
	\begin{align*}
	|u^\nu(x,t)-\psi_\xi^\nu(x,t)| \leq C(t)e^{\beta |\xi|} e^{-\beta |x-\xi|},  \qquad \xi, x \in\R, t \geq 0,
	\end{align*}
	where $ C(t)>0$ is bounded on any compact subset of $ [0,+\infty). $
\end{Lem} 


Then by Lemmas \ref{Lem-decay-ulr} and \ref{Lem-psi}, for any $ x\in\R, t\geq 0, $ one can define the anti-derivative variables:
\begin{align}
U_l^\nu(x,t)& :=\int_{-\infty}^{x} (u^\nu-u_l^\nu)(y,t) dy, \label{def-Ul} \\
U_r^\nu(x,t)& :=\int_{x}^{+\infty} (u^\nu-u_r^\nu)(y,t) dy, \label{def-Ur} \\
\check{U}_\xi^\nu(x,t)& :=\int_{-\infty}^{x} (u^\nu-\psi_\xi^\nu)(y,t) dy,\label{def-U-check} \\
\hat{U}_\xi^\nu(x,t)& :=\int_{x}^{+\infty} (u^\nu-\psi_\xi^\nu)(y,t) dy. \label{def-U-hat}
\end{align}

\begin{Lem}\label{Lem-anti}
	For any $ C^\infty $ curve $ \xi(t):[0,+\infty) \rightarrow \R $ with bounded derivatives, the functions $ U_l^\nu, U_r^\nu, \check{U}_\xi^\nu $ and $ \hat{U}_\xi^\nu $ defined above satisfy the following equations:
	\begin{align}
	\p_t U_l^\nu - \nu \p_x^2 U_l^\nu & = f(u_l^\nu) - f(u^\nu), \qquad x\in\R, t>1, \label{U-l} \\
	\p_t U_r^\nu - \nu \p_x^2 U_r^\nu & = f(u^\nu) - f(u_r^\nu), \qquad x\in\R, t>1, \label{U-r} \\
	\p_t \check{U}_\xi^\nu - \nu \p_x^2 \check{U}_\xi^\nu & = f(\psi_\xi^\nu) - f(u^\nu) -\int_{-\infty}^{x} h_\xi^\nu(y,t) dy, \quad x\in\R, t>1, \label{U-check} \\
	\p_t \hat{U}_\xi^\nu - \nu \p_x^2 \hat{U}_\xi^\nu & = f(u^\nu) -f(\psi_\xi^\nu) -\int_{x}^{+\infty} h_\xi^\nu(y,t) dy, \quad x\in\R, t>1, \label{U-hat}
	\end{align}
	where the derivatives appearing in these equations are all continuous in $ \R\times [1,+\infty). $
\end{Lem}

\begin{proof}
	Here we prove only \eqref{U-l} and \eqref{U-check}, since the proofs of the other two are similar.
	
	\begin{itemize}
		\item[(1) ] To prove \eqref{U-l}, for any $ T>1, $ one considers the following problem:
		\begin{equation}\label{heat}
		\begin{cases}
		& \p_t V - \nu \p_x^2 V = f(u_l^\nu)-f(u^\nu),  \quad x\in \R, 1<t\leq T,\\
		& V(x,1) = U_l^\nu(x,1).
		\end{cases}
		\end{equation}
		It follows from \cref{Lem-psi} that
		\begin{equation}\label{inequ1}	
		\begin{cases}
		| f(u_l^\nu)-f(u^\nu) | \leq C(T) e^{\beta x}  &\forall x\in \R, 1<t\leq T,  \\
		| U_l^\nu(x,1) | \leq C e^{\beta x}   & \forall x\in \R.
		\end{cases}	
		\end{equation}
		Note that 
		$$e^{\beta |x|}\leq C(T) e^{\frac{1}{4\nu T}|x|^2}, $$
		and the initial data $ U_l^\nu(x,1) $ and the source term $ f(u_l^\nu)-f(u^\nu) $ are smooth functions, then by the standard parabolic theories (see \cite[Chapter~1,~Theorem~12]{Friedman1964}), the function
		\begin{equation}\label{V}
		V(x,t) = \int_{0}^{t-1} K^\tau(\cdot) * (f(u_l^\nu)-f(u^\nu))(\cdot,t-\tau) d\tau +  K^{t-1}(\cdot)*U_l^\nu(\cdot,1)
		\end{equation}
		solves \eqref{heat} and all the derivatives of $ V $ appearing in the equation exist and are continuous in $ \R\times[1,T], $ where $ K^t(x) := (2\pi\nu  t)^{-\frac{1}{2}} e^{-\frac{x^2}{4\nu t}} $ denotes the heat kernel. It then follows from \eqref{inequ1} and \eqref{V} that $ V(x,t) $ vanishes as $ x \rightarrow -\infty. $ Then by \eqref{V} and the equations of $ u^\nu $ and $ u_l^\nu, $ it holds that
		\begin{equation*}
		\begin{aligned}
		\p_x V(x,t) & = \int_{0}^{t-1} \left[ \p_x K^\tau * (f(u_l^\nu)-f(u)) \right](x,t-\tau) d\tau \\
		& \quad +\left[ K^{t-1}*(u-u_l^\nu) \right](x,1)	\\
		& = (u-u_l^\nu)(x,t), \qquad\qquad  x\in\R, 1\leq t \leq T,
		\end{aligned}	
		\end{equation*}
		which implies that $ V(x,t) =U_l^\nu(x,t) $ for all $ x\in \R $ and $ 1\leq t \leq T. $ And since $ T>1 $ is arbitrary, \eqref{U-l} holds true.
		
		\item[(2) ] Now we prove \eqref{U-check}. By integrating \eqref{source-1} with the space variable on $ (-\infty,x), $ one can get from Lemmas \ref{Lem-periodic} and \ref{Lem-g} that 
		\begin{equation*}
		\begin{aligned}
		\int_{-\infty}^{x} h_\xi^\nu~dy =  & \left(f(\psi_\xi^\nu)-f(u_l^\nu) \right) g_\xi^\nu + \left(f(\psi_\xi^\nu)-f(u_r^\nu)\right) (1-g_\xi^\nu) \\
		& - 2\nu (u_l^\nu-u_r^\nu)(g_\xi^\nu)' + \int_{-\infty}^{x} (f(u_l^\nu)-f(u_r^\nu)) (g_\xi^\nu)' ~dy \\
		& - \xi'(t) \int_{-\infty}^{x} (u_l^\nu-u_r^\nu) (g_\xi^\nu)'~  dy + \nu \int_{-\infty}^{x} (u_l^\nu-u_r^\nu) (g_\xi^\nu)'' ~dy,
		\end{aligned}
		\end{equation*}
		is smooth. And it follows from Lemmas \ref{Lem-g}, \ref{Lem-psi} and $ \psi_\xi^\nu - u_l^\nu = -(u_l^\nu-u_r^\nu)(1-g_\xi^\nu) $ that for any $ T>1, $ 
		\begin{equation*}
		\Big|\int_{-\infty}^{x} h_\xi^\nu~dy\Big| \leq C(T)\Big[ 1-g_\xi^\nu(x) + |(g_\xi^\nu)'(x)| + \int_{-\infty}^x |(g_\xi^\nu)'| dy +  \int_{-\infty}^x |(g_\xi^\nu)''| dy \Big].
		\end{equation*}
		Note that $ \int_{-\infty}^x |(g_\xi^\nu)'| dy = 1-g_\xi^\nu. $ And by $ (g^\nu)''= \frac{1}{\ovlul-\ovlur}(f'(\phi^\nu)- s)(\phi^\nu)', $ there exists a unique point $ x_0 \in \R $ such that $ (g^\nu)''<0 $ for $ x < x_0, $ and $ (g^\nu)''>0 $ for $ x > x_0. $ Then $ \int_{-\infty}^x |(g_\xi^\nu)''| dy = -(g^\nu_\xi)' $ if $ x<x_0+\xi. $ 
		Hence, it follows from \cref{Lem-g} that 
		\begin{equation}\label{source3}
		\Big|\int_{-\infty}^{x} h_\xi^\nu~dy\Big| \leq C(T) e^{\beta x }  \qquad \forall x\in \R, 1\leq t\leq T.
		\end{equation}
		By \cref{Lem-psi}, one can verify easily that
		\begin{equation}\label{initial}
		| \check{U}^\nu_{\xi(1)}(x,1) | \leq C e^{\beta x}  \qquad \quad \forall x\in \R.	
		\end{equation}
		Now we consider the problem:
		\begin{equation}\label{heat1}
		\begin{cases}
		\p_t \check{V} - \nu \p_x^2 \check{V} = \check{H}^\nu:= f(\psi_\xi^\nu)-f(u^\nu)-\int_{-\infty}^{x} h^\nu_\xi(y,t)dy,  &  x\in \R, 1<t\leq T, \\
		\check{V}(x,1) = \check{U}^\nu_{\xi(1)}(x,1), & x\in \R.
		\end{cases}
		\end{equation}
		By \cref{Lem-psi} and \eqref{source3}, $ \check{H}^\nu $ satisfies that
		\begin{equation}\label{inequ2}
		| \check{H}^\nu(x,t) | \leq C(T) e^{\beta x}  \qquad \forall x\in \R, 1<t\leq T. 
		\end{equation}
		Then similar to the proof in (1), since the initial data $ \check{U}^\nu_{\xi(1)}(x,1) $ is smooth and satisfies \eqref{initial}, and the source term $ \check{H}^\nu $ is smooth and satisfies \eqref{inequ2}, one can obtain that the function
		\begin{equation}\label{V-check}
		\check{V}(x,t) = \int_{0}^{t-1} K^\tau(\cdot) * \check{H}^\nu(\cdot,t-\tau) d\tau +  K^{t-1}(\cdot)*\check{U}^\nu_{\xi(1)}(\cdot,1)
		\end{equation}
		solves \eqref{heat1} and all the derivatives of $ \check{V} $ appearing in the equation exist and are continuous in $ \R \times [1,T]. $ It follows from \eqref{inequ2} and \eqref{V-check} that $ \check{V}(x,t) $ vanishes as $ x \rightarrow -\infty. $ Then \eqref{V-check} and the equation of $ \psi_\xi^\nu $ yield 
		\begin{align*}
		\p_t (u^\nu-\psi_\xi^\nu) - \nu \p_x^2 (u^\nu-\psi_\xi^\nu) & = \p_x[ f(\psi_\xi^\nu)- f(u^\nu)- \int_{-\infty}^{x} h_\xi^\nu(y,t)dy ] \\
		& = \p_x \check{H}^\nu.
		\end{align*}
		And then
		\begin{equation*}
		\begin{aligned}
		\p_x \check{V}(x,t) & = \int_{0}^{t-1}  \p_x K^\tau(\cdot) *  \check{H}^\nu(\cdot,t-\tau) d\tau +  K^{t-1}(\cdot)*(u^\nu-\psi_\xi^\nu)(\cdot,1) \\
		& = (u^\nu-\psi_\xi^\nu)(x,t)  \qquad\qquad \forall x\in\R, 1< t \leq T,
		\end{aligned}	
		\end{equation*}
		which implies that $ \check{V} = \check{U} _\xi^\nu. $
	\end{itemize}
\end{proof}

\vspace{0.3cm}

Now we choose a fixed small number $ \e_0>0 $ such that 
\begin{equation}\label{eps0-1}
f'(\ovlul-2\e_0) -s >0~ \text{ and } ~f'(\ovlur+2\e_0)-s<0.
\end{equation}
Then it can follow from \cref{Lem-periodic} that there exists $ T_1>T_0 $ large enough such that
\begin{equation}\label{eps0-2}
u_l^\nu(x,t) > \ovlul-\e_0~ \text{ and } ~u_r^\nu(x,t) < \ovlur +\e_0  \qquad \forall x\in \R, ~t>T_1,
\end{equation}
where $ T_0 $ is the number chosen in \cref{small}.

\vspace{0.2cm}

For later use, we define
\begin{align}
a(v,w)&:=\int_{0}^{1} f'(w+\rho(v-w))\ d\rho, \label{def-a} \\
b(v,w)&:=\int_{0}^{1} f''(w+\rho(v-w))\ d\rho. \label{def-b}
\end{align}

\vspace{0.2cm}

Now, we give the results of the time-independent estimates of $ u^\nu- u_l^\nu $ and $ u^\nu- u_r^\nu. $

\begin{Prop}\label{prop-unidecay}
	Under the assumptions of \cref{Thm-2-periodic}. There exists $ T_2 > T_1 $ such that for any $\e>0$, there exists $N(\e)>0, $ independent of time $ t, $ such that 
	\begin{align}
		&  |u^\nu(x,t) - u_l^\nu(x,t)| \leq \e, \qquad t>T_2,  x<st-N(\e), \label{unidecay-l} \\
		&  |u^\nu(x,t)- u_r^\nu(x,t)| \leq \e,  \qquad t>T_2,  x>st+N(\e). \label{unidecay-r}
	\end{align}
\end{Prop}

\begin{proof} 
	We prove \eqref{unidecay-l} only, since the proof of \eqref{unidecay-r} is similar. And the proof will be divided into four steps.
	
	\vspace{0.1cm}
	
	$ Step ~1. $ We will prove that there exist $ T_2 > T_1, $ independent of $ \e, $ and  $ N_1=N_1(\e)>0 $ such that
	\begin{equation*}
	U_l^\nu(x,t) < \e, \quad t>T_2,~ x<st-N_1,
	\end{equation*}
	where $ U_l^\nu $ is defined in \eqref{def-Ul}.
	\vspace{0.2cm}
	
	By \eqref{eps0-1}, one can define a constant
	\begin{equation}\label{beta-tilde}
	\tilde{\beta} :=  f'(\ovlul-2\e_0) -s  >0.
	\end{equation}
	For a constant $ M_1>1 $ to be determined later (see \cref{def-M}), we define the function:
	\begin{equation}\label{u-tilde}
	\tilde{u}(x,t) := M_1 e^{\tilde{\beta} (x-st)} + u_l^\nu(x,t) , \quad x\in\R,  t>0.
	\end{equation}
	Then by the equation of $ u_l^\nu, $ one has that
	\begin{equation}\label{equ-u-tilde}
	\p_t \tilde{u} - \nu \p_x^2 \tilde{u} + \p_x f(\tilde{u}) = \tilde{h} := -\tilde{\beta} s M_1 e^{\tilde{\beta} (x-st)} - \nu \tilde{\beta}^2 M_1 e^{\tilde{\beta} (x-st)} + \p_x \left( f(\tilde{u})- f(u_l^\nu)\right).
	\end{equation}
    	Since $ f $ is strictly convex, \eqref{eps0-2} implies that
    	$$ f'(\tilde{u}) \geq f'(u_l^\nu) \geq f'(\ovlul-\e_0) \quad \forall x\in\R, t\geq T_1. $$
    	Hence, for the given constant $M_1,$ if $ t\geq T_1,$ the following two inequalities hold.
    	\begin{enumerate}
    	\item[(1) ] 
    		If $x$ satisfies $M e^{\tilde{\beta} (x-st)}\geq 1$, then it holds that 
    		\begin{align*}
		\p_x \left( f(\tilde{u})-f(u_l^\nu) \right) & = \left( f'(\tilde{u})-f'(u_l^\nu) \right) \left( \tilde{\beta}M_1 e^{\tilde{\beta}(x-st)}+\p_x u_l^\nu \right) + f'(u_l^\nu) \tilde{\beta}M_1 e^{\tilde{\beta} (x-st)}  \\
    	    	& \geq \left( f'(\tilde{u})-f'(u_l^\nu) \right) \left( \tilde{\beta} -|\p_x u_l^\nu|\right) + f'(\ovlul-\e_0) \tilde{\beta}M_1 e^{\tilde{\beta} (x-st)}.
    	    	\end{align*}
    	    	Then by \cref{equ-u-tilde}, \cref{Lem-g} and $ \nu\leq 1, $ one has that
    	    	\begin{align*}
    	    	\tilde{h} \geq  &\left( f'(\tilde{u})-f'(u_l^\nu) \right) \left( \tilde{\beta} -|\p_x u_l^\nu|\right) + \tilde{\beta} M_1 e^{\tilde{\beta} (x-st)} \left( f'(\ovlul-\e_0)-(s+\tilde{\beta}) \right)  \\
    	    	\geq &\left( f'(\tilde{u})-f'(u_l^\nu) \right) \left( \tilde{\beta} -C e^{-\alpha t} \right) + \tilde{\beta}M_1 e^{\tilde{\beta} (x-st)} \left( f'(\ovlul-\e_0)-f'(\ovlul-2 \e_0)\right) \\
    	    	\geq & \left( f'(\tilde{u})-f'(u_l^\nu) \right) \left( \tilde{\beta} -C e^{-\alpha t} \right).
    	    	\end{align*} 
    	    	
    	    	\item[(2) ] If $x$ satisfies $M_1 e^{\tilde{\beta} (x-st)}<1$, one has $0< \tilde{u}-u_l^\nu \leq 1$. Then it holds that
    	    	\begin{align*}
    	    	\p_x \left( f(\tilde{u})-f(u_l^\nu) \right) & = f'(\tilde{u}) \tilde{\beta}M e^{\tilde{\beta} (x-st)} + \left( f'(\tilde{u})-f'(u_l^\nu) \right) \p_x u_l^\nu \\
    	    	& = f'(\tilde{u}) \tilde{\beta}M_1 e^{\tilde{\beta} (x-st)} + f''(\cdot) (\tilde{u} - u_l^\nu) \p_x u_l^\nu \\
    	    	&\geq f'(\ovlul-\e_0) \tilde{\beta}M_1 e^{\tilde{\beta} (x-st)} -C  M_1 e^{\tilde{\beta} (x-st)} |\p_x u_l^\nu|.
    	    	\end{align*}
    	     This yields that
    	    	\begin{align*}
    	    	\tilde{h} \geq & ~\tilde{\beta} M_1 e^{\tilde{\beta} (x-st)} \left( f'(\ovlul-\e_0)-(s+\tilde{\beta})\right) -  CM_1 e^{\tilde{\beta} (x-st)} |\p_x u_l^\nu| 	\\
    	    	\geq &~ M_1 e^{\tilde{\beta} (x-st)} \left[ \tilde{\beta} \left( f'(\ovlul-\e_0)-(s+\tilde{\beta}) \right) - C e^{-\alpha t} \right] \\
    	    	\geq &~ M_1 e^{\tilde{\beta} (x-st)} \left[ \tilde{\beta} \left( f'(\ovlul-\e_0)-f'(\ovlul-2 \e_0)\right) - C e^{-\alpha t} \right] \\
    	    	\geq & ~ M_1 e^{\tilde{\beta} (x-st)} \left( \tilde{\beta} c_0 \e_0 - C e^{-\alpha t} \right), 
    	    	\end{align*}
    	    	where $ c_0:=\min f'' >0. $ 
    	\end{enumerate}
    As a result of (1) and (2),
    there exists a constant $ T_2 > T_1, $ \textbf{independent of either $ M_1 $ or $ \e, $}  such that
	\begin{equation*}
	\tilde{h} = \p_t \tilde{u} - \nu \p_x^2 \tilde{u} + \p_x f(\tilde{u}) >0, \qquad x\in \R,  t\geq T_2.
	\end{equation*}
	Set $ \tilde{U}(x,t) := \int_{-\infty}^{x} (u^\nu-\tilde{u})(y,t) dy. $ Then as for the proof of Lemma \ref{Lem-anti}, one can prove that
	\begin{equation*}
	\p_t \tilde{U} - \nu \p_x^2 \tilde{U} + \left( f(u^\nu)-f(\tilde{u})\right) = -\int_{-\infty}^{x} \tilde{h}(y,t) dy, \qquad x\in \R,  t>T_2,
	\end{equation*}
	which implies that
	\begin{equation}\label{max}
	\p_t \tilde{U} - \nu \p_x^2 \tilde{U} + a(u^\nu,\tilde{u}) \p_x \tilde{U} < 0  \qquad \forall x\in \R,  t>T_2,
	\end{equation}
	where $ a(u,v) $ is defined by \eqref{def-a}.
	
	By \cref{Lem-psi}, there exists $ N_1>0 $ such that
	\begin{equation}\label{step2-1}
	|U_l^\nu(x,T_2)| = \Big| \int_{-\infty}^{x} (u^\nu-u_l^\nu)(y,T_2) dy \Big| \leq \e  \quad \forall x \leq -N_1.
	\end{equation}
	For this $ N_1, $ it can also follow from \cref{Lem-psi} that one can define a constant
	\begin{equation}\label{step2-2}
	B := \sup_{x \geq - N_1} \Big| \int_{-N_1}^{x} (u^\nu-u_r^\nu)(y,T_2) dy \Big| < +\infty.	
	\end{equation}
	Thus, for the given constants $ \tilde{\beta}, T_2, N_1 $ and $ B, $ one can choose $ M_1 >0 $ large enough such that
	\begin{equation}\label{def-M}
	\int_{-\infty}^{-N_1} (u_l^\nu - \tilde{u})(y,T_2) dy = - \frac{M_1}{\tilde{\beta}} e^{-\tilde{\beta} (N_1+sT_2)} < - B.
	\end{equation}
	Now we claim that
	\vspace{0.2cm}
	
	$ Claim~1. $ $$ \tilde{U}(x,T_2) = \int_{-\infty}^{x} (u^\nu-\tilde{u})(y,T_2) dy \leq \e  \quad \forall x\in\R. $$
	
	\vspace{0.2cm}

	In fact, if $ x\leq-N_1, $ then \eqref{step2-1} implies that
	\begin{equation*}
	\int_{-\infty}^{x} (u^\nu-\tilde{u})(y,T_2) dy \leq \int_{-\infty}^{x} (u^\nu-u_l^\nu)(y,T_2) dy \leq \e.
	\end{equation*}
	And if $ x>-N_1, $ it follows from \eqref{step2-1} and \eqref{step2-2} that
	\begin{align*}
	\int_{-\infty}^{x} (u^\nu-\tilde{u})(y,T_2) dy = & \int_{-\infty}^{-N_1} (u^\nu-u_l^\nu) + \int_{-\infty}^{-N_1} (u_l^\nu-\tilde{u}) \\
	& + \int_{-N_1}^{x} (u^\nu-u_r^\nu) + \int_{-N_1}^{x} (u_r^\nu-\tilde{u})  \\
	\leq &~ \e - B + B + 0  = \e,
	\end{align*}
	which proves $ Claim~1. $
	
	\vspace{0.2cm}
	
	Combing \eqref{max} with $ Claim~1 $, and using the maximum principle \cite[Lemma~1]{Oleinik1960},  one can obtain that
	\begin{equation*}
	\tilde{U}(x,t) \leq \e  \qquad \forall  x\in \R,  t\geq T_2,
	\end{equation*}
	which implies that
	\begin{equation*}
	U_l^\nu(x,t) = \int_{-\infty}^{x} (u^\nu-u_l^\nu)(y,t)dy = \tilde{U}(x,t) + \frac{M_1}{\tilde{\beta}} e^{\tilde{\beta} (x-st)} \leq 2\e,
	\end{equation*}
	if $ t\geq T_2 $ and $ x-st < -N_1 $ with $ N_1>0 $ large enough.
	
	\vspace{0.5cm}
	
	$ Step ~2. $ In this step, it is aimed to construct a ``sub-solution'' $ \psi_{\xi}^\nu $ to \cref{CL} such that the anti-derivative variable of $ u^\nu-\psi_{\xi}^\nu $ has the lower bound $ -\e $ (see \cref{lower-bound}). The idea is to find a $ C^\infty $ curve 
	$\xi(t): [T_2, +\infty) \rightarrow \R $
	such that the ansatz $ \psi_{\xi}^\nu $ defined by \eqref{def-psi} satisfies that
	\begin{equation}\label{ineq5}
	h_\xi^\nu= \p_t \psi_\xi^\nu + \p_x f(\psi_\xi^\nu) - \nu \p_x^2 \psi_\xi^\nu < 0  \qquad \forall x\in \R,  t>T_2,
	\end{equation}
	with the proper initial data \cref{step2-4}. 
	\vspace{0.1cm}
	
	For two given constants $ M_2>0 $ and $ d>0$, where $ M_2 $ will be determined in this step, and $ d $ will be determined in the next step, we define a $ C^\infty $ curves $\xi(t): [T_2, +\infty) \rightarrow \R, $ which solves the following problem:
	\begin{equation}\label{def-xi}
	\begin{cases}
	& \xi'(t) = s - M_2 e^{-\alpha t},  \qquad t>T_2, \\
	& \xi(T_2) =-d.
	\end{cases}
	\end{equation}
	Then by \eqref{def-xi}, there exists a constant $ D>d, $ depending on $\alpha, T_2,  d$ and $ M_2,$ such that
	\begin{equation}\label{xibdd}
	\xi(t) \in [st-D, st+D]  \quad \forall ~t \geq T_2.
	\end{equation} 
	Now we calculate the source term $ h_\xi^\nu. $	
	It follows from \eqref{source-2} and \eqref{small} that for $ t\geq T_2, $	
	\begin{align*}
	h_\xi^\nu= & ~(u_l^\nu-u_r^\nu) \left( \xi'(t)-s + f'(\phi_\xi^\nu) - f'(\psi_\xi^\nu) + 2\nu \dfrac{\p_x u_l^\nu-\p_x u_r^\nu}{u_l^\nu-u_r^\nu} \right)\lvert (g_\xi^\nu)' \rvert \\
	& -(u_l^\nu-u_r^\nu) \left( b(\psi_\xi^\nu, u_l^\nu)\p_x u_l^\nu - b(\psi_\xi^\nu, u_r^\nu)\p_x u_r^\nu \right) g_\xi^\nu (1-g_{\xi}^\nu),
	\end{align*}
	where $ b(u,v) $ is defined in \eqref{def-b}.
	Then by \cref{Lem-periodic} and \eqref{g1}, one has that for $ t\geq T_2, $
	\begin{align*}
	h_\xi^\nu
	\leq & ~(u_l^\nu-u_r^\nu) \left( \xi'(t)-s + f'(\phi_{\xi}^\nu) - f'(\psi_\xi^\nu) + 2\nu \dfrac{\p_x u_l^\nu-\p_x u_r^\nu}{u_l^\nu-u_r^\nu} \right) \lvert (g_\xi^\nu)'\rvert \\
	& + C (u_l^\nu-u_r^\nu) \left(\lvert \p_x u_r^\nu \rvert + \lvert \p_x u_l^\nu \rvert\right) \lvert (g_\xi^\nu)'\rvert \\
	\leq & ~(u_l^\nu-u_r^\nu)~\lvert (g_\xi^\nu)'\rvert~ \Big[-M_2 e^{-\alpha t} + C \lvert \phi_{\xi}^\nu - \psi_\xi^\nu \rvert +C (\lvert \p_x u_r^\nu \rvert + \lvert \p_x u_l^\nu \rvert)\Big].
	\end{align*}		
	Note that $$ | \phi_{\xi}^\nu - \psi_{\xi}^\nu | \leq | \ovlul - u_l^\nu | g_{\xi}^\nu + | \ovlur - u_r^\nu |( 1- g_{\xi}^\nu) \leq C e^{-\alpha t}  \qquad \forall x\in\R. $$ 
	Then if $M>0 $ is large enough, it holds that
	\begin{align*}
	h_\xi^\nu \leq (u_l^\nu-u_r^\nu)~ \lvert (g_\xi^\nu)'\rvert~  (-M_2+C) e^{-\alpha t} < 0  \qquad \forall x\in \R,  t>T_2.
	\end{align*}		
	Therefore, for $ M_2>0 $ large enough, \eqref{ineq5} is fulfilled with the $ \xi $ constructed in \eqref{def-xi}.
	
	\vspace{0.5cm}
	
	$ Step ~3. $ In this step, we will prove that there exists $ N_2=N_2(\e)>0 $ such that
	\begin{equation}\label{uni-Ul}
	U_l^\nu(x,t) > -\e,  \quad t>T_2,~x<st-N_2.
	\end{equation}
	\vspace{0.1cm}
	
	For the constants $ N_1 $ and $ B $ defined in \eqref{step2-1} and \eqref{step2-2}, one can choose the constant $ d=\xi(T_2)>0 $ (which is in \eqref{def-xi}) large enough such that 
	\begin{equation}\label{step2-3}
	\begin{aligned}
	& \int_{-\infty}^{-N_1} (u_l^\nu-\psi_{-d}^\nu)(x,T_2) dx + \int_{-N_1}^{+\infty} (u_r^\nu -\psi_{-d}^\nu)(x,T_2) dx \\
	= & ~\int_{-\infty}^{-N_1} (u_l^\nu-u_r^\nu)(x,T_2)(1-g^\nu(x+d)) dx - \int_{-N_1}^{+\infty} (u_l^\nu-u_r^\nu)(x,T_2)~g^\nu(x+d) dx \\
	> & ~ B.	
	\end{aligned}
	\end{equation}
	Then we claim that
	\vspace{0.1cm}
	
	$ Claim ~2. $ 
	\begin{equation}\label{step2-4}
	\check{U}^\nu_{\xi(T_2)}(x,T_2) = \int_{-\infty}^{x} (u^\nu-\psi^\nu_{-d})(y,T_2) dy \geq -\e  \quad \forall x\in \R.
	\end{equation}
	where $ \check{U}_\xi^\nu $ is defined in \eqref{def-U-check}.
	
	\vspace{0.2cm}
	
	In fact, if $ x \leq -N_1, $ then it follows from \eqref{step2-1} that
	$$ \int_{-\infty}^{x} (u^\nu - \psi^\nu_{-d})(y,T_2) dy \geq \int_{-\infty}^{x} (u^\nu -u_l^\nu)(y,T_2) dy \geq -\e. $$
	And if $ x > -N_1, $ then by \eqref{step2-1}, \eqref{step2-2} and \eqref{step2-3}, one can get that
	\begin{align*}
	\int_{-\infty}^{x} (u^\nu-\psi^\nu_{-d}) & = \int_{-\infty}^{-N_1} (u^\nu-u_l^\nu)+ \int_{-\infty}^{-N_1}(u_l^\nu-\psi^\nu_{-d}) + \int_{-N_1}^x (u^\nu-u_r^\nu) + \int_{-N_1}^x (u_r^\nu-\psi^\nu_{-d}) \\
	& \geq -\e + \int_{-\infty}^{-N_1}(u_l^\nu-\psi^\nu_{-d}) - B + \int_{-N_1}^{+\infty} (u_r^\nu -\psi^\nu_{-d}) \\
	& \geq -\e.
	\end{align*}
	Thus $ Claim ~2 $ is proved.
	
	\vspace{0.1cm}

	It can follow from \cref{Lem-anti} that
	\begin{equation*}
	\p_t \check{U}_\xi^\nu - \nu \p_x^2 \check{U}_\xi^\nu = f(\psi_\xi^\nu) - f(u^\nu) -\int_{-\infty}^{x} h_\xi^\nu(y,t) dy  \quad \forall x\in\R, ~t\geq T_2.
	\end{equation*}
	This and \eqref{ineq5} yield that
	\begin{equation}\label{max1}
	\p_t \check{U}_\xi^\nu - \nu \p_x^2 \check{U}_\xi^\nu + a(u^\nu,\psi_\xi^\nu) \p_x \check{U}_\xi^\nu >0  \quad \forall x\in\R, ~t\geq T_2,
	\end{equation}
	where $ a $ is defined by \eqref{def-a}. Then by \eqref{step2-4}, \eqref{max1}, and using the maximum principle, one can obtain
	\begin{equation}\label{lower-bound}
	\check{U}_\xi^\nu(x,t) \geq -\e  \quad \forall x\in\R, ~t\geq T_2.
	\end{equation}
	This, together with \cref{small} and \eqref{xibdd}, shows that
	\begin{align*}
	U_l^\nu(x,t) & =\int_{-\infty}^{x} (u^\nu - u_l^\nu)(y,t)dy \\
	& = \int_{-\infty}^{x} (u^\nu -\psi_\xi^\nu)(y,t)dy + \int_{-\infty}^{x} (\psi_\xi^\nu -u_l^\nu)(y,t)dy \\
	& = \check{U}_\xi^\nu(x,t) - \int_{-\infty}^{x} (u_l^\nu-u_r^\nu)(y,t)[1-g^\nu(y-\xi(t))] dy \\
	& \geq -\e - 2(\ovlul-\ovlur) \int_{-\infty}^{x} [1-g^\nu(y-st+D)] dy \\
	& \geq -2\e,
	\end{align*}
	provided that $ t\geq T_2 $ and $ x-st<-N_2 $ with $ N_2>0 $ large enough.
	
	\vspace{0.5cm}
	
	$ Step ~4. $ In the last step, we complete the proof of \eqref{unidecay-l}.
	\vspace{0.2cm}
	
	By $ Steps ~1 $ and $ 3, $ for any $ t > T_2 $ and $ x_1<x_2<st-N_2, $ one has that
	\begin{equation}\label{integral}
	\Big| \int_{x_1}^{x_2} (u^\nu-u_l^\nu)(y,t) dy \Big| \leq 2\e. 
	\end{equation}
	And by \eqref{entropy} and \cref{Lem-g}, there exists a positive number $ M_3>0 $ such that
	\begin{equation}\label{M1}
	\p_x (u^\nu -u_l^\nu) \leq M_3  \quad \forall ~t>T_2, ~ x\in \R.
	\end{equation}
	Then \eqref{unidecay-l} follows from the following Claim.
	\vspace{0.1cm}
	
	$ Claim ~3. $ For any $ t>T_2,~ x < st-N_2- \frac{3\sqrt{\e}}{\sqrt{M_3}}, $ it holds that
	\begin{equation}\label{uni-l}
	|u^\nu(x,t) - u_l^\nu(x,t)| \leq 3 \sqrt{M_3 \e}.
	\end{equation}	
	
	\vspace{0.3cm}
	
	In fact, if there exist $ t_0>T_2$ and $ x_0 < st_0 -N_2-\frac{3\sqrt{\e}}{\sqrt{M_3}} $ such that
	$$ u^\nu(x_0,t_0) - u_l^\nu(x_0,t_0) < -3 \sqrt{M_3 \e}, $$
	then for any $ x \in (x_0, x_0+\frac{3\sqrt{\e}}{\sqrt{M_3}} ), $ it holds that
	$$ (u^\nu-u_l^\nu)(x,t_0) - (u^\nu-u_l^\nu)(x_0,t_0) = \p_x (u^\nu-u_l^\nu)(\cdot,t_0)(x-x_0). $$
	Due to \eqref{M1}, it follows from above that
	\begin{equation*}
	(u^\nu-u_l^\nu)(x,t_0) < - 3 \sqrt{M_3 \e} + M_3(x-x_0).
	\end{equation*}
	Then integrating this inequality over $ (x_0, x_0+\frac{3\sqrt{\e}}{\sqrt{M_3}}) $ yields that
	\begin{align*}
	\int_{x_0}^{x_0+\frac{3\sqrt{\e}}{\sqrt{M_3}}} (u^\nu-u_l^\nu)(x,t_0) dx  \leq - 3 \sqrt{M_3 \e} \times \frac{3\sqrt{\e}}{\sqrt{M_3}} + \frac{M_3}{2} \times \frac{9\e}{M_3} 
	= -\frac{9}{2}\e < -2\e,
	\end{align*}
	which contradicts \eqref{integral}.
	For the other case that $ u^\nu(x_0,t_0)-u_l^\nu(x_0,t_0)> 3 \sqrt{M_3 \e}, $ it is also a contradiction by considering the interval $(x_0-\frac{3\sqrt{\e}}{\sqrt{M_3}}, x_0)$ instead. So $ Claim ~3 $ holds true.
\end{proof}

\vspace{0.5cm}

\subsubsection{Anti-derivative variables}~\\

In this part, 
we consider the equation of the anti-derivative of $u^\nu-\psi_{X^\nu}^\nu$. It turns out that the error term in the equation for the anti-derivative variable decays exponentially both in space and in time; see \cref{prop-U}. 
And then the idea of Il'in and Ole\v{\i}nik \cite{Oleinik1960} is applied to proving \eqref{ineq-ansatz}, i.e. we
construct an auxiliary function ($ \Theta(x) $ constructed below), and then use the maximal principle.

\vspace{0.2cm}

We first define the anti-derivative variable of $ -h^\nu_{X^\nu} $ in \cref{psi-equ} as
\begin{align}
H^\nu(x,t):= - \int_{-\infty}^{x} h_{X^\nu(t)}^\nu(y,t)~dy,  \quad  t\geq T_0, x \in \R,. \label{def-H}
\end{align} 
Then due to \eqref{source-1}, one has that 
\begin{equation}\label{integral2}
\begin{aligned}
H^\nu = &  -\left(f(\psi_{X^\nu}^\nu)-f(u_l^\nu)\right) g_{X^\nu}^\nu - \left(f(\psi_{X^\nu}^\nu)-f(u_r^\nu)\right)(1-g_{X^\nu}^\nu) \\
& + 2\nu (u_l^\nu-u_r^\nu) (g_{X^\nu}^\nu)' - \int_{-\infty}^{x}(f(u_l^\nu)-f(u_r^\nu))(y,t) (g_{X^\nu}^\nu)'(y) dy \\
& + (X^\nu)' \int_{-\infty}^{x} (u_l^\nu-u_r^\nu)(y,t) (g_{X^\nu}^\nu)'(y) dy \\
& - \nu \int_{-\infty}^{x} (u_l^\nu-u_r^\nu)(y,t) (g_{X^\nu}^\nu)''(y) dy.
\end{aligned}
\end{equation}

On the other hand, by \cref{Lem-psi}, one can define the anti-derivative variable of $ u^\nu-\psi_{X^\nu}^\nu $ as:
\begin{equation}\label{def-U}
U^\nu(x,t) := \check{U}_{X^\nu}^\nu(x,t) = \int_{-\infty}^{x} (u^\nu-\psi_{X^\nu}^\nu)(y,t) dy,  \quad  t\geq 0,  x \in \R.
\end{equation}

For convenience, in the following part of this paper we define
\begin{equation}\label{def-a-ovl}
\ovl{a}(x,t) := a(u^\nu,\psi_{X^\nu}^\nu)(x,t).
\end{equation}

\begin{Prop}\label{prop-U}
	The functions $ H^\nu(x,t) $ and $ \ovl{a}(x,t) $ are smooth in $ \R\times [T_0,+\infty), $ and $ U^\nu(x,t) $ solves the equation
	\begin{equation}\label{equ-U2}
		\partial_t U^\nu - \nu \partial_x^2 U^\nu + \ovl{a} ~\partial_x U^\nu = H^\nu, \quad x \in \R,~t \geq T_0,
	\end{equation}
	where $ a(u,v) $ is defined by \eqref{def-a}, and all the derivatives of $ U^\nu $ appearing in \eqref{equ-U2} are continuous in $ \R\times [T_0,+\infty). $
	Moreover, $ H^\nu(x,t) $ satisfies
	\begin{equation}\label{ineq-H}
	\lvert H^\nu(x,t)\rvert \leq C_1 e^{-\alpha t}e^{-\beta \lvert x-X^\nu(t)\rvert}   \quad \forall x \in \R, ~t \geq T_0,
	\end{equation}
	and $ U^\nu(x,t) $ satisfies
	\begin{equation}\label{U-estimate}
		\lvert U^\nu(x,t)\rvert \leq C_2(t) e^{-\beta \lvert x-X^\nu(t)\rvert}  \quad \forall x \in \R, ~t \geq T_0,
	\end{equation}
	where $ C_1>0 $ is a constant, independent of time $ t, $ and $ C_2(t)>0 $ is bounded on any compact subset of $ [T_0,+\infty). $
\end{Prop}

\begin{proof}	
	The smoothness of $ H^\nu(x,t) $ and $ \ovl{a}(x,t) $ can be derived easily from \cref{Lem-g} and the smoothness of $ u^\nu, u_l^\nu, u_r^\nu $ and $ \psi_{X^\nu}^\nu. $ By Lemma \ref{Lem-anti}, $ U^\nu = \check{U}_{X^\nu}^\nu $ solves \eqref{equ-U2} and all the derivatives of $ U^\nu $ appearing in \eqref{equ-U2} are continuous in $ \R\times [T_0, +\infty). $
	
	Now we prove \eqref{ineq-H}. Note that there exists a unique $ x_0 \in \R $ such that, if $ x<x_0, $ $ (g^\nu)''(x)<0, $ and if $ x>x_0, $ $ (g^\nu)''(x)>0. $ Then for $ x<X^\nu(t)+x_0, $ it follows from \eqref{integral2} and \cref{Lem-g} that
	\begin{align*}
	H^\nu &(x,t) =  -\left(f(\psi_{X^\nu}^\nu)-f(u_l^\nu)\right)- \left(f(\ovlul)-f(\ovlur)+ O(e^{-\alpha t}) \right) (1-g_{X^\nu}^\nu) \notag \\
	& + 2\nu \left(\ovlul- \ovlur +O(e^{-\alpha t})\right) (g_{X^\nu}^\nu)' - \int_{-\infty}^{x} \left(f(\ovlul)-f(\ovlur) +O(e^{-\alpha t}) \right) (g_{X^\nu}^\nu)'(y) dy \\
	& + \left(s+O(e^{-\alpha t}) \right) \int_{-\infty}^{x} \left( \ovlul-\ovlur +O(e^{-\alpha t})\right) (g_{X^\nu}^\nu)'(y) dy \\
	& - \nu \int_{-\infty}^{x} \left( \ovlul-\ovlur +O(e^{-\alpha t})\right)  ~(g_{X^\nu}^\nu)''(y) dy.
	\end{align*}
	If $ x<X^\nu(t)+x_0, ~ \int_{-\infty}^{x} |(g^\nu_{X^\nu})''| = -\int_{-\infty}^{x} (g^\nu_{X^\nu})'' = -(g^\nu_{X^\nu})', $ and $ 1-g^\nu_{X^\nu}(x)+|(g_{X^\nu}^\nu)'(x)| =  O\left(e^{\beta (x-X^\nu(t))}\right), $ then  
	\begin{align*}
	H^\nu(x,t) =  & -\left(f(\psi_{X^\nu}^\nu)-f(u_l^\nu)\right) - \left(f(\ovlul)-f(\ovlur) \right) (1-g_{X^\nu}^\nu)+ 2\nu (\ovlul- \ovlur)(g_{X^\nu}^\nu)'   \\
	& + \left(f(\ovlul)-f(\ovlur)\right) (1-g_{X^\nu}^\nu) - s (\ovlul-\ovlur)(1-g_{X^\nu}^\nu) \\
	& - \nu (\ovlul-\ovlur) (g_{X^\nu}^\nu)' + O\left( e^{-\alpha t} e^{\beta (x-X^\nu(t))}\right)  \\
	= & -(f(\psi_{X^\nu}^\nu)-f(u_l^\nu)) + (f(\phi_{X^\nu}^\nu)- f(\ovlul)) \\
	= & -a(\psi_{X^\nu}^\nu, u_l^\nu) (\psi_{X^\nu}^\nu - u_l^\nu) + a(\phi_{X^\nu}^\nu,\ovlul) (\phi_{X^\nu}^\nu-\ovlul) \\
	= & ~a(\psi_{X^\nu}^\nu, u_l^\nu) (u_l^\nu-u_r^\nu)(1-g_{X^\nu}^\nu) - a(\phi_{X^\nu}^\nu,\ovlul) (\ovlul-\ovlur)(1-g_{X^\nu}^\nu) \\
	= & (\ovlul-\ovlur)(1-g_{X^\nu}^\nu) \left(a(\psi_{X^\nu}^\nu,u_l^\nu)- a(\phi_{X^\nu}^\nu, \ovlul)\right) + O\left( e^{-\alpha t} e^{\beta (x-X^\nu(t))}\right).
	\end{align*}	
	And it is easy to verify that
	$$
	a(\psi_{X^\nu}^\nu,\ovlul)- a(\phi_{X^\nu}^\nu, u_l^\nu) = O(e^{-\alpha t}),
	$$
	which implies that if $ x<X^\nu(t)+x_0, $ $ H^\nu(x,t) = O\left(e^{-\alpha t} e^{\beta (x-X^\nu(t))}\right). $ On the other hand, if $ x > X^\nu(t)+x_0, $ $ \int_{x}^{+\infty} |(g^\nu_{X^\nu})''| = \int_{x}^{+\infty} (g^\nu_{X^\nu})'' = -(g^\nu_{X^\nu})', $ and $ g^\nu_{X^\nu}(x)+|(g_{X^\nu}^\nu)'(x)| =  O\left( e^{-\beta (x-X^\nu(t))}\right). $ It follows from \eqref{ode-shift} that for any $ t>T_0$ and $x\in\R, $
	\begin{equation}\label{integral3}
	\begin{aligned}
	H^\nu(x,t) = & -(f(\psi_{X^\nu}^\nu)-f(u_l^\nu))~g_{X^\nu}^\nu - \left( f(\psi_{X^\nu}^\nu)-f(u_r^\nu)\right) (1-g_{X^\nu}^\nu) \\
	& + 2\nu (u_l^\nu-u_r^\nu)~(g_{X^\nu}^\nu)' + \int_{x}^{+\infty}(f(u_l^\nu)-f(u_r^\nu))(y,t) (g_{X^\nu}^\nu)'(y) dy \\
	& - (X^\nu)'(t) \int_{x}^{+\infty} (u_l^\nu-u_r^\nu)(y,t) ~(g_{X^\nu}^\nu)'(y) dy \\
	& + \nu \int_{x}^{+\infty} (u_l^\nu-u_r^\nu)(y,t) ~(g_{X^\nu}^\nu)''(y) dy.
	\end{aligned}
	\end{equation}

	Then by similar arguments as above, one can prove that if $ x<X^\nu(t)+x_0, $ $ H^\nu(x,t) = O\left( e^{-\alpha t} e^{-\beta (x-X^\nu(t))}\right). $ Hence, one can get \eqref{ineq-H}.
		
	Now we will prove \eqref{U-estimate}. Similar to the proof of \cref{def-X-hat-0}, one can apply \cref{ode-shift} and \cref{def-X-0} to obtain that
	$$ \int_{-\infty}^{+\infty} \left( u^\nu(x,t)-\psi^\nu_{X^\nu(t)}(x,t)\right) dx = 0  \quad \forall t\geq T_0.  $$
%
%
	Hence, for any $ x \in \R $ and $ t> 0, $ one has that
	\begin{equation}\label{equ-U}
		U^\nu(x,t) = \int_{-\infty}^{x} (u^\nu-\psi^\nu_{X^\nu})(y,t) dy = - \int_{x}^{+\infty} (u^\nu-\psi^\nu_{X^\nu})(y,t) dy.	
	\end{equation}
	Lemma \ref{Lem-shift} implies that $ |X^\nu(t)| \leq C+|s|t. $ Then by \cref{Lem-psi}, one has that
	\begin{align*}
		|u^\nu(x,t)-\psi^\nu_{X^\nu}(x,t)| \leq C(t) e^{\beta|X^\nu(t)|} e^{-\beta |x-X^\nu(t)|}
		\leq C_2(t) e^{-\beta |x-X^\nu(t)|},
	\end{align*}
	which, together with \eqref{equ-U}, yields \eqref{U-estimate}.
\end{proof}

\vspace{0.2cm}


Denote the positive constant $\e_1 :=\min \left\lbrace \frac{f'(\ovlul)-s}{2}, -\frac{f'(\ovlur)-s}{2} \right\rbrace >0. $

\begin{Lem}\label{est-b}
		There exist positive constants $ T_3>T_2 $ and $N_0,$ independent of time, such that
		\begin{align}	
		& \ovl{a}(x,t) - (X^\nu)'(t)>\e_1 >0, \qquad  x-X^\nu(t)<-N_0,~t>T_3, \label{positive}\\
		& \ovl{a}(x,t) -(X^\nu)'(t)<-\e_1 <0,  \quad x-X^\nu(t)>N_0,~t>T_3. \label{negative}
		\end{align}
\end{Lem}
\begin{proof} 
Here we prove \eqref{positive} only, since \eqref{negative} can be proved similarly.

By \eqref{def-a}, $ \ovl{a}(x,t) = \int_{0}^{1} f'\left( \psi_{X^\nu}^\nu  +\rho(u^\nu-\psi_{X^\nu}^\nu)\right)  d\rho, $ where 
\begin{equation*}
	\begin{aligned}
		&\left| ~\psi_{X^\nu}^\nu+\rho(u^\nu-\psi_{X^\nu}^\nu)-\ovlul~\right|  \\
		=~& |~u_l^\nu g_{X^\nu}^\nu+u_r^\nu (1-g_{X^\nu}^\nu)+\rho (u^\nu-u_l^\nu) + \rho (u_l^\nu-u_r^\nu)(1-g_{X^\nu}^\nu) -\ovlul~| \\
		=~& |~(u_l^\nu-\ovlul) -(1- \rho )(u_l^\nu-u_r^\nu)(1-g_{X^\nu}^\nu) +\rho(u^\nu-u_l^\nu)~|  \\
		\leq~ & |u_l^\nu-\ovlul|+2(\ovlul-\ovlur) (1-g_{X^\nu}^\nu)+|u^\nu-u_l^\nu|. 
	\end{aligned}
\end{equation*}
Lemma \ref{Lem-shift} shows that $ |X^\nu(t)-st|\leq C. $ Then by combining Lemmas \ref{Lem-periodic} and \ref{Lem-g} with Proposition \ref{prop-unidecay}, one can get that for any $\e >0,$ there exist $T(\e)>T_2$ and $ N(\e)>0$ such that
$$ |\psi_{X^\nu}^\nu+\rho(u^\nu -\psi_{X^\nu}^\nu)-\ovlul| < \e  \qquad \forall t>T(\e), ~x-X^\nu(t) <-N(\e), $$
and hence,
$$ |~f'(\psi_{X^\nu}^\nu+\rho(u^\nu -\psi_{X^\nu}^\nu))- f'(\ovlul)~| \leq C\e  \qquad \forall t>T(\e), ~x-X^\nu(t) <-N(\e). $$
Then it follows from Lemma \ref{Lem-shift} that 
\begin{align*}
	\ovl{a} - (X^\nu)' & > f'(\ovlul)-C\e-s +[s-(X^\nu)'(t)] \\
	& \geq f'(\ovlul)-C\e-s - Ce^{-\alpha t} \\
	& \geq \frac{f'(\ovlul)-s}{2}   \qquad \forall t>T_3, ~x-X^\nu(t) <-N_0,
\end{align*}
provided that $\e$ is small enough, and $N_0>N(\e) $ and $ T_3>T(\e)$ are large enough. 
\end{proof}
\vspace{0.5cm}

Define the linear operator $L$ as 
\begin{equation}\label{def-L}
L:=\partial_t  - \nu \partial_x^2 + \ovl{a}(x,t)\partial_x .
\end{equation}
Therefore, Proposition \ref{prop-U} shows that $L U^\nu = H^\nu$ on $ \{t\geq T_0 \}. $
For the constant $N_0$ given in Lemma \ref{est-b}, we can define a convex $C^2$ function $\theta$ on $\R$ and the auxiliary function $\Theta$ as \cite{Oleinik1960}:
\begin{align}
&\theta(x):=\begin{cases}
\cosh(\gamma x), &~~\lvert x\rvert \leq N_0,\\
\theta \in C^2,~ 0\leq \theta''\leq \gamma^2 \cosh(\gamma x), & N_0< \lvert x\rvert\leq N_0+1, \\
\text{linear function}, &\lvert x\rvert> N_0+1,
\end{cases} \label{def-theta}\\
& \Theta(x):=e^{-\delta \theta(x)}, \notag
\end{align}
where $\gamma$ and $\delta$ are two positive constants to be determined. 	 

\begin{Lem}\label{ineq-Theta}
	There exist positive constants $\gamma, \delta $ and $ \mu, $ independent of time $ t, $ such that the auxiliary function $ \Theta $ defined above satisfies that
	$$ L\left( \Theta(x-X^\nu(t))\right) \geq 2\mu \Theta\left( x-X^\nu(t)\right)   \qquad \forall x\in \R, ~t>T_3. $$
\end{Lem}

\begin{proof}	
	In the following we define $\zeta:=x-X^\nu(t)$ for simplicity. 
	
	By \eqref{def-L}, it holds that
	\begin{align*}
		L\left( \Theta(\zeta)\right) = & e^{-\delta \theta(\zeta)}\left[  \delta (X^\nu)'(t) \theta'(\zeta)  -\delta^2 \nu (\theta'(\zeta))^2+\delta \nu \theta''(\zeta)  -\delta \ovl{a}~ \theta'(\zeta) \right] \\
		=&\delta \Theta(\zeta) \Big[ \nu \theta''(\zeta)-\left( \ovl{a}(x,t)-(X^\nu)'(t)\right)\theta'(\zeta)-\delta \nu (\theta'(\zeta))^2 \Big] 
	\end{align*} 
	There are three cases to be considered.
	
	\vspace{0.2cm}

	(1) If $\lvert\zeta\rvert <N_0,$ $\theta(\zeta)= \cosh(\gamma \zeta)$.  Therefore, 
	\begin{align*}
	& \nu \theta''(\zeta)-\left( \ovl{a}(x,t)-(X^\nu)'(t)\right)\theta'(\zeta)-\delta \nu (\theta'(\zeta))^2\\
	= & \nu \gamma^2 \cosh(\gamma \zeta) -\left( \ovl{a}-(X^\nu)'(t) \right)\gamma \sinh(\gamma\zeta)-\nu \delta\gamma^2 \left(\sinh(\gamma\zeta)\right)^2\\
	= & \gamma \cosh(\gamma\zeta)\Big[ \nu \gamma \left(1-\delta \sinh(\gamma\zeta)\tanh(\gamma\zeta) \right) -\left( \ovl{a}-(X^\nu)'(t)\right) \tanh(\gamma\zeta)  \Big].
	\end{align*}	
	Since $ \ovl{a}-(X^\nu)' $ is bounded, one can first choose $\gamma>0$ large enough such that 
	$$ \lvert \ovl{a}-(X^\nu)' \rvert < \frac{\nu\gamma}{4} \qquad \forall x\in\R, t>0. $$
	For $\lvert \zeta\rvert<N_0, \lvert \sinh(\gamma \zeta)\rvert < e^{\gamma N_0}, $ then one can choose $\delta=\delta(\gamma,N_0) > 0$ small enough such that $$ 1-\delta e^{\gamma N_0} \geq \frac{1}{2}. $$
	Hence, by $\lvert \tanh(\gamma \zeta)\rvert \leq 1$ and $ \cosh(\gamma \zeta) \geq 1, $ it holds that for $\lvert\zeta\rvert <N_0$ and $t>T_3,$ 
	\begin{equation*}
	L(\Theta(\zeta))\geq \delta \Theta(\zeta) \cdot \gamma \cosh(\gamma \zeta) \cdot \frac{\nu\gamma}{4} \geq \frac{\delta \nu \gamma^2}{4}\Theta(\zeta).
	\end{equation*}
	
	\vspace{0.3cm}
	
	(2) If $\zeta>N_0,$ by $0\leq \theta''(\zeta)\leq\gamma^2 \cosh(\gamma \zeta), $ it holds that $0<k_1 \leq \theta'(\zeta) \leq k_2$, where  $k_1=\gamma \sinh (\gamma N_0)$ and $k_2=\gamma \sinh (\gamma (N_0+1)).$  
	It follows from the fact $\theta''\geq 0$ and Lemma \ref{est-b} that
	\begin{align*}
	& \nu \theta''(\zeta)-\left( \ovl{a}-(X^\nu)'(t)\right)\theta'(\zeta)-\delta \nu (\theta'(\zeta))^2\\
	\geq&-\left( \ovl{a}-(X^\nu)'\right)\theta'(\zeta)-\delta \nu (\theta'(\zeta))^2\\
	\geq &  \theta'(\zeta) \left( \e_1-\delta \nu \theta'(\zeta) \right)
	\end{align*}
	One can choose $ \delta= \delta(\gamma, N_0, \e_1)>0$ small enough such that for $ \zeta>N_0, $ 
	$$\delta \nu \theta'(\zeta)\leq \delta \nu k_2 \leq \frac{\e_1}{2}. $$ 
	Hence,
	\begin{equation}\label{ineq6}
	L(\Theta(\zeta))\geq \delta \Theta(\zeta) \cdot k_1 \cdot \frac{\e_1}{2} \geq \frac{\delta \e_1 k_1}{2} \Theta(\zeta).
	\end{equation}

	\vspace{0.3cm}
	
	(3) For the case $\zeta<-N_0,$ \eqref{ineq6} can be proved similarly.
	\vspace{0.3cm}

	Collecting (1), (2) and (3), one can prove the lemma by choosing $\gamma$ sufficiently large, $\delta$ sufficiently small and $\mu= \min\left\lbrace \frac{\delta \nu \gamma^2}{8}, \frac{\delta \e_1 k_1}{4}\right\rbrace. $ 
\end{proof}

\vspace{0.3cm}

\subsubsection{Proof of \cref{Prop-ansatz}}~\\

Set $$Z(x,t):=M_2e^{-\mu t} \Theta(x-X^\nu(t))\pm U^\nu(x,t), \quad x\in\R, t\geq 0, $$ 
where $M_2>0$ is a constant to be determined, and $\mu$ is the constant in Lemma \ref{ineq-Theta}, which can be actually chosen small enough with $0<\mu \leq \min\{1,\alpha\}. $ Due to \eqref{ineq-H} and Lemma \ref{ineq-Theta}, it holds that
\begin{align*}
	LZ& = -\mu M_2e^{-\mu t}\Theta(x-X^\nu(t))+M_2e^{-\mu t}L\Theta(x-X^\nu(t)) \pm H^\nu \\
	&\geq (-\mu+2\mu) M_2e^{-\mu t}\Theta(x-X^\nu(t)) - C_0 e^{-\alpha t}e^{-\beta \lvert x-X^\nu(t)\rvert} \\
	&\geq e^{-\mu t}\left(\mu M_2e^{- \delta\theta (x-X^\nu(t))}-C_0 e^{-\beta \lvert x-X^\nu(t)\rvert}\right).
\end{align*}
By \eqref{U-estimate}, one has
$$ Z(x,T_3)\geq M_2 e^{-\mu T_3} e^{-\delta \theta (x-X^\nu(T_3))}-C(T_3)e^{-\beta \lvert x-X^\nu(T_3)\rvert}. $$
We need to consider two cases:
\begin{itemize}
	\item[(1)] If $\lvert x-X^\nu(t) \rvert\leq N_0+1$, then 
	\begin{align*}
		\mu M_2e^{- \delta \theta(x-X^\nu(t))}-C_0 e^{-\beta \lvert x-X^\nu(t)\rvert} & \geq \mu M_2e^{-\delta\theta(N_0+1)}-C_0, \\
		 M_2 e^{-\mu T_3} e^{-\delta \theta (x-X^\nu(T_3))}-C(T_3)e^{-\beta \lvert x-X^\nu(T_3)\rvert} &\geq M_2 e^{-\mu T_3} e^{-\delta \theta (N_0+1)} - C(T_3)
	\end{align*}
	Therefore, choosing $M_2 > \max\left\lbrace \frac{C_0}{\mu} e^{\delta \theta(N_0+1)}, C(T_3) e^{\mu T_3} e^{\delta \theta (N_0+1)} \right\rbrace, $ one gets that $LZ > 0$ and $ Z(x,T_3) >0. $
		
	\item[(2)] If $\lvert x-X^\nu(t) \rvert> N_0+1$, then $\theta$ is linear and $$ \lvert\theta'(x-X^\nu(t)) \rvert <k_2=\gamma\sinh (\gamma(N_0+1)).$$ Therefore, by $ \theta(x) \leq \cosh(x) $ for $ |x|>N_0+1, $ it holds that $$\theta(x-X^\nu(t)) < k_2\lvert x-X^\nu(t)\rvert+\cosh(\gamma (N_0+1)).$$
	Then one can get that
	\begin{align*}
		& \mu M_2e^{- \delta \theta(x-X^\nu(t))}-C_0 e^{-\beta \lvert x-X^\nu(t)\rvert} \\
		& \qquad \geq \mu M_2e^{-\delta \cosh(\gamma (N_0+1))}e^{- \delta k_2 \lvert x-X^\nu(t)\rvert}-C_0 e^{-\beta \lvert x-X^\nu(t)\rvert},\\
		& M_2 e^{-\mu T_3} e^{-\delta \theta (x-X^\nu(T_3))}-C(T_3)e^{-\beta \lvert x-X^\nu(T_3)\rvert} \\
		& \qquad \geq M_2 e^{-\mu T_3} e^{-\delta \cosh(\gamma (N_0+1))} e^{- \delta k_2 \lvert x-X^\nu(T_3)\rvert}-C(T_3)e^{-\beta \lvert x-X^\nu(T_3)\rvert}
	\end{align*}
	Then by choosing $\delta$ small enough with $\delta k_2 \leq \beta,$ and $M_2$ large enough with $$ M_2 > \max\left\lbrace \frac{C_0}{\mu}e^{\delta \cosh(\gamma (N_0+1))}, C(T_3)e^{\mu T_3}e^{\delta \cosh(\gamma (N_0+1))}  \right\rbrace, $$ one can also obtain that $LZ > 0$ and $ Z(x,T_3) >0. $ 
\end{itemize}
By combining (1) with (2), one gets that if $\delta$ is small and $M_2$ is large, $LZ > 0$ and $ Z(x,T_3) > 0 $ for any $x\in \R $ and $t\geq T_3$. 	
Therefore, the maximum principle implies that $Z(x,t)\geq 0$ for any $x\in \R, t\geq T_3$, which yields that
$$\lvert U^\nu(x,t)\rvert \leq M_2e^{-\mu t}\Theta(x-X^\nu(t)) \leq M_2e^{-\mu t} \qquad \forall x\in \R, t\geq T_3.$$
Hence, by the definition \eqref{def-U} of $ U^\nu, $ one has that for any $ x_1 < x_2$ and $t\geq T_3, $ 
\begin{equation} 
    \label{est-U}
    \left| \int_{x_1}^{x_2} \left(u^\nu(y,t)-\psi_{X^\nu}^\nu(y,t)  \right) dy \right|=\lvert U^\nu(x_2,t)-U^\nu(x_1,t)\rvert\leq 2M_2e^{-\mu t}.
\end{equation}
By \cref{entropy} and Lemmas \ref{Lem-periodic} and \ref{Lem-g}, there exists a constant $ M_3>0, $ independent of time $ t, $ such that for any $ x\in \R $ and $ t \geq T_3, $
\begin{equation}\label{ineq7}
    \p_x( u^\nu-\psi^\nu_{X^\nu} ) = \p_x(u^\nu-u_l^\nu)g_{X^\nu}^\nu+\p_x(u^\nu-u_r^\nu)(1-g_{X^\nu}^\nu)-(u_l^\nu-u_r^\nu) (g^\nu_{X^\nu})'
     \leq M_3.
\end{equation}
It then follows from \eqref{est-U} and \eqref{ineq7} that the following claim holds true.

\vspace{0.3cm}
	
$ Claim ~5. $ 
$$ |u^\nu(x,t)-\psi_{X^\nu}^\nu(x,t)| \leq 3\sqrt{M_2M_3} e^{-\frac{\mu}{2}t} \qquad \forall x\in \R, t\geq T_3. $$
	
\vspace{0.3cm}
	
Indeed, if there exists $(x_0, t_0)$ with $x_0 \in \R $ and $ t_0 \geq T_3, $ such that 
$$u^\nu(x_0,t_0)-\psi_{X^\nu}^\nu(x_0,t_0) < -3\sqrt{M_2M_3}e^{-\frac{\mu}{2}t_0}. $$
Then for any $x\in(x_0, x_1),$ where $ x_1:=x_0 +3\sqrt{\frac{M_2}{M_3}}e^{-\frac{\mu}{2}t_0}, $ \eqref{ineq7} yields that $$
\left( u^\nu(x,t_0)-\psi_{X^\nu}^\nu(x,t_0) \right) - \left( u^\nu(x_0,t_0)-\psi_{X^\nu}^\nu(x_0,t_0) \right) \leq M_3(x-x_0). $$ 
Then 
\begin{align*}
	&\int_{x_0}^{x_1} \left( u^\nu(x,t_0)- \psi_{X^\nu}^\nu(x,t_0) \right) dx \\
	\qquad\leq  & ~\left( u^\nu(x_0,t_0) - \psi_{X^\nu}^\nu(x_0,t_0) \right)  (x_1-x_0)+\frac{M_3}{2} (x_1-x_0)^2\\
	\qquad \leq & -3\sqrt{M_2M_3}e^{-\frac{\mu}{2}t_0} \cdot3\sqrt{\frac{M_2}{M_3}}e^{-\frac{\mu}{2}t_0} +\frac{M_3}{2} \cdot 9 \frac{M_2}{M_3}e^{-\mu t_0}\\
	\qquad = & -\frac{9}{2} M_2 e^{-\mu t_0} < -2 M_2 e^{-\mu t_0},
\end{align*}
which contradicts \eqref{est-U}. In the other case for $u^\nu(x_0,t_0)-\psi_{X^\nu}^\nu(x_0,t_0)> 3\sqrt{M_2M_3}e^{-\frac{\mu}{2}t_0}$ at some point $(x_0,t_0)$, a contradiction can be obtained similarly by considering the interval $\left( x_0-3\sqrt{\frac{M_2}{M_3}}e^{-\frac{\mu}{2}t_0}, x_0\right) $ instead.

Therefore, the claim above is proved. This, together with the fact that $ u^\nu $ and $ \psi_{X^\nu}^\nu $ are both bounded, yields \eqref{ineq-ansatz}.

\vspace{0.5cm}


\subsection{More results for Burgers' equation}\label{Sec-Burgers}~\\

In this section, we prove the result (1) in \cref{Thm-1-periodic} and \cref{Prop-Burgers} for the Burgers' equation \eqref{CL}, where $ f(u)=u^2/2$ and the two periodic perturbations are the same.

First, under the assumptions of \cref{Thm-1-periodic}, one can use the Galilean transformation to verify that the periodic functions $ w_l^\nu $ and $ w_r^\nu $ defined in \cref{def-wlr} satisfy
\begin{equation}\label{Galilean}
w_l^\nu (x,t)=w_r^\nu \left(x-(\ovlul-\ovlur)t,t \right),  \qquad  x\in \R, t\geq 0,
\end{equation}
Therefore, it holds that for any $ t>0, $
\begin{equation}\label{equ2}
\begin{aligned}
&\int_{0}^{t}\int_{0}^{p} \Big\lbrace  [f(u_l^\nu)-f(\ovlul)]- [f(u_r^\nu) -f(\ovlur)]\Big\rbrace dx d\tau\\
=&\int_{0}^{t}\int_{0}^{p} \frac{1}{2}\Big\{ \big[2\ovlul+w_l^\nu(x,\tau)\big]w_l^\nu(x,\tau)-  \big[2\ovlur+w_r^\nu(x,\tau)\big]w_r^\nu(x,\tau)\Big\} dx d\tau\\
=&\int_{0}^{t}\int_{0}^{p} \frac{1}{2}\Big[ (w_l^\nu)^2(x,\tau)-(w_r^\nu)^2(x,\tau) \Big] dx d\tau\\
=&\int_{0}^{t}\int_{0}^{p} \frac{1}{2} (w_r^\nu)^2\big(x-(\ovlul-\ovlur)t,\tau\big) dx d\tau-\int_{0}^{t}\int_{0}^{p}  \frac{1}{2} (w_r^\nu)^2(x,\tau) dxd\tau\\
=&\int_{0}^{t}\int_{-(\ovlul-\ovlur)t}^{p-(\ovlul-\ovlur)t} \frac{1}{2} (w_r^\nu)^2(x,\tau) dx d\tau-\int_{0}^{t}\int_{0}^{p}\frac{1}{2}  (w_r^\nu)^2(x,\tau) dx d\tau=0,
\end{aligned}
\end{equation}
where the second equality holds since the averages of $w_l^\nu$ and $w_r^\nu$ are zero. Thus, letting $ t\rightarrow +\infty $ in \cref{equ2} shows that $ X_{\infty,2}^\nu $ defined in \cref{def-X-2} is identically zero.

\vspace{0.2cm}

Now it remains to prove \cref{Prop-Burgers}.
Under the assumptions of \cref{Prop-Burgers}, it holds that $ u_l^\nu(x,t) \geq u_r^\nu(x,t) $ for all $ x\in\R $ and $ t>0, $ thus
\begin{align*}
\int_\R (u_l^\nu(x,t) - u_r^\nu(x,t)) (g^\nu_\xi)'(x) dx  & = \int_0^p (u_l^\nu(x,t) - u_r^\nu(x,t)) \sum_{k \in \mathbb{Z}} (g^\nu_\xi)'(x+kp) dx \\
& \leq -C \int_0^p (u_l^\nu(x,t) - u_r^\nu(x,t)) \sum_{k \in \mathbb{Z}} e^{-\beta(x+kp)} dx \\
& \leq -C e^{-\beta p} \int_0^p (u_l^\nu(x,t) - u_r^\nu(x,t)) dx \\
& = -C e^{-\beta p} p (\ovlul - \ovlur) <0.
\end{align*}
Then the number $ M $ in \cref{add-M} can be chosen to be zero. Moreover, since $ M=0 $ and the initial data given in \cref{Prop-Burgers} satisfies $ \int_\R (u_0(x) - \psi^\nu_0(x)) dx = 0, $ 
the unique number $ \hat{X}_0^\nu $ satisfying \cref{def-X-hat-0} is zero.
Hence, one can get that the curve $ \tilde{X}^\nu $ defined in \cref{def-X-tilde} actually solves 
\begin{equation}\label{ode-shift-1}
	\begin{cases}
	(\tilde{X}^\nu)'(t) = F^\nu(\tilde{X}^\nu,t), &\quad t>0, \\
	\tilde{X}^\nu(0)=0. &\\
	\end{cases}
\end{equation}
Due to \cref{Galilean}, if $ (\ovlul-\ovlur) t_k = kp $ for any $ k\geq 0, $ one has $ w_l^\nu (x,t_k) \equiv w_r^\nu(x,t_k). $
Thus the term $ J^N $ defined in \cref{J-N} and the limit $ J $ satisfy that 
\begin{equation}\label{equ3}
J^N(y,t_k)\equiv 0 \quad \Rightarrow \quad J(y,t_k) = \lim\limits_{N\rightarrow+\infty} J^N(y,t_k) \equiv 0.
\end{equation}
Then taking \cref{equ2}, $ M=\hat{X}_0^\nu=0, $ and \cref{equ3} into \cref{equality-shift} implies that $ \tilde{X}^\nu(t_k) = st_k, $ and thus \cref{X-k} holds true. 

\vspace{0.2cm}

It remains to prove \eqref{burgers-sol}. 
In \cite{Hopf1950}, Hopf introduced the well known Hopf transformation to obtain an explicit formula for the solution to \eqref{CL} with any initial data $ u_0 \in L^\infty, $ which is given by:
\begin{equation}\label{hopf-formula}
u^\nu(x,t) = \dfrac{ \int_\R \frac{x-y}{t} \exp\left\lbrace -\frac{(x-y)^2}{4\nu t} - \frac{1}{2\nu} \int_0^y u_0(\eta) d\eta \right\rbrace dy }{\int_\R \exp\left\lbrace -\frac{(x-y)^2}{4\nu t} - \frac{1}{2\nu} \int_0^y u_0(\eta) d\eta \right\rbrace  dy},  \qquad  x \in \R,  t > 0.
\end{equation}
Since $ u_0 $ is bounded, then integration by parts on the numerator of \eqref{hopf-formula} yields that
\begin{equation}\label{hopf-formula*}
u^\nu(x,t) = \dfrac{ \int_\R u_0(y) \exp\left\lbrace -\frac{(x-y)^2}{4\nu t} - \frac{1}{2\nu} \int_0^y u_0(\eta) d\eta \right\rbrace dy }{\int_\R \exp\left\lbrace -\frac{(x-y)^2}{4\nu t} - \frac{1}{2\nu} \int_0^y u_0(\eta) d\eta \right\rbrace dy},  \qquad x \in \R,  t> 0.
\end{equation}

Without loss of generality (the viscous shock profile is unique up to a shift), the viscous shock $ \phi^\nu $ connecting the end states $ \ovlul $ at $ x=-\infty $ and $ \ovlur $ at $ x=+\infty $ to the Burgers' equation can be given by the explicit formula
$$ \phi^\nu(x) = \frac{\ovlul+\ovlur}{2} - \frac{\ovlul-\ovlur}{2}  \tanh(\frac{\ovlul-\ovlur}{4\nu} x). $$
Set $ \lambda := \frac{\ovlul-\ovlur}{4\nu} $ for convenience. Then the associated $ g^\nu $ defined in \eqref{def-g} is given by
\begin{equation}\label{burgers-g}
	g^\nu(x) = \dfrac{1-\tanh (\lambda x)}{2} = \dfrac{e^{-\lambda x}}{e^{\lambda x}+e^{-\lambda x}},
\end{equation}
satisfying
\begin{equation}\label{vshock-formula}
	\phi^\nu(x) = \ovlul g^\nu(x) + \ovlur (1-g^\nu(x)).
\end{equation} 
And one also has that
\begin{equation}\label{int-g}
	\begin{aligned}
		\int_{0}^{x} g^\nu(y) ~dy & =  \frac{1}{2} \int_{0}^{x} [1- \tanh(\lambda y)]  ~dy = \frac{1}{2\lambda} \log \frac{e^{\lambda x}}{e^{\lambda x}+ e^{-\lambda x}} + \frac{1}{2\lambda} \log 2 \\
		& =  \frac{1}{2\lambda} \log (1-g^\nu(x)) + \frac{1}{2\lambda} \log 2.
	\end{aligned}
\end{equation}
Similarly, 
\begin{equation}\label{int-1-g}
	\int_{0}^{x} [1-g^\nu(y)] ~dy =  -\frac{1}{2\lambda} \log g^\nu(x) - \frac{1}{2\lambda} \log 2.
\end{equation}
If the initial data $ u_0(x) = \phi^\nu(x) + w_0(x) $ with $ \int_{0}^{p} w_0(x) dx = 0, $ 
then it follows from \cref{hopf-formula*} and \cref{vshock-formula} that 
\begin{equation}\label{formula-u}
	u^\nu(x,t) = \dfrac{P_l(x,t)+P_r(x,t)}{Q_l(x,t) +Q_r(x,t)},
\end{equation}
where the two terms in the numerator are
\begin{equation*}
	\begin{aligned}
		P_l(x,t) &:= 2 \int_\R \left(\ovlul+w_0(y)\right) g^\nu(y) \exp\left\lbrace -\frac{(x-y)^2}{4\nu t} - \frac{1}{2\nu} \int_0^y u_0(\eta) d\eta \right\rbrace dy, \\
		P_r(x,t) &:= 2  \int_\R \left(\ovlur+w_0(y)\right) \left(1-g^\nu(y)\right) \exp\left\lbrace -\frac{(x-y)^2}{4\nu t} - \frac{1}{2\nu} \int_0^y u_0(\eta) d\eta \right\rbrace dy,
	\end{aligned}
\end{equation*}
and the two terms in the denominator are
\begin{equation*}
	\begin{aligned}
		Q_l(x,t) & := 2 \int_\R g^\nu(y)  \exp\left\lbrace -\frac{(x-y)^2}{4\nu t} - \frac{1}{2\nu} \int_0^y u_0(\eta) d\eta \right\rbrace dy, \\
		Q_r(x,t) & := 2 \int_\R \left(1-g^\nu(y)\right) \exp\left\lbrace -\frac{(x-y)^2}{4\nu t} - \frac{1}{2\nu} \int_0^y u_0(\eta) d\eta \right\rbrace dy. 
	\end{aligned}
\end{equation*}
It can follow from \eqref{int-1-g}, $ \lambda = \frac{\ovlul-\ovlur}{4\nu}, $ and $u_0=\ovlul g^\nu +\ovlur (1-g^\nu)+w_0 = \ovlul+w_0-(\ovlul-\ovlur)(1-g^\nu)$ that
\begin{align*}
	& g^\nu(y)\exp\left\lbrace -\frac{1}{2\nu} \int_0^y u_0(\eta) d\eta \right\rbrace \\
	= &~ g^\nu(y) \exp\left\lbrace 2 \lambda \int_{0}^{y} \left( 1-g^\nu(\eta)\right) d\eta \right\rbrace  \exp\left\lbrace - \frac{1}{2\nu} \int_0^y \left( \ovlul+w_0(\eta)\right) d\eta \right\rbrace  \\
	=&~ \frac{1}{2}  \exp\left\lbrace  - \frac{1}{2\nu} \int_0^y \left( \ovlul+w_0(\eta)\right) d\eta \right\rbrace.
\end{align*}
Due to \eqref{int-g}, similar calculations yield
\begin{align*}
	 \left( 1-g^\nu(y)\right) \exp\left\lbrace - \frac{1}{2\nu} \int_0^y u_0(\eta) d\eta \right\rbrace 
	=  \frac{1}{2} \exp\left\lbrace - \frac{1}{2\nu} \int_0^y \left( \ovlur+w_0(\eta) \right)  d\eta\right\rbrace .
\end{align*}
Hence, one has that
\begin{equation*}
\begin{aligned}
P_l(x,t) & = \int_\R \left( \ovlul+w_0(y) \right) \exp\left\lbrace -\frac{(x-y)^2}{4\nu t} - \frac{1}{2\nu} \int_0^y \left( \ovlul+w_0(\eta) \right)  d\eta \right\rbrace dy, \\
P_r(x,t) & = \int_\R \left( \ovlur+w_0(y) \right) \exp\left\lbrace -\frac{(x-y)^2}{4\nu t} - \frac{1}{2\nu} \int_0^y \left( \ovlur+w_0(\eta) \right)  d\eta \right\rbrace dy, \\
Q_l(x,t) & = \int_\R ~\exp\left\lbrace -\frac{(x-y)^2}{4\nu t} - \frac{1}{2\nu} \int_0^y \left( \ovlul+w_0(\eta) \right)   d\eta \right\rbrace dy, \\
Q_r(x,t) & = \int_\R \exp\left\lbrace -\frac{(x-y)^2}{4\nu t} - \frac{1}{2\nu} \int_0^y \left( \ovlur+w_0(\eta) \right)  d\eta \right\rbrace dy.
\end{aligned}
\end{equation*}
Moreover, by using the Hopf formula \eqref{hopf-formula*} for $ u_l^\nu(x,t) $ and $ u_r^\nu(x,t), $ respectively, one gets that 
\begin{equation}\label{formula-ulr}
	u_l^\nu(x,t)= \dfrac{P_l(x,t)}{Q_l(x,t)} \quad \text{ and } \quad u_r^\nu(x,t)= \dfrac{P_r(x,t)}{Q_r(x,t)}.
\end{equation}
If $ t=t_k = \frac{kp}{\ovlul-\ovlur}=\frac{kp}{4\nu\lambda}, $ it holds that
\begin{align*}
	Q_r(x,t_k) & = \int_\R \exp\left\lbrace  -\frac{(x-y)^2}{4\nu t_k} +\frac{1}{2\nu}\int_{0}^{y} 4\nu\lambda d\eta - \frac{1}{2\nu} \int_0^y \left( \ovlul+w_0(\eta) \right)   d\eta \right\rbrace dy \\
	& = e^{2\lambda x + 4\lambda^2 \nu t_k} \int_\R \exp\left\lbrace  -\frac{(y-x-4\lambda\nu t_k)^2}{4\nu t_k} - \frac{1}{2\nu} \int_0^y \left( \ovlul+w_0(\eta) \right)  d\eta \right\rbrace dy \\
	& = e^{2\lambda x + 4\lambda^2\nu t_k}  \int_\R \exp\left\lbrace  -\frac{(y-x)^2}{4\nu t_k} - \frac{1}{2\nu} \int_0^{y+4\lambda\nu t_k} \left( \ovlul+w_0(\eta) \right)  d\eta \right\rbrace dy\\
	&= e^{2\lambda x + 4\lambda^2 \nu t_k - 2\lambda \ovlul t_k}  \int_\R \exp\left\lbrace  -\frac{(y-x)^2}{4\nu t_k} - \frac{1}{2\nu} \int_0^{y} \left( \ovlul+w_0(\eta) \right) d\eta \right\rbrace \\
	& \qquad \qquad \qquad\qquad\qquad \times \exp \left\lbrace - \frac{1}{2\nu} \int_y^{y+4\lambda\nu t_k} w_0(\eta)  d\eta \right\rbrace dy.
\end{align*}
Then due to $ \int_{0}^{p} w_0(x) dx =0 $ and $ 4\lambda\nu t_k=kp, $ it holds that
\begin{equation*}
Q_r(x,t_k) = e^{2\lambda (x-st_k)}  Q_l(x,t_k), \qquad x \in \R.
\end{equation*}
Similarly, one can get that
\begin{equation*}
P_r(x,t_k) = e^{2\lambda (x-st_k)} P_l(x,t_k) - 4\lambda e^{2\lambda (x-st_k)} Q_l(x,t_k), \qquad x \in \R.
\end{equation*}
Hence, \eqref{formula-u} yields that
\begin{equation}\label{equ1}
	\begin{aligned}
		u^\nu(x,t_k) & = \dfrac{P_l(x,t_k) + e^{2\lambda (x-st_k)} P_l(x,t_k) - 4\lambda e^{2\lambda (x-st_k)}Q_l(x,t_k) }{Q_l(x,t_k)+ e^{2\lambda (x-st_k)} Q_l(x,t_k)} \\
		& = \dfrac{P_l(x,t_k)}{Q_l(x,t_k)} - 4\lambda \left(1-g^\nu(x-st_k)\right),
	\end{aligned}
\end{equation}
where $ g^\nu $ is defined in \eqref{burgers-g}.
Meanwhile, \eqref{formula-ulr} yields that
$$ u_l^\nu(x,t_k) = \dfrac{P_l(x,t_k)}{Q_l(x,t_k)}~ \text{ and } ~u_r^\nu(x,t_k) = \dfrac{P_l(x,t_k)-4\lambda~ Q_l(x,t_k)}{Q_l(x,t_k)} = \dfrac{P_l(x,t_k)}{Q_l(x,t_k)} - 4\lambda, $$
which, together with \eqref{equ1}, yields \eqref{burgers-sol}.

\vspace{0.5cm}


\subsection{An example of non-zero shift}\label{Sec-counter-ex}~\\
	
In this section, we prove the result (2) in \cref{Thm-1-periodic}, where the two periodic perturbations are assumed to be the same.
For any given periodic perturbation $w_0$ with zero average and $0<\|w_0\|_{L^{\infty}(\R)}< (\ovlul-\ovlur)/2$, it holds that $ \ovlur + \|w_0\|_{L^{\infty}(\R)} < \ovlul-\|w_0\|_{L^{\infty}(\R)}, $ then one can construct a smooth and strictly convex function $f$ such that $f(u)=\frac{1}{2n}u^2$ when $u\leq \ovlur+\|w_0\|_{L^{\infty}(\R)}$ and $f(u)=\frac{1}{2}u^2$ when $u\geq \ovlul-\|w_0\|_{L^{\infty}(\R)}$, where $n$ is a positive number to be determined later; see \cref{Fig-f}.
	\begin{figure}[htbp!]
		\scalebox{0.4}{\includegraphics{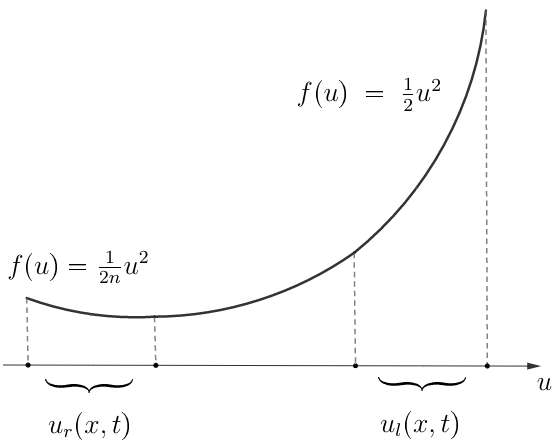}}  
		\caption{Construction of $ f $}\label{Fig-f}
	\end{figure}
	
Since for any $ x \in \R $ and $ t > 0, $ it holds that
\begin{equation*}
	\begin{aligned}
	& u_l^\nu(x,t) \geq \inf u_l^\nu(\cdot,0) \geq \ovlul-\|w_0\|_{L^{\infty}(\R)},  \\
	& u_r^\nu(x,t) \leq \sup u_r^\nu(\cdot,0) \leq \ovlur + \|w_0\|_{L^{\infty}(\R)}.
	\end{aligned}
\end{equation*}
And note that $ w_l^\nu=u_l^\nu-\ovlul $ and $ w_r^\nu=u_r^\nu-\ovlur $ have zero average, then one has that
\begin{equation}\label{numl}
	\begin{aligned}
	\int_{0}^{\infty}\int_{0}^{p} \left[f(u_l^\nu)-f(\ovlul)\right] dxd\tau & = \frac{1}{2}\int_{0}^{\infty}\int_{0}^{p} \left[ (u_l^\nu)^2 - \ovlul^2 \right] dxd\tau \\
	& = \frac{1}{2}\int_{0}^{\infty}\int_{0}^{p}\left[ (\ovlul+w_l^\nu)^2 - \ovlul^2 \right]  dxd\tau \\
	& = \frac{1}{2}\int_{0}^{\infty}\int_{0}^{p} (w_l^\nu)^2 \ dxd\tau.
	\end{aligned}
\end{equation}
Similarly,
\begin{equation}\label{num2}
\begin{aligned}	
\int_{0}^{\infty}\int_{0}^{p} \left[f(u_r^\nu)-f(\ovlur)\right] dxd\tau & = \frac{1}{2n}\int_{0}^{\infty}\int_{0}^{p} \left[  (u_r^\nu)^2-\ovlur^2 \right] dxd\tau \\
& = \frac{1}{2n}\int_{0}^{\infty}\int_{0}^{p} (w_r^\nu)^2 dxd\tau.
\end{aligned}
\end{equation}

Since $ w_0 $ is not identically zero, the solution $ u_l^\nu(x,t) $ with the initial data $\ovlul+w_0(x)$ cannot be a constant in $ \R \times [0,+\infty), $ thus the integral of \eqref{numl} is positive. And more importantly, this integral is independent of $ n, $ since no matter what $ n $ is, the range of $ u_l^\nu(x,t) $ is always in the interval where $ f(u) $ is $ u^2/2$, which means that $ u_l^\nu(x,t) $ is actually a solution to the Burgers' equation.
	
On the other side, for the solution $ u_r^\nu(x,t), $ \cref{app2} yields that
\begin{equation*}
	\int_{0}^{+\infty} \int_{0}^{p} (w_r^\nu)^2 dx d\tau \leq C,
\end{equation*}
where $ C $ is independent of $ f, $ depending only on $ \nu, p $ and $ \|u^\nu_r(x,0)\|_{L^\infty}. $ 
It follows from this and \eqref{num2} that
\begin{equation}\label{num3}
	\int_{0}^{\infty}\int_{0}^{p} [f(u_r^\nu)-f(\ovlur)]\ dxd\tau 
	\leq  \frac{C}{2n}. 
\end{equation} 
By \eqref{numl} and \eqref{num3}, one can choose $n$ sufficiently large such that
$$ \int_{0}^{\infty}\int_{0}^{p} [f(u_l^\nu)-f(\ovlul)]\ dxd\tau > \int_{0}^{\infty}\int_{0}^{p} [f(u_r^\nu)-f(\ovlur)]\ dxd\tau, $$
which implies that $ X_{\infty,2}^\nu \neq 0. $ 	
The proof of Theorem \ref{Thm-1-periodic} is finished.

\vspace{0.5cm}


\subsection{Vanishing viscosity limit for the shift}\label{Sec-vanishing}~\\

We now study the vanishing viscosity limit for $ X_{\infty,2}^\nu $ defined in \cref{def-X-2}.

\begin{Lem}\label{Lem-Xin}
	There exists a constant $ C>0, $
	independent of time $ t $ or viscosity $ \nu, $ such that for all $ 0\leq \nu \leq 1, $ it holds that
	\begin{equation}\label{uni-bdd}
		\sup_{x \in \R}|u_l^\nu - \ovlul| + \sup_{x \in \R}|u_r^\nu - \ovlur| \leq \frac{C}{1+t},  \quad t>0.
	\end{equation}	
\end{Lem}
\begin{proof}
	The proof can be found in Xin \cite{Xin}, which relies on the Ole\v{\i}nik's entropy condition \eqref{entropy}. For $ i = l $ or $ r, $ it follows from $ \int_{0}^{p_i} \p_x u_i^\nu(x,t) dx =0 $ and \eqref{entropy} that for any $ 0< \nu \leq 1 $ and $ t>0, $
	$$ \int_{x \in (0,p_i), \p_x u_i^\nu <0} |\p_x u_i^\nu(x,t)| dx = \int_{x \in (0,p_i), \p_x u_i^\nu >0} \p_x u_i^\nu(x,t) dx \leq \frac{p_i E}{t}, $$
	which yields that for any $ 0 <\nu \leq 1 $ and $ t>0, $
	$$ \sup_x |u_i^\nu - \ovlul| \leq \int_{0}^{p_i} |\p_x u_i^\nu(x,t)| dx \leq \frac{2p_i E}{t}. $$
	Since $ u_l^\nu $ and $ u_r^\nu $ converges almost everywhere to the periodic entropy solutions $ u^0_l $ and $ u^0_r, $  respectively, \eqref{uni-bdd} also holds true for $ \nu=0. $	
\end{proof}

\vspace{0.3cm}

\begin{proof}[Proof of \cref{Thm-vanishing}]	
	It follows from Taylor's expansion and zero average of $ u_l^\nu-\ovlul $ that 
	\begin{align*}
	\int_{0}^{p_l} [f(u_l^\nu)-f(\ovlul)]\ dx = & \int_{0}^{p_l} \left(f'(\ovlul) (u_l^\nu- \ovlul) + \frac{1}{2} f''(\cdot) (u_l^\nu- \ovlul)^2\right) dx \\
	= & \frac{1}{2} \int_{0}^{p} f''(\cdot) (u_l^\nu- \ovlul)^2 dx.
	\end{align*}
	This, together with Lemma \ref{Lem-Xin} and the strict convexity of $ f, $ implies that
	\begin{equation}\label{uni1}
	0< \int_{0}^{p_l} [f(u_l^\nu)-f(\ovlul)]\ dx \leq \frac{C}{(1+t)^2},
	\end{equation} 
	where $ C $ is independent of $ \nu $ or $ t. $
	Similarly, one has that
	\begin{equation}\label{uni2}
	0< \int_{0}^{p_r} [f(u_r^\nu)-f(\ovlur)]\ dx \leq \frac{C}{(1+t)^2}.
	\end{equation}
	Hence, applying the dominated convergence theorem in \cref{def-X-2} yields that, as $ \nu \rightarrow 0+, $
	\begin{equation}\label{limit}	
	\begin{aligned}
	X_{\infty,2}^\nu \rightarrow & \int_{0}^{+\infty} \frac{1}{p_l} \int_{0}^{p_l} [f(u_l^0)-f(\ovlul)] dxdt -\frac{1}{p_l}\int_{0}^{p_l}\int_{0}^{x} w_{0l}(y) dydx \\
	& - \int_{0}^{+\infty} \frac{1}{p_l} \int_{0}^{p_r} [f(u_r^0) -f(\ovlur)]\ dxdt +\frac{1}{p_r}\int_{0}^{p_r}\int_{0}^{x} w_{0r}(y) dydx.
	\end{aligned}	
	\end{equation}
	Now we prove that this limit is equal to 
	\begin{equation*}
	X_{\infty,2}^0 = - \min_{x\in\R} \int_{0}^{x} w_{0l}(y)dy + \min_{x\in\R} \int_{0}^{x} w_{0r}(y)dy.
	\end{equation*}
	
	\vspace{0.2cm}
	
	In fact, for $ i = l $ or $ r, $ since the anti-derivative variable $ \int_{0}^{x} w_{0i}(y)dy $ is continuous and periodic with the period $ p_i, $ one can choose a constant $ x_i \in [0,p_i) $ such that
	\begin{equation}\label{point-min}
	\int_{0}^{x_i} w_{0i}(y)dy = \min_{x\in [0,p_i]} \int_{0}^{x} w_{0i}(y)dy,
	\end{equation}
	which is equivalent to
	\begin{equation*}
	\int_{x_i}^{x} w_{0i}(y)dy \geq 0, \qquad x\in \R.
	\end{equation*}
	Then one can finish the proof of \cref{vanishing limit} if it holds that
	\begin{equation}\label{equality-inviscid}
	\int_{0}^{+\infty}\int_{0}^{p_i} \left( f(u^0_i)-f(\ovl{u}_i) \right) dxdt = \int_{0}^{p_i}\int_{x_i}^{x} w_{0i}(y) dydx.
	\end{equation}
	To prove \cref{equality-inviscid}, it follows from \cite[Theorem 14.1.1]{Dafe2016} that the periodic entropy solution $ u^0_i $ takes the constant value $ \ovl{u}_i $ along the straight line $ x=x_i+f'(\ovl{u}_i) t. $ Then for any given $ y \in (x_i, x_i+p_i) $ and $ t>0, $ denote the domain:
	\begin{equation*}
	\Omega_{(y,t)} := \left\lbrace  (x,\tau): x_i+f'(\ovl{u}_i) \tau < x < y+f'(\ovl{u}_i) \tau,~ 0<\tau<t \right\rbrace.
	\end{equation*}
	Integrating the equation $ \p_t u^0_i + \p_x f(u^0_i) = 0 $ over $ \Omega_{(y,t)}, $ one can obtain that
	\begin{align}
	0 = & -\int_{x_i}^{y} \left( \ovl{u}_i + w_{0i}(x)\right) dx + \int_{0}^{t} \left(  f(u^0_i)-f'(\ovl{u}_i)u^0_i \right) ( y+f'(\ovl{u}_i)\tau,\tau) d\tau \notag \\
	& - \int_{0}^{t} \left(  f(\ovl{u}_i)-f'(\ovl{u}_i)\ovl{u}_i \right) d\tau + \int_{x_i+f'(\ovl{u}_i)t}^{y+f'(\ovl{u}_i)t} u^0_i(x,t) dx \notag\\
	= & -\int_{x_i}^{y} w_{0i}(x) dx + \int_{0}^{t} \left(  f(u^0_i)-f(\ovl{u}_i)  \right) (y+f'(\ovl{u}_i)\tau,\tau)d\tau \label{int-part1} \\
	& - f'(\ovl{u}_i) \int_{0}^{t} (u^0_i-\ovl{u}_i)(y+f'(\ovl{u}_i) \tau, \tau) d\tau + \int_{x_i+f'(\ovl{u}_i)t}^{y+f'(\ovl{u}_i)t} \left( u^0_i(x,t)- \ovl{u}_i\right)  dx. \notag
	\end{align}	
	Since for any $ y \in \R, t\geq 0, $ $ \int_{y}^{y+p_i} (u^0_i-\ovl{u}_i)(x,t)dx = 0, $ thus one can integrate \eqref{int-part1} with respect to $ y $ over $ (x_i,x_i+p_i) $ to get that
	\begin{align*}
	0=& - \int_{x_i}^{x_i+p_i} \int_{x_i}^{y} w_{0i}(x) dxdy + \int_{x_i}^{x_i+p_i} \int_{0}^{t} \left(  f(u^0_i)-f(\ovl{u}_i) \right) (y+f'(\ovl{u}_i)\tau,\tau) d\tau dy \\
	& + \int_{x_i}^{x_i+p_i} \int_{x_i+f'(\ovl{u}_i)t}^{y+f'(\ovl{u}_i)t} \left( u^0_i(x,t)- \ovl{u}_i\right) dxdy
	\end{align*}	
	Since $ \int_{x_i}^{y} w_{0i}(x) dx $ is periodic with respect to $ y, $ and $$ \int_{x_i}^{x_i+p_i} \int_{x_i +f'(\ovl{u}_i)t}^{y+f'(\ovl{u}_i)t} \left( u^0_i(x,t)- \ovl{u}_i \right) dxdy = O\left( \frac{1}{1+t}\right), $$ then it holds that
	\begin{align}
	\int_{0}^{t} \int_{0}^{p_i} \left(  f(u^0_i(x, \tau))-f(\ovl{u}_i)\right)  dx d\tau = \int_{0}^{p_i} \int_{x_i}^{y} w_{0i}(x) dxdy + O\left( \frac{1}{1+t}\right). \label{equ4}
	\end{align}
	Similar to the proof of \eqref{uni1}, one can show that $$ \int_{0}^{p_i} \left(  f(u^0_i(x, \tau))-f(\ovl{u}_i)\right)  dx = O\left(\frac{1}{(1+\tau)^2}\right). $$ 
	Then \cref{equality-inviscid} follows by letting $ t \rightarrow +\infty $ in \eqref{equ4} and using the dominated convergence theorem.
	
	\vspace{0.2cm}
	
	It remains to prove \cref{rate-2} to finish the proof of \cref{Thm-vanishing}.
	If both the periodic perturbations $ w_{0l} $ and $ w_{0r} $ have bounded total variations on the respective periodic domains: 
	$$ \text{TV}_{[0,p_l]} w_{0l} < +\infty, \quad \text{TV}_{[0,p_r]} w_{0r} < +\infty, $$ then it can be derived from Kruzhkov's theory (see \cite{Kuznetsov1976,Kreiss1988}) that, for $ i=l $ or $ r, $ the viscous solution $ u_i^\nu $ to \cref{CL} tends to the inviscid entropy solution $ u_i^0 $ in the $ L^1 $ norm at the following rate:
	\begin{equation}\label{rate}
	\int_{0}^{p_i} \left| u_i^\nu(x,t) - u_i^0(x,t)\right|  dx \leq C (\nu t)^{1/2} ~\text{TV}_{[0,p_i]} w_{0i}, \quad  t>0, 0 \leq \nu \leq 1,
	\end{equation}
	where $ C>0 $ is independent of $ \nu $ or $ t. $
	
	Then for any given $ T >0, $ and $ i =l $ or $ r, $ one has that
	\begin{align*}
	& \int_{0}^{+\infty}\int_{0}^{p_i} \left( f(u_i^\nu)-f(\ovl{u}_i) \right) dxd\tau - \int_{0}^{+\infty}\int_{0}^{p_i} \left( f(u_i^0)-f(\ovl{u}_i) \right) dxd\tau \\
	= & ~\int_{0}^{T}\int_{0}^{p_i} \left( f(u_i^\nu)-f(u_i^0) \right) dxd\tau
	+ \int_{T}^{+\infty}\int_{0}^{p_i} \left( f(u_i^\nu) -f(\ovl{u}_i) \right)  dxd\tau \\
	& - \int_{T}^{+\infty}\int_{0}^{p_i} \left( f(u_i^0) -f(\ovl{u}_i) \right) dxd\tau.
	\end{align*}
	Then it follows from \eqref{uni1}, \eqref{uni2} and \eqref{rate} that for $ i =l $ or $ r, $
	\begin{equation*}
	\begin{aligned}	
	& \Big|\int_{0}^{+\infty}\int_{0}^{p_i} \left( f(u_i^\nu)-f(\ovl{u}_i) \right) dxd\tau - \int_{0}^{+\infty}\int_{0}^{p_i} \left( f(u_i^0)-f(\ovl{u}_i) \right) dxd\tau \Big| \\
	\leq &~ C~ \Big\{ \nu^{1/2} \int_{0}^{T} t^{1/2} dt + \int_{T}^{+\infty} \frac{1}{(1+t)^2} dt \Big\} \leq C ~\left( \nu^{1/2} T^{3/2} + T^{-1} \right).
	\end{aligned}	
	\end{equation*}
	Letting $ T = \nu^{-1/5} $ in the above inequality yields that for $ i =l $ or $ r, $
	\begin{equation*}
	\left| \int_{0}^{+\infty}\int_{0}^{p_i} \left( f(u_i^\nu)-f(\ovl{u}_i)\right) dxd\tau - \int_{0}^{+\infty}\int_{0}^{p_i} \left( f(u_i^0)-f(\ovl{u}_i)\right) dxd\tau \right|
	\leq  C \nu^{1/5}.	
	\end{equation*}
	It follows from this and the formulas of $ X_{\infty,2}^\nu $ and $ X_{\infty,2}^0 $ that
	\begin{equation*}
	| X_{\infty,2}^\nu - X_{\infty,2}^0 | \leq C \nu^{1/5},
	\end{equation*}
	where $ C>0 $ is independent of $ \nu. $ 

\end{proof}
\vspace{0.5cm}

\section{Stability of rarefaction waves}\label{Sec-rare}

The proof of Theorem \ref{Thm-rare} can follow from the idea in \cite{Oleinik1960}. To make this paper complete, we still give the details here.
The proof consists of two steps. 
The first step is to prove a time-independent estimate of the solution $ u, $ just as Proposition \ref{prop-unidecay} for the shock profile.
Step 2 is to construct an auxiliary function and use the maximal principle to complete the proof.

\begin{Prop}\label{prop-rare-unidecay}
	For any $\e>0$, there exist $N_\e>0$ and $T_\e>0 $ such that 
	\begin{align}
	&  |u^\nu(x,t)- \ovlul| \leq \e, \qquad t>T_\e, ~ x-f'(\ovlul)t<-N_\e, \label{rare-unidecay-l} \\
	&  |u^\nu(x,t)- \ovlur| \leq \e, \qquad t>T_\e,~ x-f'(\ovlur)t>N_\e.\label{rare-unidecay-r}
	\end{align}
\end{Prop}
\begin{proof}
	We prove only \cref{rare-unidecay-l}, since the proof of \cref{rare-unidecay-r} is similar. 
	
	For any $ \e>0, $ there exists $T>0$ such that $|u_l^\nu-\ovlul|<\e\nu/2, $ for any $x\in \R, t\geq T$. Since $ f $ is strictly convex, there exists $B>0$ such that, for any $0<\sigma<2 \|u_0\|_{L^{\infty}(\R)}$, 
	$$ f'(\ovlul +\sigma)-f'(\ovlul)>B \sigma. $$	
	Without loss of generality, one can assume that $ \e>0 $ is small enough such that $ B\e<2 \beta, $ where $ \beta>0 $ is the constant in \cref{Lem-decay-ulr}.
	
	\vspace{0.2cm}
	
	\textbf{(1) }
	By \cref{Lem-decay-ulr} and that $ \|u^\nu\|_{L^\infty(dx,dt)} \leq \|u_0\|_{L^\infty}, $ one can choose $M>0$ large enough such that 
	$$
	Me^{\frac{B \e}{2} \left( x-f'(\ovlul) T\right)}+\frac{\e\nu}{2}+\ovlul > u^\nu(x,T), \quad x \in \R.
	$$
	Define $$ V(x,t):=Me^{\frac{B \e}{2} (x-f'(\ovlul) t)} + \frac{\e\nu}{2} +\ovlul-u^\nu(x,t). $$
	It can be checked easily that $ V(x,T)>0, $ and
	\begin{equation}\label{aux-v}
	\nu \p_x^2 V-\p_t V -f'(u^\nu) \p_x V 
	=M \frac{B \e}{2} e^{\frac{B \e}{2} \left( x-f'(\ovlul) t\right)} \left( \frac{B \e\nu}{2}+f'(\ovlul)-f'(u^\nu)\right).
	\end{equation}
	If
	$$ \min_{x\in\R,t\geq T} V(x,t) = V(x_0,t_0) <0, $$ by the definition of $V$, one has that $\ovlul-u^\nu(x_0,t_0)<-\frac{\e\nu}{2}, $ and then
	$$
	f'(\ovlul)-f'(u^\nu(x_0,t_0))<f'(\ovlul)-f'(\ovlul+\frac{\e\nu}{2})\leq -\frac{B\e\nu}{2}.
	$$
	Thus \eqref{aux-v} yields that at the point $ (x_0,t_0), $
	$$ \nu \p_x^2 V -\p_t V -f'(u^\nu) \p_x V < 0.
	$$
	Therefore, it follows from the maximum principle (\cite[Lemma~1]{Oleinik1960}) that $ V(x,t) \geq 0 $ for any $ x\in\R, t\geq T. $ Choosing $ N := \frac{2}{B\e}\ln \frac{\e}{2M}, $ one has that for any $ t>T $ and $ x-f'(\ovlul)t<-N, $ it holds that
	\begin{equation}\label{ineq}
	u^\nu(x,t) \leq \ovlul+\frac{\e\nu}{2}+Me^{\frac{B \e}{2} (x-f'(\ovlul) t)} <\ovlul+\e.
	\end{equation}
	
	\textbf{(2) } On the other hand, for the initial data \eqref{ic-2-per} and any $ \e>0, $ one can let $ \tilde{u}(x,t) $ be the unique solution to \eqref{CL} with the $ L^\infty $ initial data
	\begin{equation*}
	\tilde{u}(x,0) =
	\begin{cases}
	u_0(x), \quad & x<0, \\
	\ovlul - \e/2 + \left( u_0(x) - \ovlur \right), \quad & x>0,
	\end{cases}
	\end{equation*}
	which satisfies
	\begin{equation*}
	\begin{cases}
	|\tilde{u}(x,0) - \ovlul - w_{0l}(x)| \leq C_0 e^{\beta_0 x}, & \\
	|\tilde{u}(x,0) - (\ovlul-\e/2) - w_{0r}(x)| \leq C_0 e^{-\beta_0 x}, &
	\end{cases} \quad x\in\R.
	\end{equation*}
	\cref{Thm-2-periodic} implies that as $ t\rightarrow +\infty, $ $ \tilde{u} $ tends to a viscous shock profile $ \tilde{\phi} $ connecting $ \ovlul $ as $ x\rightarrow -\infty $ and $ \ovlul - \e/2 $ as $ x\rightarrow +\infty. $ Thus, there exists $ T>0 $ such that for any $ x\in\R, t>T, $
	\begin{equation}\label{ineq12}
		\tilde{u}(x,t) \geq \tilde{\phi}(x-st) -\e/2 \geq \ovlul - \e.
	\end{equation}
	By $ \tilde{u}(x,0) \leq u_0(x) $ and the comparison principle, one has 
	\begin{equation}\label{ineq11}
		u^\nu(x,t) \geq \tilde{u}(x,t) \geq \ovlul - \e,  \qquad x\in\R, t\geq 0.
	\end{equation}
 	Then \cref{rare-unidecay-l} follows from \cref{ineq} and
	\cref{ineq11}.
\end{proof}

\vspace{0.3cm}

\begin{proof}[\textbf{Proof of Theorem \ref{Thm-rare}}]
It is equivalent to prove that for any $ \e>0, $ there exists $ T>0 $ such that 
\begin{equation}\label{rare-1}
\sup_{x \in \R} | u^\nu(x, t) - u^R(x, t) | < \e,  \qquad t>T.
\end{equation}

For the constants $ N_\e $ and $ T_\e $ in Proposition \ref{prop-rare-unidecay}, one can define two constants
$$
x_0:=N_\e \frac{f'(\ovlul)+f'(\ovlur)}{f'(\ovlul)-f'(\ovlur)} ~\text{ and }~ t_0:=\frac{-2N_\e}{f'(\ovlul)-f'(\ovlur)}>0,
$$
and the region $$\Omega_\e:=\left\lbrace (x,t): f'(\ovlul)t-N_\e <x<f'(\ovlur)t+N_\e,~t > T_\e \right\rbrace.$$
Then the shifted rarefaction wave $\tilde{u}^R(x,t) := u^R(x-x_0, t+t_0) $ satisfies that
\begin{equation*}
\tilde{u}^R(x,t) = 
\begin{cases}
\ovlul & \text{ for } x-f'(\ovlul)t \leq -N_\e, \\
(f')^{-1}(\frac{x-x_0}{t+t_0}) & \text{ for } x \in \Omega_\e, \\
\ovlur & \text{ for } x-f'(\ovlur)t \geq N_\e.
\end{cases}
\end{equation*}
Therefore, Proposition \ref{prop-rare-unidecay} implies that for any $x<f'(\ovlul)t-N_\e ~ \text{or} ~x>f'(\ovlur)t+N_\e, ~ t\geq T_\e$, one has that
\begin{equation}\label{ineq13}
|u^\nu(x,t)-\tilde{u}^R(x,t)|<\e.
\end{equation}
Define
\begin{equation}\label{def-y}
Z(x,t):=(t+t_0)^{\kappa} \left(  u^\nu(x,t)-\tilde{u}^R(x, t)\right),
\end{equation}
where $ 0< \kappa < 1 $ is a constant to be determined. Direct calculations show that
\begin{equation}\label{equ-y}
	\nu \p_x^2 Z-f'(u^\nu)\p_x Z -\p_t Z = \left( f''(v) \p_x \tilde{u}^R -\frac{\kappa}{t+t_0} \right) Z - \nu(t+t_0)^{\kappa} \p_x^2 \tilde{u}^R \quad \text{ in } \Omega_\e,
\end{equation}
where $v$ is the function satisfying
$$f'(u^\nu)-f'(\tilde{u}^R)=f''(v) (u^\nu-\tilde{u}^R).$$

For $ (x,t)\in \Omega_\e, $ it holds that $ f''(v)\p_x \tilde{u}^R = f''(v)\frac{1}{f''(\tilde{u}^R)} \frac{1}{t+t_0}>\frac{2\omega}{t+t_0}$ for some $\omega>0.$ Then $0<\kappa<1$ can be chosen small enough such that 
\begin{equation}\label{ineq14}
f''(v) \p_x \tilde{u}^R -\frac{\kappa}{t+t_0}>\frac{\omega}{t+t_0}>0.
\end{equation}

Note that for any $x\in \R, $ 
$$ | Z(x,T_\e) |\leq M_4 (T_\e+t_0)^{\kappa}, $$ 
provided that $ M_4 > \|u_0\|_{L^\infty} + \|u^R\|_{L^\infty} $ is large enough. 
Due to \eqref{ineq13}, one has that for any $t\geq T_\e,$
$$|Z(x,t)| \leq \e (t+t_0)^{\kappa},$$ 
if $x= f'(\ovlul)t-N_\e $ or $f'(\ovlur)t+N_\e. $
Now, assume that the maximum value
$ \max_{\Omega_\e} Z $ is achieved at $ (x_1,t_1) \in \Omega_\e $ and satisfies that
\begin{equation}\label{ineq15}
Z(x_1,t_1) \geq \max \Big\{ M_4(T_\e+t_0)^{\kappa}, \e(t_1+t_0)^{\kappa}, \frac{M_4}{\omega} (T_\e+t_0)^{\kappa-1} \Big\}>0.
\end{equation}
Then by \eqref{equ-y}--\eqref{ineq15}, one has that at $ (x_1,t_1), $
\begin{align*}
\nu \p_x^2 Z - f'(u^\nu)\p_x Z  -\p_t Z \geq \frac{\omega}{t_1+t_0} \frac{M_4}{\omega} (T_\e+t_0)^{\kappa-1} - \nu (t_1+t_0)^{\kappa} \frac{M_5}{(t_1+t_0)^{2}}>0,
\end{align*}
where $ M_5=\max \left|\left( (f')^{-1}\right)''\right| $ and $ M_4 > M_5 $ is large enough.
Therefore, the maximal principle implies that for any $ (x,t)\in \Omega_\e, $
$$
Z(x,t)\leq \max \Big\{ M_4(T_\e+t_0)^{\kappa}, \e (t+t_0)^{\kappa}, \frac{M_4}{\omega} (T_\e+t_0)^{\kappa-1} \Big\}.
$$
Similarly, one can verify that for any $ (x,t)\in \Omega_\e, $
$$ Z(x,t) \geq \min \Big\{ -M_4(T_\e+t_0)^{\kappa}, -\e (t+t_0)^{\kappa}, -\frac{M_4}{\omega} (T_\e+t_0)^{\kappa-1} \Big\}.
$$
As a result, by choosing a large $ T>T_\e, $ one can get that
\begin{equation}\label{rare-2}
|u^\nu(x,t)-u^R(x-x_0,t+t_0)|\leq \max \Big\{ \frac{M_4(T_\e +t_0)^{\kappa}}{(t+t_0)^{\kappa}}, \e, \frac{M_4}{(T_\e +t_0)^{1-\kappa} (t+t_0)^{\kappa}}  \Big\}\leq \e
\end{equation}
for any $ f'(\ovlul)t-N_\e < x <f'(\ovlur)t + N_\e $ and $t > T. $
Moreover, since $ u^R $ is Lipschitz continuous and $\frac{x}{t}-\frac{x-x_0}{t+t_0}\rightarrow 0$ as $t\rightarrow \infty,$ \eqref{rare-1} follows easily from combining Proposition \ref{prop-rare-unidecay} and \eqref{rare-2}, and thus \cref{Thm-rare} is proved.
\end{proof}

\vspace{2cm}


\appendix

\section{Proof of \cref{Lem-periodic}}

\begin{proof}
	For convenience, we let $ \nu=1 $ and omit the symbol $ \nu. $
	By multiplying $ u-\ovl{u} $ on each side of \eqref{CL} and integrating on $ [0,p], $ it holds that
	\begin{equation}\label{app1}
	\frac{d}{dt} \int_{0}^{p} (u-\ovl{u})^2(x,t) dx + 2 \int_{0}^{p} (\p_x u)^2(x,t) dx =0, \qquad t>0.
	\end{equation}
	By the Poincar\'{e} inequality on $[0,p],$ there exists a constant $ \alpha>0, $ which depends only on $ p, $ such that $$ \int_{0}^{p} (\p_x u)^2(x,t) dx \geq \frac{\alpha}{2} \int_{0}^{p} (u-\ovl{u})^2(x,t) dx. $$
	Then by \eqref{app1}, one has 
	\begin{equation}\label{app2}
	\int_{0}^{p} (u-\ovl{u})^2(x,t) dx \leq C_0 e^{-\alpha t}  \quad \forall t\geq 0,
	\end{equation}
	where $ C_0 = \int_{0}^{p} (u_0-\ovl{u})^2 dx $ depends on $ p $ and $ \|u_0\|_{L^\infty}. $
	
	\vspace{0.2cm}
	
	$ Claim $ 1. For any integer $ k\geq 1, $ $$ \int_{0}^{p} (\p_x^k u)^2(x,t) dx \leq C  \quad \forall  t\geq 1, $$ where $ C $ depends on $ k, p, f$ and $ \|u_0\|_{L^\infty}. $
	
	\vspace{0.2cm}
		
	In fact, for each $ k\geq 0, $ we let $ t_k:=\frac{1}{2}-\frac{1}{k+3}, $ and define smooth functions $ \zeta_k(t): [0, +\infty) \rightarrow [0,1], $ which are non-decreasing and satisfy that for all $ t\geq 0, $
	\begin{equation*}
	\zeta_k(t)=\begin{cases}
	0, & \quad t \in [0,t_k] \\
	1, & \quad t \in [1, +\infty)
	\end{cases} \quad \text{ and } \quad \zeta_k'(t)+\zeta_k(t) \leq B_k \zeta_{k-1}(t),
	\end{equation*}
	where $ B_k>0 $ is a constant depending on $ k; $ see \cref{Fig-cutoff}.
	\begin{figure}[htbp!]
		\scalebox{0.5}{\includegraphics{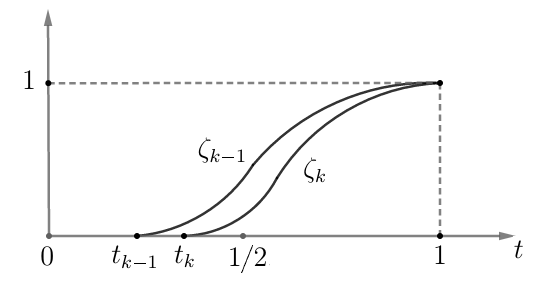}}
		\caption{}
		\label{Fig-cutoff}
	\end{figure}
	
	Then we prove $ Claim $ 1 by the induction method. We will prove that for each $ k\geq 1, $ there exists a constant $ C >0, $ depending on $ k, p, f$ and $ \|u_0\|_{L^\infty}, $  such that
	\begin{equation}\label{ineq_k}
		\int_{0}^{p} \left(\p_x^{k-1}(u-\ovl{u}) \right)^2 dx + \int_{0}^{t} \zeta_{k-1}(\tau) \int_{0}^{p} (\p_x^{k} u)^2 dx d\tau \leq C \quad \forall t>t_k.
	\end{equation} 
	
	In fact, when $ k=1, $ \eqref{ineq_k} follows from \eqref{app1}. Then we assume that \eqref{ineq_k} holds for $ k=1,2,\cdots,m $ with $ m\geq 1, $ and then we will prove that \eqref{ineq_k} also holds for $ k=m+1. $	
	By taking the derivative $ \p_x^m $ in \eqref{CL} and multiplying $ \zeta_m \p_x^m u $ on each side, one can obtain
	\begin{align*}
		& \p_t (\zeta_m(\p_x^m u)^2)-\zeta_m'(\p_x^m u)^2 +\p_x \left( \zeta_m \p_x^m u~ \p_x^m f(u) \right) - \zeta_m \p_x^{m+1} u~ \p_x^m f(u) \\
		 = & \p_x ( \zeta_m \p_x^m u ~\p_x^{m+1} u)-\zeta_m(\p_x^{m+1} u)^2.
	\end{align*}
	This, with the Cauchy-Schwartz inequality, yields that for all $ t>0, $
	\begin{equation}\label{app2_1}
	\begin{aligned}
	& \frac{d}{dt} \int_{0}^{p} \zeta_m(t)(\p_x^m u)^2 dx + \zeta_m(t)\int_{0}^{p} (\p_x^{m+1} u)^2 dx  \\
	\leq~ & C(\zeta_m'+\zeta_m) \sum_{k=1}^{m} \int_{0}^{p} (\p_x^k u)^2 dx \\
	\leq~ & CB_m \zeta_{m-1}(t) \sum_{k=1}^{m} \int_{0}^{p} (\p_x^k u)^2 dx. 
	\end{aligned}	 
	\end{equation}
	Thus one can have that for all $ t>0, $
	\begin{equation}\label{app2_2} 
	\begin{aligned}
	& \int_{0}^{p} \zeta_m(t)(\p_x^m u)^2 dx + \int_0^t \zeta_m(\tau)\int_{0}^{p} (\p_x^{m+1} u)^2 dx d\tau  \\
	\leq~ & C_m \sum_{k=1}^{m} \int_0^t\zeta_{k-1}(\tau) \int_{0}^{p} (\p_x^k u)^2 dx d\tau,
	\end{aligned}	
	\end{equation}
	where $ C_m>0 $ depends on $ f, p$ and $ \|u_0\|_{L^\infty}. $ Then by \eqref{ineq_k} for $ k=1,2,\cdots,m, $ for any $ t>t_m, $ the right hand side of \eqref{app2_2} is bounded by a constant, so \eqref{ineq_k} holds true for $ k=m+1. $ Thus, by the induction method, \eqref{ineq_k} holds true for any $ k\geq 1 $ and any $ t\geq 1, $ which completes the proof of $ Claim $ 1.
	
	\vspace{0.2cm}
	
	Then it follows from Sobolev inequality, $ Claim $ 1, and the equation \eqref{CL} that for any integers $ k,l \geq 0, $
	\begin{equation*}
		\| \p_t^l \p_x^k (u-\ovl{u}) \|_{L^\infty(\R)} \leq C_{kl}  \quad \forall ~t\geq 1.
	\end{equation*}	
	And since for each $ k\geq 0, $ $ \zeta_k(t)=1 $ and $ \zeta_k'(t)=0 $ for all $ t \geq 1, $ \eqref{app1} and \eqref{app2_1} yield that
	\begin{equation}\label{app3}
		\frac{d}{dt} \int_{0}^{p} \left(\p_x^k (u-\ovl{u}) \right)^2 dx + \int_{0}^{p} (\p_x^{k+1} u)^2 dx \leq C \sum_{l=1}^{k} \int_{0}^{p} (\p_x^l u)^2 dx  \quad \forall t\geq 1,
	\end{equation}
	where $ C>0 $ depends on $ k, p, f$ and $ \|u_0\|_{L^\infty}. $
	
	\vspace{0.3cm}
	
	$ Claim $ 2. For each $ k\geq 0, $ there holds that
	\begin{equation}\label{ineq_k2}
		\int_{0}^{p} \left(\p_x^k (u-\ovl{u})\right)^2(x,t) dx \leq C e^{-\alpha t}  \quad \forall t\geq 1,
	\end{equation}
	where $ C>0 $ depends on $ k, p, f$ and $ \|u_0\|_{L^\infty}. $
	
	\vspace{0.2cm}
	
	To prove $ Claim $ 2, we also use the induction method. For $ k=0, $ \eqref{ineq_k2} follows from \eqref{app2}. Thus, one can assume that for $ k=0,1,\cdots,m-1 $ with $ m\geq 1, $ $ Claim $ 2 is true. 
	Then for $ k=m, $ by \eqref{app3} with $ k=m, $ one has that for all $ t\geq 1, $
	\begin{equation}\label{ineq_m}
	\begin{aligned}
		\frac{d}{dt} \int_{0}^{p} (\p_x^m u)^2 dx  & \leq C_m \sum_{k=1}^{m} \int_{0}^{p} (\p_x^k u)^2 dx \\
		& \leq C e^{-\alpha t}  + C_m \int_{0}^{p} (\p_x^m u)^2 dx,
	\end{aligned}
	\end{equation}
	where $ C, C_m >0 $ depend on $ m, p, f$ and $ \|u_0\|_{L^\infty}, $ and $ C_m $ can be large enough such that $ C_m > \alpha. $
	Letting $ k=m-1 $ in \eqref{app3}, one gets that for all $ t\geq 1, $
	\begin{equation}\label{ineq_m-1}
		\frac{d}{dt} \int_{0}^{p} \left(\p_x^{m-1} (u-\ovl{u})\right)^2 dx + \int_{0}^{p} (\p_x^m u)^2 dx \leq C e^{-\alpha t}.
	\end{equation}	
	Then by multiplying $ 2C_m $ on \eqref{ineq_m-1} and then adding it to \eqref{ineq_m}, one can obtain that for all $ t\geq 1, $
	\begin{equation}\label{app33*}
	\begin{aligned}
	\frac{d}{dt} \left[2C_m \int_{0}^{p} \left(\p_x^{m-1} (u-\ovl{u})\right)^2 dx + \int_{0}^{p} (\p_x^m u)^2 dx \right] + C_m \int_{0}^{p} (\p_x^m u)^2 dx \leq C e^{-\alpha t}.
	\end{aligned}	
	\end{equation}
	Denote
	\begin{equation*}
		E_m(t) := 2 C_m \int_{0}^{p} \left(\p_x^{m-1} (u-\ovl{u})\right)^2 dx + \int_{0}^{p} (\p_x^m u)^2 dx. 
	\end{equation*}
	Then \eqref{ineq_k2} with $ k=m-1 $ and \eqref{app33*} yield that for all $ t\geq 1, $
	\begin{equation}\label{app33}
		\begin{aligned}
		E_m'(t) + C_m E_m(t) \leq C_m'  e^{-\alpha t}.
		\end{aligned}
	\end{equation}
	Since $ C_m >\alpha, $ one can easily obtain that
	$ E_m(t) \leq Ce^{-\alpha t}, $ where $ C>0 $ depends on depends on $ m, p, f$ and $ \|u_0\|_{L^\infty}. $ The proof of $ Claim $ 2 is finished.
	
	\vspace{0.3cm}

	Then by Sobolev inequality and $ Claim $ 2, and combined with the equation \eqref{CL}, one can have that for any integers $ k,l \geq 0 $ and $ t\geq 1, $
	\begin{equation*}
		\| \p_t^l \p_x^k (u-\ovl{u}) \|_{L^\infty(\R)} \leq Ce^{-\alpha t},
	\end{equation*}
	which finishes the proof of \cref{Lem-periodic}.
\end{proof}
\vspace{0.5cm}

\section{Proof of \cref{Lem-decay}}
\begin{proof}
	For convenience, we assume that $ \nu=1 $ and omit the symbol $ \nu. $ And let $ K^t(x) := \frac{1}{\sqrt{2\pi t}}e^{-\frac{x^2}{4t}} $ denote the heat kernel.
	
	As in \cite{Freistuhler1998}, the solution $S_tu_0$ can be obtained by constructing the following approximating sequence
	\begin{align*}
	u^{(1)}&=K^t* u_0,\\
	u^{(n+1)}&=K^t* u_0-\int_{0}^{t} \partial_x K^{\tau}(\cdot) * f\big(u^{(n)}(\cdot, t-\tau)\big) d\tau, \qquad n=1,2,3,\ldots
	\end{align*}
	where ``$*$'' represents the convolution operation with respect to the space variable.
	
	
	Suppose that $\{\tilde{u}^{(n)}\}_{n=1}^{\infty}$ is the approximating sequence induced by $\tilde{u}_0$, constructed in the same way as $ u^{(n)}. $ Therefore, one has that
	\begin{align*}
	|u^{(1)}-\tilde{u}^{(1)}|&=|K^t * (u_0-v_0)|\leq \int_{\R} \frac{1}{\sqrt{2\pi t}}e^{-\frac{y^2}{4t}} C e^{\delta(x-y)} dy\\
	&\leq \int_{\R} \frac{1}{\sqrt{2\pi t}}e^{-\frac{1}{4t}(y+2\delta t)^2}C e^{\delta x+\delta^2 t} dy=C e^{\delta x+\delta^2 t},\\
	|u^{(2)}-\tilde{u}^{(2)}|&\leq |K^t * (u_0-\tilde{u}_0)| + \Big| \int_{0}^{t} \partial_x K^\tau(\cdot)  *\Big[f\big(u^{(1)}(\cdot,t-\tau)\big)\\
	&\qquad\qquad\qquad\qquad\qquad\qquad\qquad\quad-f\big(\tilde{u}^{(1)}(\cdot,t-\tau)\big)\Big] d\tau \Big|  \\
	&\leq C e^{\delta x+\delta^2 t} +\int_{0}^{t}\int_{\R} \frac{1}{\sqrt{2\pi \tau}}\frac{|y|}{2\tau}e^{-\frac{y^2}{4\tau}} C_0 C e^{\delta(x-y)+\delta^2(t-\tau)} dy d\tau \\
	&\leq C e^{\delta x+\delta^2 t} + C_0 C e^{\delta x+\delta^2t} \int_{0}^{t}\int_{\R}  \frac{1}{\sqrt{2\pi \tau}}\frac{|y|}{2\tau} e^{-\frac{(y+2\delta \tau)^2}{4\tau}}  dy d\tau \\
	&\leq C e^{\delta x+\delta^2 t}+ C_0 C e^{\delta x+\delta^2t} \int_{0}^{t}\int_{\R}  \frac{1}{\sqrt{2\pi \tau}}\big(\frac{|y|}{2\tau}+|\delta| \big)e^{-\frac{y^2}{4\tau}} dy d\tau \\
	&= Ce^{\delta x+\delta^2 t}+ C_0 C e^{\delta x+\delta^2t} \int_{0}^{t} \big(\sqrt{\frac{2}{\pi \tau}}+|\delta| \big)  d\tau\\
	&\leq C e^{\delta x+\delta^2 t}\left(1+\frac{2\sqrt{2}}{\sqrt{\pi}}C_0\sqrt{t}+C_0|\delta|t\right),\\        
	\ldots,  &        
	\end{align*}
	where $ C_0:=\max\{|f'(u)|:~|u|\leq \|u_0, \tilde{u}_0\|_{L^{\infty}(\R)}\}. $
	By induction, one has that for $ t>0 $ small, 
	\begin{align*}
	|u^{(n)}-\tilde{u}^{(n)}|&\leq C e^{\delta x+\delta^2 t}\Big[1+\big(\frac{2\sqrt{2}}{\sqrt{\pi}}C_0\sqrt{t}+C_0|\delta|t\big)+\cdots+\big(\frac{2\sqrt{2}}{\sqrt{\pi}}C_0\sqrt{t}+C_0|\delta|t\big)^{n-1}\Big]\\
	&\leq C e^{\delta x+\delta^2 t}\frac{1}{1-\big(\frac{2\sqrt{2}}{\sqrt{\pi}}C_0\sqrt{t}+C_0|\delta|t\big)}.
	\end{align*}
	Therefore, by letting $n\rightarrow +\infty$, there exists a small enough $t_0=t_0(|\delta|, C_0)$ with 
	$$ e^{\delta^2 t_0}\frac{1}{1-\big(\frac{2\sqrt{2}}{\sqrt{\pi}}C_0\sqrt{t_0}+C_0|\delta|t_0\big)}<2, $$ 
	such that $|S_t u_0-S_t\tilde{u}_0|\leq 2C e^{\delta x}$ holds for any $x\in \R, t\in (0,t_0]$.
	At time $t=k t_0$, $k=1,2,3,\ldots$, one can take $S_{kt_0}u_0, S_{kt_0}\tilde{u}_0$ instead of $u_0, \tilde{u}_0$ as the initial data and then repeat the same estimates as above in the interval $[kt_0,(k+1)t_0]$. It  concludes that for any $x\in \R, t>0$,  $|S_tu_0-S_t\tilde{u}_0|\leq 2^{ \frac{t}{t_0}+1}C e^{\delta x}$.
\end{proof}

\vspace{0.5cm}

\section{Proof of \cref{Lem-g}}
\begin{proof}
	For convenience, we let $ \nu=1 $ and omit the symbol $ \nu. $
	Integrating the equation \eqref{ode1} shows that the shock profile $ \phi $ satisfies
	$$ \phi' = f(\phi)-f(\ovlul)-s(\phi-\ovlul), $$
	which implies that
	$$
	\phi'=(\phi-\ovlul)\left(\frac{f(\phi)-f(\ovlul)}{\phi-\ovlul}-s \right)=(\phi-\ovlur)\left(\frac{f(\phi)-f(\ovlur)}{\phi-\ovlur}-s \right).
	$$
	thus one has that the function $ g^\nu(x) $ defined in \eqref{def-g} satisfies the equation:
	\begin{equation}\label{def-gprime}
	\begin{aligned}
	g'&=\frac{1}{\ovlul-\ovlur}\Big[f\big((\ovlul-\ovlur)g +\ovlur \big)-f(\ovlur)\Big]-sg \\
	&=\frac{1}{\ovlul-\ovlur}\Big[f\big(\ovlul g +\ovlur(1-g) \big)-f(\ovlul)g -f(\ovlur)(1-g)\Big].
	\end{aligned}
	\end{equation}
	(i). Since $f$ is smooth, for any $x>y, 0\leq \rho \leq 1, z=\rho x+(1-\rho)y$, one has 
	\begin{align*}
	&f(z)-[\rho f(x)+(1-\rho)f(y)]\\
	=&\rho (z-x)\int_{0}^{1} f'\left(\tau z+(1-\tau)x \right) d\tau
	+(1-\rho)(z-y)\int_{0}^{1} f'\left(\tau z+(1-\tau)y \right) d\tau \\
	=&\rho(1-\rho)(x-y)\int_{0}^{1}\int_{0}^{1} (y-x)(1-\tau) \times \\ 
	& \qquad \qquad \qquad \qquad\quad f''\left(\tilde{\tau}[\tau z+(1-\tau)y] + (1-\tilde{\tau})[\tau z+(1-\tau)x]  \right)d\tilde{\tau} d\tau. 
	\end{align*}
	Therefore, 
	\begin{align*}
	\frac{1}{2}\min\limits_{u\in [y,x]}f''(u)\leq -\frac{f(z)-[\rho f(x)+(1-\rho)f(y)]}{\rho(1-\rho)(x-y)^2} \leq \frac{1}{2}\max\limits_{u\in [y,x]} f''(u).
	\end{align*}
	Then \eqref{g1} follows by substituting $x=\ovlul, y=\ovlur, \rho=g $ and $z=\ovlul g +\ovlur(1-g)$, and applying the definition \eqref{def-gprime}.
	\vspace{0.2cm}
	
	(ii). Integrating the equation \eqref{g1} yields 
	\begin{equation*}
	\begin{aligned}
	\beta_1x\leq \ln \frac{1-g(x)}{g(x)} -\beta_3\leq \beta_2 x, &\quad x>0,\\
	\beta_2 x\leq \ln \frac{1-g(x)}{g(x)} -\beta_3\leq \beta_1 x, &\quad x<0,
	\end{aligned}
	\end{equation*}
	where $\beta_3=\ln \frac{1-g(0)}{g(0)}$.
	And then 
	\begin{equation*}
	\begin{aligned}
	\frac{1}{1+e^{\beta_2x+\beta_3}}\leq g(x)\leq \frac{1}{1+e^{\beta_1x+\beta_3}},& \quad x>0,\\
	\frac{e^{\beta_2x+\beta_3}}{1+e^{\beta_2x+\beta_3}}\leq 1-g(x)\leq \frac{e^{\beta_1x+\beta_3}}{1+e^{\beta_1x+\beta_3}}, &\quad x<0.
	\end{aligned}
	\end{equation*}
	Therefore, \eqref{g2} follows, and $C$ depends on $\beta_1, \beta_2 $ and $ \beta_3$.
\end{proof}

\vspace{1.5cm}




\bibliographystyle{amsplain}
\bibliography{conservation_laws}

\providecommand{\bysame}{\leavevmode\hbox to3em{\hrulefill}\thinspace}
\providecommand{\MR}{\relax\ifhmode\unskip\space\fi MR }
\providecommand{\MRhref}[2]{%
  \href{http://www.ams.org/mathscinet-getitem?mr=#1}{#2}
}
\providecommand{\href}[2]{#2}
\begin{thebibliography}{10}

\bibitem{Dafermos1995}
C.~M. Dafermos, \emph{Large time behavior of periodic solutions of hyperbolic
  systems of conservation laws}, Journal of Differential Equations \textbf{121}
  (1995), no.~1, 183--202.

\bibitem{Dafe2016}
\bysame, \emph{Hyperbolic conservation laws in continuum physics}, fourth ed.,
  Grundlehren der Mathematischen Wissenschaften [Fundamental Principles of
  Mathematical Sciences], vol. 325, Springer-Verlag, Berlin, 2016. \MR{3468916}

\bibitem{Dalibard2010}
A.-L. Dalibard, \emph{Long time behavior of parabolic scalar conservation laws
  with space periodic flux}, Indiana Univ. Math. J. \textbf{59} (2010), no.~1,
  257--300.

\bibitem{Dalibard2017}
A.~L. Dalibard and M.~J. Kang, \emph{Existence and stability of planar shocks
  of viscous scalar conservation laws with space-periodic flux}, Journal des
  Mathematiques Pures et Appliquees \textbf{107} (2017), no.~3, 336--366.

\bibitem{Freistuhler1998}
H.~Freist\"{u}hler and D.~Serre, \emph{$l^1$ stability of shock waves in scalar
  viscous conservation laws}, Communications on Pure and Applied Mathematics
  \textbf{51} (1998), no.~3, 291--301.

\bibitem{Friedman1964}
A.~Friedman, \emph{{Partial Differential Equations of Parabolic Type}},
  Prentice-Hall, Inc., Englewood Cliffs, N.J., 1964.

\bibitem{Glimm1970}
J.~Glimm and P.~D. Lax, \emph{Decay of solutions of systems of nonlinear
  hyperbolic conservation laws}, Memoirs of the American Mathematical Society,
  No. 101, American Mathematical Society, Providence, R.I., 1970.

\bibitem{Goodman1986}
J.~Goodman, \emph{Nonlinear asymptotic stability of viscous shock profiles for
  conservation laws}, Archive for Rational Mechanics and Analysis \textbf{95}
  (1986), no.~4, 325--344.

\bibitem{Goodman1989}
\bysame, \emph{Stability of viscous scalar shock fronts in several dimensions},
  Trans. Amer. Math. Soc. \textbf{311} (1989), no.~2, 683--695.

\bibitem{Goodman1992}
J.~Goodman and Z.~Xin, \emph{Viscous limits for piecewise smooth solutions to
  systems of conservation laws}, Archive for Rational Mechanics and Analysis
  \textbf{121} (1992), no.~3, 235--265.

\bibitem{Harabetian1988}
E.~Harabetian, \emph{{Rarefactions and large time behavior for parabolic
  equations and monotone schemes}}, Communications in Mathematical Physics
  \textbf{114} (1988), no.~4, 527--536.

\bibitem{Hoff2000}
D.~Hoff and K.~Zumbrun, \emph{Asymptotic behavior of multidimensional scalar
  viscous shock fronts}, Indiana University Mathematics Journal \textbf{49}
  (2000), no.~2, 427--474.

\bibitem{Hoff2002}
\bysame, \emph{Pointwise green's function bounds for multidimensional scalar
  viscous shock fronts}, J. Differential Equations \textbf{183} (2002), no.~2,
  368--408.

\bibitem{Hopf1950}
E.~Hopf, \emph{The partial differential equation $ u_t + u u_x = \mu u_{xx} $},
  Communications on Pure and Applied Mathematics \textbf{3} (1950), no.~3,
  201--230.

\bibitem{Howard1999}
P.~Howard, \emph{Pointwise green's function approach to stability for scalar
  conservation laws}, Communications on Pure and Applied Mathematics
  \textbf{52} (1999), no.~10, 1295--1313.

\bibitem{Oleinik1960}
A.~M. Il'in and O.~A. Ole\v{\i}nik, \emph{Asymptotic behavior of solutions of
  the cauchy problem for some quasilinear equations for large values of time},
  Matematicheskii Sbornik \textbf{51(93)} (1960), no.~2, 191--216.

\bibitem{Kenig2006}
C.~E. Kenig and F.~Merle, \emph{Asymptotic stability and liouville theorem for
  scalar viscous conservation laws in cylinders}, Communications on Pure and
  Applied Mathematics \textbf{59} (2006), no.~6, 769--796.

\bibitem{Kopell1975}
N.~Kopell and L.~N. Howard, \emph{Bifurcations and trajectories joining
  critical points}, Advances in Mathematics \textbf{18} (1975), no.~3,
  306--358.

\bibitem{Kreiss1988}
H.~O. Kreiss, \emph{Fourier expansions of the solutions of the navier-stokes
  equations and their exponential decay rate}, Analyse math\'{e}matique et
  applications, Gauthier-Villars, Montrouge, 1988, pp.~245--262.

\bibitem{Kruzkov1970}
S~N Kruzkov, \emph{First order quasilinear equations with several independent
  variables}, Mat. Sb. (N.S.) \textbf{81 (123)} (1970), 228--255.

\bibitem{Kuznetsov1976}
N.~N. Kuznetsov, \emph{{Accuracy of some approximate methods for computing the
  weak solutions of a first-order quasi-linear equation}}, USSR Computational
  Mathematics and Mathematical Physics \textbf{16} (1976), no.~6, 105--119.

\bibitem{Lax1957}
P.~D. Lax, \emph{Hyperbolic systems of conservation laws ii}, Communications on
  Pure and Applied Mathematics \textbf{10} (1957), no.~4, 537--566.

\bibitem{Liu1985}
T.-P. Liu, \emph{Nonlinear stability of shock waves for viscous conservation
  laws}, Bullentin (New Series) of the American Mathematical Society
  \textbf{12} (1985), no.~2, 233--236.

\bibitem{Liu1997}
\bysame, \emph{Pointwise convergence to shock waves for viscous conservation
  laws}, Communications on Pure and Applied Mathematics \textbf{50} (1997),
  no.~11, 1113--1182.

\bibitem{Liu1988}
T.-P. Liu and Z.~Xin, \emph{{Nonlinear stability of rarefaction waves for
  compressible Navier- Stokes equations}}, Communications in Mathematical
  Physics \textbf{118} (1988), 451--465.

\bibitem{Matsumura1985}
A.~Matsumura and K.~Nishihara, \emph{On the stability of travelling wave
  solutions of a one-dimensional model system for compressible viscous gas},
  Japan J. Appl. Math. \textbf{2} (1985), no.~1, 17--25.

\bibitem{Nishihara1985}
K.~Nishihara, \emph{A note on the stability of travelling wave solutions of
  burgers' equation}, Japan Journal of Applied Mathematics \textbf{2} (1985),
  no.~1, 27--35.

\bibitem{Oleinik1957}
O.~A. Ole\v{\i}nik, \emph{Discontinuous solutions of non-linear differential
  equations}, Uspehi Mat. Nauk (N.S.) \textbf{12} (1957), no.~3(75), 3--73.

\bibitem{Szepessy1993}
A.~Szepessy and Z.~Xin, \emph{Nonlinear stability of viscous shock waves},
  Archive for Rational Mechanics and Analysis \textbf{122} (1993), no.~1,
  53--103.

\bibitem{Xin}
Z.~Xin, \emph{{Lectures on Hyperbolic Conservation Laws}}, IMS, CUHK.

\bibitem{Xin1988}
\bysame, \emph{{Asymptotic stability of rarefaction waves for $ 2 \times 2 $
  viscous hyperbolic conservation laws}}, Journal of Differential Equations
  \textbf{73} (1988), no.~1, 45--77.

\bibitem{Xin1990}
\bysame, \emph{Asymptotic stability of planar rarefaction waves for viscous
  conservation laws in several dimensions}, Trans. Amer. Math. Soc.
  \textbf{319} (1990), no.~2, 805--820.

\bibitem{Xin2019}
Z.~Xin, Q.~Yuan, and Y.~Yuan, \emph{Asymptotic stability of shock waves and
  rarefaction waves under periodic perturbations for 1-d convex scalar
  conservation laws}, SIAM Journal on Mathematical Analysis \textbf{51} (2019),
  no.~4, 2971--2994.

\bibitem{Yuan2019}
Q.~Yuan and Y.~Yuan, \emph{{On Riemann solutions under different initial
  periodic perturbations at two infinities for 1-d scalar convex conservation
  laws}}, arXiv: 1907.13043 (2019), 1--15.

\end{thebibliography}

\vspace{1cm}

\end{document}